\documentclass[11pt,A4paper]{article}
\usepackage{amsmath,amsthm,amssymb,amsfonts}
\usepackage[numbers,sort&compress]{natbib}
\usepackage{enumerate}
\usepackage{color,colordvi}
\usepackage{soul}
\usepackage{fullpage}

\usepackage{changes}

\usepackage[letterpaper,top=2cm,bottom=2cm,left=3cm,right=3cm,marginparwidth=2cm]{geometry}

\usepackage{amsmath,amsfonts,amsthm,mathtools,bm}
\usepackage{graphicx}
\usepackage[colorlinks=true, allcolors=blue]{hyperref}
\usepackage[capitalise]{cleveref}
\usepackage{standalone}
\usepackage{algorithm} 
\usepackage{indentfirst}
\usepackage{enumerate}
\usepackage{authblk}
\usepackage{diagbox}
\usepackage{changes}
\usepackage{algpseudocode}
\usepackage{bbm}
\usepackage{enumerate}

\algrenewcommand\algorithmicrequire{\textbf{Input:}}
\algrenewcommand\algorithmicensure{\textbf{Output:}}

\usepackage[mathscr]{euscript}
\usetikzlibrary{calc}

\usetikzlibrary{decorations.pathreplacing} 

\theoremstyle{plain}
\newtheorem{lemma}{Lemma}[section]
\newtheorem{theorem}[lemma]{Theorem}
\newtheorem{proposition}[lemma]{Proposition}
\newtheorem{corollary}[lemma]{Corollary}

\theoremstyle{definition}

\newtheorem{case}{Case}
\newtheorem{subcase}{Subcase}[case]

\newtheorem{claim}{Claim}

\newtheorem{conjecture}[lemma]{Conjecture}
\newtheorem{definition}[lemma]{Definition}

\newtheorem{example}[lemma]{Example}

\newtheorem{problem}[lemma]{Problem}

\newtheorem{remark}[lemma]{Remark}

\newtheorem*{claim*}{Claim}

\numberwithin{equation}{section}

\newcommand{\textif}{\text{if }}
\newcommand{\textotherwise}{\text{otherwise}}

\newcommand{\RR}{\mathbb{R}}

\newcommand{\RomanNum}[1]{\uppercase\expandafter{\romannumeral #1}}

\newcounter{mycase}
\newcommand{\mycase}[2]{%
  \refstepcounter{mycase}%
  \par\noindent\textbf{Case~\themycase.}\label{#1} #2%
}

\newcounter{mysubcase}[mycase]
\renewcommand{\themysubcase}{\themycase.\arabic{mysubcase}}
\newcommand{\mysubcase}[2]{%
  \refstepcounter{mysubcase}%
  \par\noindent\textbf{Subcase~\themysubcase.}\label{#1} #2%
}

\usepackage{tikz}

\tikzset{
    int/.style={circle,fill=black,draw=black,inner sep=1.7pt},
    leaf/.style={circle,fill=white,draw=black,inner sep=1.7pt},
    dots/.style={inner sep=0pt}
}

\usetikzlibrary{arrows.meta,fit,backgrounds,decorations.pathreplacing}

\DeclarePairedDelimiter{\set}{\lbrace}{\rbrace}

\title{Krahn--Szeg\H{o} type inequalities and nodal domain methods on graphs}

\author{Huiqiu Lin\footnote{email: huiqiulin@126.com}}
\author{Lianping Liu\footnote{email: y30231283@mail.ecust.edu.cn}}
\author{Xilong Yin\footnote{email: xilongyin@126.com}}
\author{Zhe You\footnote{email: y30231280@mail.ecust.edu.cn}}

\affil{School of Mathematics, East China University of Science and Technology, 130 Meilong Road, Shanghai 200237, China.}

\date{}

\begin{document}
\maketitle

\begin{abstract}

We study discrete analogues of classical spectral geometric inequalities and extremal eigenvalue problems on graphs.
The classical Krahn--Szeg\H{o} inequality states that, among bounded open subsets of $\mathbb{R}^n$ with fixed volume, the minimum of $\lambda_2(\Omega)$ is attained by the union of two congruent balls.
Firstly, we establish a Krahn--Szeg\H{o} type inequality for trees. 
For trees with a fixed number of interior vertices and boundary leaves, we completely characterize the extremal structures that minimize the second Dirichlet eigenvalue. 
Secondly, we develop a nodal domain method for  adjacency matrices. 
By proving an adjacency version of the nodal domain theorem for graphs, we obtain upper bounds for the second largest adjacency eigenvalue $\rho_2(G)$ of $G$ in given graph classes.
These bounds imply some previous results.
Finally, we settle the Aouchiche--Hansen conjecture (2010) on the second largest eigenvalue with  given number of edges and clique number. 
We prove that for connected graphs $G$ of odd order $n \geq 5$, $|\rho_2| \cdot \omega \leq m-2$, with equality if and only if $G$ consists of two complete graphs of orders $\frac{n+1}{2}$ and $\frac{n-1}{2}$ joined by an edge or a path. For even $n \geq 2$, the quantity $|\rho_2| \cdot \omega - m$ is maximized exactly when $G$ is obtained by adding one edge between the two copies of $K_{n/2}$ by an edge.

The core of the methods developed in this paper is to regard a connected graph as an internally disconnected graph with Dirichlet boundary condition. This perspective allows us to transfer nodal domain techniques from continuous spectral geometry to discrete settings and to obtain sharp extremal characterizations across diverse graph classes.


\end{abstract}



\tableofcontents
\section{Introduction}
The Dirichlet eigenvalue problem is one of the classical eigenvalue problems in spectral geometry.
Let $\Omega$ be a bounded open subset in $\mathbb{R}^d$ with a smooth boundary $\partial\Omega$.
The Dirichlet eigenvalue problem is 
$$\Delta f+\lambda f=0\text{ in $\Omega$}$$
with Dirichlet boundary condition 
$$f|_{\partial\Omega}=0.$$
Equivalently, this is the eigenvalue problem for the Dirichlet Laplacian $\Delta_\Omega$, which is defined by $\Delta_\Omega f \coloneqq (\Delta \widehat{f})|_\Omega$, where $\widehat{f}$ is an extension of $f$ to $\mathbb{R}^d$ by assigning 0 to the region outside $\Omega$.
The eigenvalues of the Dirichlet Laplacian can be ordered as 
\begin{align*}
    0 <\lambda_1(\Omega) \leq \lambda_2(\Omega)  \leq \cdots \leq\lambda_k(\Omega)\leq \cdots.
\end{align*}

In 1961, P\'olya~\cite{polya1961eigenvalues} proved that $$\lambda_k(\Omega)\geq C_d(\frac{k}{V(\Omega)})^\frac{2}{d}$$ 
for tiling domains in the plane, for all $k\geq 1$, where $V(\Omega)$ and $C_d=(2\pi)^2V_d^{-\frac{2}{d}}$ respectively denote the volume of $\Omega$ and the Weyl constant. 
$V_d$ is the volume of the unit ball in $\mathbb{R}^d$.
The same proof extends to tiling domains in $\mathbb{R}^d$.
Moreover, he conjectured that this lower bound is true for all bounded domains in $\mathbb{R}^d$.
This conjecture is still open.
 For general domains $\Omega$, Polya’s conjecture holds for the first two Dirichlet eigenvalues, which are respectively known as Faber-Krahn inequality  and  Krahn-Szeg\H{o} inequality~\cite{Henrot}.
Until recently, Pólya's conjecture had remained one of the core problems in spectral geometry.
Related results can be found in~\cite{Berezin, Li-Yau, kovarik2009Polya-CMP, Freitas2019Polya, Filonov2023Polya-Invent, He-Wang2024Polya, Freitas2025Polya, Guo2025Polya, Jiang-Lin2025Polya, Filonov2026Polya-JLMS, Frank2026Polya-CPAM, Lin2026Polya-SCM} and their references.

In discrete settings, one can also study a discrete version of P\'olya's conjecture.
The Faber--Krahn inequality states that, among bounded domains in $\mathbb{R}^n$ with fixed volume, the ball minimizes the first Dirichlet eigenvalue.
 Friedman~\cite{friedman1993some} first introduced the concept of ``a graph with boundary", and he further conjectured that a Faber--Krahn type inequality should hold for regular trees with given volume. 
However, this conjecture was shown to be false, see~\cite{leydold1996faber, pruss1998discrete}. 
Subsequently, Leydold~\cite{leydold2002geometry} established a Faber--Krahn inequality for regular trees and gave a complete characterization of the extremal trees, which are ball approximations.
B{\i}y{\i}ko{\u{g}}lu and Leydold~\cite{biyikouglu2007faber} also show that trees that have lowest first Dirichlet eigenvalue for a given degree sequence are ball approximations.
Moreover, they proposed the following problem.
\begin{problem}[{\cite[Problem 1]{biyikouglu2007faber}}]\label{pro:Faber-Krahn}
    Give a characterization of all graphs in a given class $\mathcal{C}$ with the $\mathit{Faber}$–$\mathit{Krahn}$ property, i.e., characterize those graphs in $\mathcal{C}$ which have minimal first Dirichlet eigenvalue for a given “volume”.
\end{problem}

Further results on Faber--Krahn inequalities for graphs can be found in~\cite{Rose, Quantitative, supertree, You, He-Yu1, He-Yu2, glued graphs}.
For higher Dirichlet eigenvalues, Bauer and Lippner~\cite{bauer2022eigenvalue} considered a discrete version of Pólya's conjecture based on Li--Yau's method~\cite{Li-Yau} for finite induced subgraphs of the $n$-dimensional integer lattice $\mathbb{Z}^n$, which can be regarded as a discrete form of $\mathbb{R}^n$. 
Furthermore, Hua and Li~\cite{MR4956579} extended Bauer and Lippner's result to poly-Laplace operators.

In this paper, we first establish a Krahn--Szeg\H{o} type inequality for trees.
Motivated by the problem of B{\i}y{\i}ko{\u{g}}lu and Leydold, we propose the following analogous problem.
\begin{problem}
Give a characterization of all graphs in a given class $\mathcal{C}$ with the $\mathit{Krahn}$–$\textit{Szeg\"o}$ property, i.e., characterize those graphs in $\mathcal{C}$ which have minimal second Dirichlet eigenvalue for a given “volume”.
\end{problem}
Our first result is a Krahn-Szeg\H{o} type inequality for trees with a prescribed number of interior vertices and boundary leaves.

For integers $s\geq 1$ and $q\geq 1$, let $C_{s,q}$ be the tree constructed as follows. 
Start with an internal path $v_1v_2\cdots v_s$.
Attach $q-1$ leaves to $v_1$ and one leaf to $v_s$. 
For integers $s\geq 1$ and $q\geq 2$, and for $w\geq1$, let $C_{s,q}^{w}$ be
a weighted tree obtained from $C_{s,q}$ by assigning  weight $w$ to one of
$q-1$ boundary edges incident to $v_1$, while all other edges have weight $1$.

\begin{theorem}\label{cor:tree_fixed_interior_leaves_lambda2}
Let $\mathcal T_{k,b}$ be the class of trees.
Then the minimum of $\lambda_2(T)$ among $T\in\mathcal T_{k,b}$ is
attained as follows.
\begin{enumerate}[(1)]
\item If $k=2t+1$, then equality holds if and only if $T$ is obtained from a path $P_{2t+3}=v_1\sim v_2\sim \cdots \sim v_{2t+3}$ by attaching the remaining
$b-2$ leaves to the  vertex $v_{t+2}$.
See~\cref{fig:tree-k-2t-plus-1}.
\item If $k=2t$ and $b=2r$, then equality holds if and only if $T$ is obtained from two copies of $C_{t,r}$ by joining $v_1$ in one $C_{t,r}$ and $v_1'$ in the other $C_{t,r}$. 
See~\cref{fig:tree-k-2t-b-2r}.
\item If $k=2t$ and $b=2r+1$, then equality holds if and only if $T$ is obtained from one copy of $C_{t,r}$ and one copy of
$C_{t,r+1}$ by joining $v_1$ in $C_{t,r}$ and $v_1'$ in  $C_{t,r+1}$. 
See~\cref{fig:tree-k-2t-b-2r}.
\end{enumerate}
\end{theorem}

\begin{figure}[htbp]
\centering
\begin{tikzpicture}[
    x=0.8cm,y=0.8cm,
    int/.style={circle,fill=black,inner sep=1.8pt},
    leaf/.style={circle,draw=black,fill=white,inner sep=1.8pt},
    every node/.style={font=\small}
]

\node[leaf] (a1)  at (-7.2,0) {};
\node[int]  (a2)  at (-5.4,0) {};
\node[int]  (at1) at (-1.6,0) {};
\node[int]  (ac)  at (0,0) {};
\node[int]  (at3) at (1.6,0) {};
\node[int]  (a22) at (5.4,0) {};
\node[leaf] (a23) at (7.2,0) {};

\node at (-3.5,0) {$\cdots$};
\node at (3.5,0) {$\cdots$};

\draw (a1)--(a2);
\draw (a2)--(-4.0,0);
\draw (-3.0,0)--(at1);
\draw (at1)--(ac);
\draw (ac)--(at3);
\draw (at3)--(3.0,0);
\draw (4.0,0)--(a22);
\draw (a22)--(a23);

\node[below=3pt] at (a1)  {$v_1$};
\node[below=3pt] at (a2)  {$v_2$};
\node[below=3pt] at (at1) {$v_{t+1}$};
\node[below=3pt] at (ac)  {$v_{t+2}$};
\node[below=3pt] at (at3) {$v_{t+3}$};
\node[below=3pt] at (a22) {$v_{2t+2}$};
\node[below=3pt] at (a23) {$v_{2t+3}$};

\node[leaf] (al1) at (-0.8,1.15) {};
\node[leaf] (al2) at (0.8,1.15) {};
\node at (0,1.15) {$\cdots$};

\draw (ac)--(al1);
\draw (ac)--(al2);

\node[above left=1pt]  at (al1) {$v_{t+2,1}$};
\node[above right=1pt] at (al2) {$v_{t+2,b-2}$};

\end{tikzpicture}
\caption{$k=2t+1$}
\label{fig:tree-k-2t-plus-1}
\end{figure}

\begin{figure}[htbp]
\centering
\begin{tikzpicture}[
    x=0.8cm,y=0.8cm,
    int/.style={circle,fill=black,inner sep=1.8pt},
    leaf/.style={circle,draw=black,fill=white,inner sep=1.8pt},
    every node/.style={font=\small}
]

\node[leaf] (bl)  at (-7.2,0) {};
\node[int]  (bt)  at (-5.4,0) {};
\node[int]  (b2)  at (-2.5,0) {};
\node[int]  (b1)  at (-0.9,0) {};
\node[int]  (b1p) at (0.9,0) {};
\node[int]  (b2p) at (2.5,0) {};
\node[int]  (btp) at (5.4,0) {};
\node[leaf] (br)  at (7.2,0) {};

\node at (-3.95,0) {$\cdots$};
\node at ( 3.95,0) {$\cdots$};

\draw (bl)--(bt);
\draw (bt)--(-4.45,0);
\draw (-3.45,0)--(b2)--(b1)--(b1p)--(b2p);
\draw (b2p)--(3.45,0);
\draw (4.45,0)--(btp)--(br);

\node[below=3pt] at (bt)  {$v_t$};
\node[below=3pt] at (bl)  {$v_{t,1}$};
\node[below=3pt] at (b2)  {$v_2$};
\node[below=3pt] at (b1)  {$v_1$};

\node[below=3pt] at (b1p) {$v_1'$};
\node[below=3pt] at (b2p) {$v_2'$};
\node[below=3pt] at (btp) {$v_t'$};
\node[below=3pt] at (br)  {$v_{t,1}'$};

\node[leaf] (b11) at (-1.95,1.30) {};
\node[leaf] (b12) at (-0.30,1.30) {};
\node at (-1.12,1.3) {$\cdots$};

\draw (b1)--(b11);
\draw (b1)--(b12);

\node[above left=1pt, xshift=-2pt, yshift=0pt] at (b11) {$u_1$};
\node[above right=1pt, xshift=6pt, yshift=1pt] at (b12) {$u_1'$};

\node[leaf] (bp11) at (0.30,1.30) {};
\node[leaf] (bp12) at (1.95,1.30) {};
\node at (1.12,1.3) {$\cdots$};

\draw (b1p)--(bp11);
\draw (b1p)--(bp12);

\node[above left=1pt, xshift=-6pt, yshift=1pt] at (bp11) {$u_{r-1}$};
\node[above right=1pt, xshift=0pt, yshift=-2pt] at (bp12) {$u_s'$};

\end{tikzpicture}
\caption{For $b=2r$, $s=r-1$. For $b=2r+1$, $s=r$.}
\label{fig:tree-k-2t-b-2r}
\end{figure}
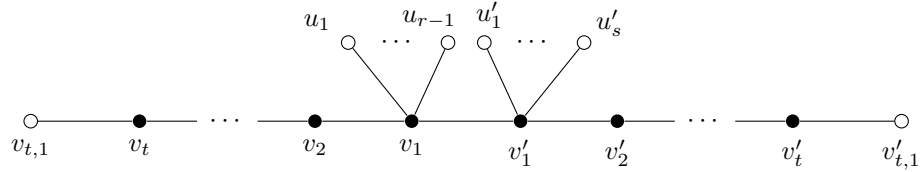

Our proof is based on  nodal domain  and geometric representation of graphs~\cite{friedman1993some}.
By geometric representation, we can always cut the  geometric representation of graph along certain vertices where the eigenfunction has zero entry, making it (internally) disconnected. 
This operation does not change the second Dirichlet eigenvalue of the graph, and therefore every connected graph can be viewed as a disconnected one.
Our proof is a little bit different from the proof in Euclidean case since discrete nodal domain theorem~\cite{nodal_domain_LAA} is different from nodal domain theorem in $\mathbb{R}^n$.
This idea may also apply to general discrete Schr\"odinger operators, including the negative of the adjacency matrix $A$.
Thus, by a similar method, we also obtain some general results for the second largest adjacency eigenvalues of graphs in some particular classes.
The key point is to construct a Dirichlet Schr\"odinger type adjacency operator.

For each positive integer $s$, let
$\alpha(s) \coloneqq \max\{\rho_1(H): H\in\mathcal C,\ |V(H)|=s\}$, where $\rho_1(H)$ is the adjacency spectral radius of $H$.
Let $G\cup H$ denote the disjoint union of two graphs $G$ and $H$. 
Let $G_1 \bullet K_1 \bullet G_2$ be the set of connected graphs obtained from disjoint union of $G_1 \cup G_2$ by adding some edges from their vertices to an additional vertex $K_1$, and let $\{G_{1}\cup G_{2}\}+e$ be the graph obtained from the disjoint union $G_1\cup G_2$ by joining an edge $e = uv$.
\begin{theorem}\label{lem:adjacency_nodal_domain}
Let $\mathcal C$ be a graph class. 
Assume that $\alpha(s)$ is strictly increasing in $s$.
Let $G \in \mathcal C$ of order $n$. 
Then
\begin{align*}
        \rho_2(G)\leq \alpha(\lfloor n/2\rfloor).
\end{align*}
Equality holds if and only if one of the following alternatives holds.
\begin{enumerate}[(1)]
\item If $n$ is even, then equality holds if and only if the graph $G$ is of the form $G_1^*(n/2)\cup G_2^*(n/2)$, where $G_1^*(n/2), G_2^*(n/2)\in \mathcal C, \rho_1(G_1^*(n/2))=\rho_1(G_2^*(n/2))=\alpha(n/2)$;

\item If $n$ is odd, then equality holds if and only if the graph $G$ is of the form $G_1^*(\lfloor n/2\rfloor)\cup G_2(\lceil n/2\rceil)$ or $G_1^*(\lfloor n/2\rfloor) \bullet K_1 \bullet G_2^*(\lfloor n/2\rfloor)$, where $G_1^*(n/2), G_2^*(n/2), G_2(\lceil n/2\rceil) \in \mathcal C,$\\ 
$\rho_1(G_1^*(\lfloor n/2\rfloor))=\rho_1(G_2^*(\lfloor n/2\rfloor))=\alpha(\lfloor n/2\rfloor)$, and
$\rho_1(G_2(\lceil n/2\rceil))\geq \rho_1(G_1^*(\lfloor n/2\rfloor)).$

In particular, equality holds when the graph $G$ is of the form $2G_1^*(\lfloor n/2\rfloor)\cup K_1$ and $ G_1^*(\lfloor n/2\rfloor)\cup G_1^*(\lceil n/2\rceil).$
\end{enumerate}
Moreover, if $G$ is connected, then equality holds if and only if $n$ is odd and the graph $G$ is of the form $G_1^*(\lfloor n/2\rfloor) \bullet K_1 \bullet G_2^*(\lfloor n/2\rfloor)$, where $G_1^*(\lfloor n/2\rfloor), G_2^*(\lfloor n/2\rfloor)\in \mathcal C, G_1^*(\lfloor n/2\rfloor) \bullet K_1 \bullet G_2^*(\lfloor n/2\rfloor)\in\mathcal C, \rho_1(G_1^*(\lfloor n/2\rfloor))=\rho_1(G_2^*(\lfloor n/2\rfloor))=\alpha(\lfloor n/2\rfloor).$
\end{theorem}

The monotonicity assumption on $\alpha(s)$ holds for many graph classes.
For example, it holds for the class of connected graphs, trees, bipartite graphs, planar graphs, outerplanar graphs, $K_{r}$-free graphs, graphs with no long cycles, graphs with no chorded cycles, graphs containing no cycles of prescribed congruence length, $K_{r}$-minor-free graphs, and so on.

For a graph $G\in\mathcal C$ and a vertex $u\in V(G)$, let $I_{uu}$ denote the
diagonal matrix whose only nonzero entry is the $(u,u)$-entry, which is equal to $1$.
For each positive integer $t$, define
$ \beta(t)\coloneqq \max\{\sigma_1(A(G)-I_{v v}):G\in\mathcal C,\ |V(G)|=t,\ v\in V(G)\}.$

\begin{theorem}\label{thm:connected_even_bridge}
Let $\mathcal C$ be a graph class.
Let $n$ be even. 
Assume that  $\beta(t)>\alpha(s)$ for every positive integer $s<t$.
Suppose that the value $\beta(\frac{n}{2})$ is attained by a unique pair $(G^*(\frac{n}{2}),u^*)$, up to isomorphism. 
Let $x$ be the positive eigenvector corresponding to $\sigma_1(A(G^*(\frac{n}{2}))-I_{u^*u^*})$.
Then $x_{u^*}=\min_{v\in V(G^*(\frac{n}{2}))}x_v$.
Moreover, for every connected graph $G\in\mathcal C$ of order $n$, we have
\begin{align*}
    \rho_2(G)\leq \beta(\frac{n}{2}).
\end{align*}
The equality holds if and only if $G\cong 2G^*(\frac{n}{2})+u_1u_2$, where $2G^*(\frac{n}{2})+u_1u_2 \in \mathcal{C}$ and $u_i$ is the copy of $u^*$ in the $i$-th copy of $G^*(\frac{n}{2})$ for $i=1,2$.
\end{theorem}

In 1993, Favaron, Mah\'eo and  Sacl\'e~\cite{Favaron1993DM} proved the following bound on the second largest adjacency eigenvalue.

\begin{theorem}[\cite{Favaron1993DM}]
     Let $G \not\cong K_2$ be a graph with $m$ edges and clique number $\omega$. 
     Then $|\rho_2(G)|\leq m/\omega$.
\end{theorem}
This bound is not tight.
In 2010, Aouchiche and Hansen \cite{A-H2010} posed the following conjecture which enhanced this upper bound.
\begin{conjecture}\label{conj:second largest eigenvalue}
Let $G$ be a connected graph on $n$ vertices with clique number $\omega$
and second largest eigenvalue $\rho_2$. Then the following hold.
\begin{itemize}
    \item If $n$ is odd, then $|\rho_2|\cdot \omega \leq m-2$,
    with equality if and only if $G$ is composed of
    $K_{\frac{n+1}{2}}$ and $K_{\frac{n-1}{2}}$ linked by an edge, or
    $K_{\frac{n-1}{2}}$ and $K_{\frac{n-1}{2}}$ linked by a path.
    \item If $n$ is even, then $|\rho_2|\cdot \omega - m$ is maximized if and only if $G$ is composed of two copies of
    $K_{\frac{n}{2}}$ linked by an edge.
\end{itemize}
\end{conjecture}

Note that the conjecture is false when $n=3$. 
Indeed, for $G\cong K_3$, we have $m=3$, $\omega=3$, and $\rho_2=-1$. 
Then $|\rho_2|\omega=3>1=m-2$.
There is nothing to prove in the case $n=1$.
Therefore, the odd-order case of the conjecture has to be stated with the additional assumption $n\geq 5$.
After excluding this exceptional case, we prove the following corrected version of the conjecture.

\begin{theorem}\label{thm:second largest eigenvalue}
Let $G$ be a connected graph on $n$ vertices with clique number $\omega$
and second largest eigenvalue $\rho_2$. Then the following hold.
\begin{itemize}
    \item If $n \geq5$ is odd, then $|\rho_2|\cdot \omega \leq m-2$,
    with equality if and only if $G$ is composed of
    $K_{\frac{n+1}{2}}$ and $K_{\frac{n-1}{2}}$ linked by an edge, or
    $K_{\frac{n-1}{2}}$ and $K_{\frac{n-1}{2}}$ linked by a path.
    \item If $n\geq2$ is even, then $|\rho_2|\cdot \omega - m$ is maximized if and only if $G$ is composed of two copies of
    $K_{\frac{n}{2}}$ linked by an edge.
\end{itemize}
\end{theorem}

The rest of the paper is organized as follows. 
In Section 2, we introduce the geometric representation and nodal domains of a weighted graph, and we recall several tools from matrix theory. 
Section 3 is devoted to a Krahn--Szeg\H{o} type inequality on trees. 
In Section 4, we introduce the Dirichlet adjacency operator and establish general theorems concerning the second largest adjacency eigenvalue. 
In Section 5, we apply our main theorems to several concrete cases. 
In particular, we strengthen the result of Brooks, Gu, Hyatt, Linz, and Lu on outerplanar graphs by removing the “sufficiently large” condition. 
Moreover, we resolve the Aouchiche–Hansen conjecture (2010) on the second largest eigenvalue for graphs with a given number of edges and a given clique number. 
Finally, in Section 6, we illustrate that many further results can be derived from our main theorems.

\section{Preliminaries}
In this section, we introduce the basic notation and tools used in the proofs of our main theorem.

Let $G = (V, E)$ be an undirected simple finite graph, where $V(G)$ and $E(G)$ denote the vertex set and edge set of $G$, respectively.
For a vertex $v \in V(G)$, we use $N_S(v)$ to denote the neighborhood of $v$ in $S\subseteq V(G)$.
 Let $y \sim x$ denote that $y$ is adjacent to $x$.
For two subsets $X, Y \subseteq V(G)$, denote $E (X, Y) \coloneqq \set{xy \in E \mid x \in X, y \in Y}$.
Let $e(H)=|E(H)|$ for $H\subseteq G$ and $\omega(G)$ be the clique number of $G$.
A weighted finite graph is a triple $(G,m,w)$, where $m$ is vertex measure and $w$ is edge measure (weight) which satisfies $w_{xy}=w_{yx}>0$ if and only if $xy\in E(G)$.






\subsection{Geometric representation and nodal domain of a weighted graph}
Our proofs rely on the geometric representation of a graph, which we recall as follows.
\begin{definition}[Geometric representation]
\begin{enumerate}
    \item For a simple graph $G$, let $K(G)$ be the one-dimensional simplicial complexwhose vertex set $V(G)$ corresponds to the set of $0$-simplices in $K(G)$ and the edge set $E(G)$ corresponding the set of 1-simplices such that the boundary points of the 1-simplex $\{x, y\}$ are $x$ and $y$. 
    Then, $K(G)$ is called the one-dimensional simplicial complex representing $G$.

\item A weight on an abstract one-dimensional simplicial complex assigns a measure to each $0$-simplex and a length to each $1$-simplex.

\item Let $(G, m, w)$ be a weighted graph. Assign to each 0-simplex $x$ of $K(G)$ the measure $m_x$ and assign to each 1-simplex $\{x, y\}$ of $K(G)$ the length $l_{x y}=\frac{1}{w_{x y}}$. 
We will simply denote such a weighted one-dimensional simplicial complex as $\left(K(G), m, \frac{1}{w}\right)$ and call it the geometric representation of $(G, m, w)$. We also simply denote $\left(K(G), m, \frac{1}{w}\right)$ as $K(G)$.

\item Let $(G, m, w)$ be a weighted graph and $K(G)$ be its geometric representation.
We then identify each 1-simplex $\{x, y\}$ in $K(G)$ with the interval $\left[0, \frac{1}{w_{x y}}\right]$. 
Let $f \in \mathbb{R}^{V(G)}$, we denote by $\widetilde{f}:|K(G)| \rightarrow \mathbb{R}$ the edgewise linear extension of $f$.
Here $|K(G)|$ is the underlying topological space of $K(G)$.
\end{enumerate}
\end{definition}

Based on geometric representation, we can define the nodal domain of graphs.
\begin{definition}[Nodal domain]

Let $(G,m,w)$ be a weighted connected finite graph with boundary and $K(G)$ be its geometric realization. 
Let $f\in \mathbb{R}^{V(G)}$ be an eigenfunction of $G$ and $U$ be a connected component of $|K(G)|\setminus \widetilde f^{-1}(0)$.
Then, $U$  is called a nodal domain of $f$.
\end{definition}

\subsection{Fundamental tools in matrix theory}
In this subsection, we recall some classical tools in matrix theory, which are useful in our proofs.

For an $n\times n$ matrix $M$, we denote its eigenvalues by $\sigma_1(M)\geq \sigma_2(M)\geq \cdots \geq \sigma_n(M)$.
A real matrix is called a Metzler matrix if all its off-diagonal entries are nonnegative.
Let $I$ be the identity matrix.
The following lemma is a general form and a direct corollary of Perron-Frobenius theorem.
\begin{lemma}
\label{lem:Metzler matrix_PF_theorem}
Let $M=B-sI\in \mathbb{R}^{n\times n}$, where $B$ is a non-negative matrix and $s\in \mathbb{R}$, and assume that $M$ is an irreducible Metzler matrix. Then the following statements hold:
\begin{enumerate}
\item[(i)] $\sigma_1(M)=\sigma_1(B)-s$, and it is a simple eigenvalue of $M$;
\item[(ii)] there exists a unique vector $\mathbf{v}\in \mathbb{R}^n$, $\mathbf{v}>\mathbf{0}$, such that $M\mathbf{v}=\sigma_1(M)\mathbf{v}$;
\item[(iii)] there exists a unique vector $\mathbf{w}\in \mathbb{R}^n$, $\mathbf{w}>\mathbf{0}$, such that $\mathbf{w}^{\top}M=\mathbf{w}^{\top}\sigma_1(M)$.

\end{enumerate}
\end{lemma}
 In our proofs, we also need Weyl's inequality and Cauchy's interlacing theorem.
\begin{lemma}[{\cite[Corollary~4.3.15]{horn1985matrix}}]\label{thm:Weyl_inequality}
Let $A,B\in M_n$ be Hermitian matrices. 
Assume that the eigenvalues are arranged in decreasing order. 
Then
\begin{align*}
     \sigma_i(A)+\sigma_n(B)
    \leq
    \sigma_i(A+B)
    \leq
    \sigma_i(A)+\sigma_1(B),
    \qquad i=1,\cdots,n.
\end{align*}
Moreover, equality in the upper bound holds if and only if there exists a
nonzero vector $x$ such that $Ax=\sigma_i(A)x, Bx=\sigma_1(B)x,$
and $(A+B)x=\sigma_i(A+B)x$.
The equality in the lower bound holds if and only if there exists a nonzero vector $x$ such that $Ax=\sigma_i(A)x, Bx=\sigma_n(B)x,$
and $(A+B)x=\sigma_i(A+B)x.$
If $A$ and $B$ have no common eigenvector, then every inequality above is strict.
\end{lemma}

\begin{lemma}[{\cite[Theorem~9.1.1]{Godsil2001GTM}}]\label{cor:spectral_radius_monotonicity}
Let $A$ be a real symmetric $n \times n$ matrix and let $B$ be a principal submatrix of $A$ with order $m \times m$. Then, for $i = 1, \dots, m$,
\[
\sigma_{n-m+i}(A) \leq \sigma_i(B) \leq \sigma_i(A).
\]
\end{lemma}


\section{The second smallest Dirichlet eigenvalue}

\subsection{Dirichlet Laplacian operator}

A finite weighted graph with boundary is a quadruple $(G, B,m,w)$, where $B \subset V$ is a non-empty set. 
The set $B$ is called the boundary and $\Omega \coloneqq V \setminus B$ is called  the interior. 
If $G$ is a tree, then we denote the pair $(G,B)$ by $T$ for simplicity, where we choose all leaves as boundary vertices.
For $G$, let $\Delta(G)$ be the Laplacian operator on $G$.
When the graph $G$ is clear from the context, we write $\Delta$ for simplicity.

The Dirichlet eigenvalue problem on a graph $G$ with boundary $B$ is defined as follows.
\begin{equation*}
    \begin{cases}
        -\Delta f(x) = \lambda f(x), & x\in \Omega, \\ 
         f(x) = 0, & x\in B,
    \end{cases}
\end{equation*}
where $\Delta f(x) = \frac{1}{m_x}\sum\limits_{y \sim x} w_{xy}(f(y)-f(x))$.
We call $\lambda$ the Dirichlet eigenvalue and $f$ the eigenfunction corresponding to $\lambda$.
The Dirichlet eigenvalues are ordered by $0 < \lambda_{1} \leq \lambda_{2} \leq \ldots \leq \lambda_{|\Omega|}$. 
For any $0 \neq f \in \RR^{\Omega}$, the Rayleigh quotient of Dirichlet Laplacian $\Delta_\Omega$ is defined as
\begin{align*}
    R_{(G,B)}(f) &\coloneqq \frac{\sum\limits_{xy \in E(G) }w_{xy}(\widehat{f}(x)-\widehat{f}(y))^{2}}{\sum\limits_{x \in \Omega} f^{2}(x)m_x},
\end{align*}
where $\widehat{f}$ is obtained by extending $f$ to $\RR^V$ by $0$.
The variational characterizations of $\lambda_{k}$ are given by
\begin{align*}
    \lambda_{k}(G,B) &= \min_{\substack{W \subset \mathbb{R}^{\Omega}\\ \dim W=k}} \max_{\substack{0 \neq f \in W}} R_{(G,B)}(f).
\end{align*}
Define a matrix $L_\Omega = (l_{xy})_{\Omega \times \Omega}$, where
\begin{align*}
    l_{xy} = 
    \begin{cases}
        \frac{1}{m_x}\sum\limits_{z\sim x}w_{xz}, & \textif x = y, \\
        -\frac{w_{xy}}{m_x}, & \textif x \sim y, \\
        0, &  \textotherwise.
    \end{cases}
\end{align*}
 When $\Omega$ is finite, the eigenvalue problem for $L_\Omega$ is equivalent to the Dirichlet eigenvalue problem.
Therefore, by Lemma~\ref{lem:Metzler matrix_PF_theorem}, we obtain the following proposition directly.
\begin{proposition}\label{pro:Dirichlet eigenvalue_properties}
Let $(G, B,m,w)$ be a weighted connected finite graph with boundary. 
Then
\begin{enumerate}
\item $-\Delta_\Omega$ is a positive operator, i.e. $\lambda_1(G,B) > 0$;
\item an eigenfunction $f$ corresponding to the eigenvalue $\lambda_1(G,B)$ of $-\Delta_\Omega$ is either positive or negative on all interior vertices of $G$;
\item $\lambda_1(G,B)$ is a simple eigenvalue.
\end{enumerate}
\end{proposition}

 For any $x, y \in |K (G)|$, if $x$ and $y$ are contained in the same edge of $K (G)$, we denote by $[xy]$ the part of that edge lying between $x$ and $y$.
Next, we introduce  notation for geometric realization  and  the edge weights in Dirichlet setting.
\begin{definition}\label{def:Dirichlet_U}
Let $(G,B,m,w)$ be a finite weighted graph with boundary, and $K(G)$ be the geometric realization of $G$. 
Let $\Omega=V(G)\setminus B$. 
For any connected open subset $U$ of $|K(G)|$, we define its induced graph $G_U$ to be the weighted
graph with Dirichlet boundary as follows.

\begin{enumerate}
    \item
    $V(G_U)=\{x\in \overline U:\ x\in V(G)\text{ or }x\in \partial U\}$;

    \item
    $E(G_U)=\{xy:\ x\neq y\in V(G_U),\ x,y\text{ lie on the same edge of }
    K(G),\text{ and }(xy)\subset U\}$, where $(xy) = [xy] \setminus \{x, y\}$;

    \item $B_U=B\cap \overline U$, $B_D(G_U)=\partial U\setminus B_U$, $\partial_DG_U=B_U\cup B_D(G_U)$;

    \item $\Omega_U=V(G_U)\setminus \partial_DG_U=U\cap \Omega$;

    \item for any $xy\in E(G_U)$, $w_{xy}=\frac{1}{l_{x y}}$;

    \item for any $x\in \Omega_U$, $m_x=m_x$.
\end{enumerate}

The measures of boundary vertices do not affect the Dirichlet eigenvalue problem.
Let $\widetilde f: |K(G)|\to \mathbb R$ denote the edgewise linear extension of $\widehat f$ to $|K(G)|$.
Let $U$ be a connected component of $|K(G)|\setminus \widetilde f^{-1}(0)$. 
By the definition
of $G_U$, every vertex in $\partial_DG_U$
lies in $\widetilde f^{-1}(0)$. 
For $B_D(G_U)$, if $uv\in E(G)$,
$u\in \Omega_U$, $v\in \Omega\setminus \Omega_U$, and $f(v)=0$, then we
put a copy $\widetilde v$ of $v$ into $B_D(G_U)$. If $uv\in E(G)$,
$u\in \Omega_U$, $v\in \Omega\setminus \Omega_U$, and $f(u)f(v)<0$, then
there is a unique zero point $z_{uv}$ of $\widetilde f$ on the edge
$uv$, and we put $z_{uv}$ into $B_D(G_U)$. 
Thus $B_D(G_U)= B_U^0\cup \widetilde B_U^0,$
where
\begin{align*}
    B_U^0
    &=
    \{\widetilde v:\ u\in \Omega_U,\ v\in \Omega\setminus \Omega_U,\ uv\in E(G),
    \ f(v)=0\},\\
    \widetilde B_U^0
    &=
    \{z_{uv}:\ u\in \Omega_U,\ v\in \Omega\setminus \Omega_U,\ uv\in E(G),
    \ f(u)f(v)<0,\ \widetilde f(z_{uv})=0\}.
\end{align*}
\end{definition}
For $uv\in E(G_U)$, if $f(u)f(v)\geq 0$, then $uv\in E(G)$ and its weight remains unchanged.
If $f(u)f(v)<0$, then we obtain a vertex $z_{uv}\in \widetilde B_U^0$ that lies on  $uv$, and we have
\begin{align*}
    w_{uz_{uv}}=\frac{1}{l_{uz_{uv}}} =w_{uv}\frac{|f(u)-f(v)|}{|f(u)|}, \quad 
    w_{vz_{uv}}=\frac{1}{l_{vz_{uv}}} =w_{uv}\frac{|f(u)-f(v)|}{|f(v)|}.
\end{align*}


Now, we can prove the Dirichlet nodal domain theorem.

\begin{theorem}\label{thm:dirichlet_nodal_D}
Let $(G,B,m,w)$ be a finite connected weighted graph with boundary. 
Let $f$ be a Dirichlet eigenfunction of $(G,B)$ corresponding to the eigenvalue $\lambda$, which is not the smallest one.
For each nodal domain $U$ of $f$,
$ f|_{\Omega_U}$ is a Dirichlet eigenfunction of $G_U$ with
vanishing data on $\partial_DG_U=B_U\cup B_D(G_U)$. Moreover, $\lambda_1(G_U,\partial_DG_U)=\lambda.$
\end{theorem}

\begin{proof}
Fix a  nodal domain $U$ of $f$.
By the definition of nodal domain,
$\Omega_U\neq\emptyset$ and $G_U[\Omega_U]$ is connected. 
Moreover, since $\lambda$ is not the smallest eigenvalue, the eigenfunction
$f$ cannot be everywhere positive or everywhere negative. 
Then $U\neq |K(G)|$. 
Take $u\in \Omega_U$. 
If $uv$ is an edge of $G$ with
$v\in \Omega\setminus \Omega_U$ and $f(v)\neq 0$, then $uv$ is a
sign-changing edge. 
Let $z_{uv}$ be the unique zero point of
$\widetilde f$ on $uv$. 
We have 
\begin{align*}
    w_{uz_{uv}}=\frac{1}{l_{uz_{uv}}} =w_{uv}\frac{|f(u)-f(v)|}{|f(u)|}, \quad 
    w_{vz_{uv}}=\frac{1}{l_{vz_{uv}}} =w_{uv}\frac{|f(u)-f(v)|}{|f(v)|}.
\end{align*}
It follows that, for any $u\in\Omega_U$,
\begin{align*}
\lambda f(u)&=\frac{1}{m_u}
\sum_{\substack{v\in V(G)\\ v\sim u}}
w_{uv}\bigl(\widehat f(u)-\widehat f(v)\bigr)\\
&=\frac{1}{m_u}
\sum_{\substack{v\sim u\\ v\in \Omega_U}}
w_{uv}\bigl( f(u)- f(v)\bigr)
+
\frac{1}{m_u}\sum_{\substack{v\sim u\\ v\in B_U}}
w_{uv}\bigl( f(u)-0\bigr)\\
&\quad
+
\frac{1}{m_u}\sum_{\substack{\widetilde v\sim u\\ \widetilde v\in B_U^0}}
w_{u\widetilde v}\bigl( f(u)-\widetilde f(\widetilde v)\bigr)
+
\frac{1}{m_u}\sum_{\substack{z_{uv}\sim u\\ z_{uv}\in \widetilde B_U^0}}
w_{uz_{uv}}\bigl( f(u)-\widetilde f(z_{uv})\bigr)\\
&=\frac{1}{m_u}
\sum_{\substack{y\in V(G_U)\\ y\sim u}}
w_{uy}\bigl( \widetilde f(u)- \widetilde f(y)\bigr)=-\Delta_{\Omega_U}({G_U}) f(u).
\end{align*}
Thus, $\lambda$ is an eigenvalue of $-\Delta_{\Omega_U}({G_U})$ and $ f|_{\Omega_U}$ is a Dirichlet eigenfunction of $G_U$ with vanishing data on $\partial_DG_U$.
It remains to show that $\lambda =\lambda_1(G_U,\partial_DG_U)$. 
Since $U$ is a nodal domain, $f$ has a fixed sign on $\Omega_U$ and $f(u)\neq 0$ for all $u\in\Omega_U$.
By Proposition~\ref{pro:Dirichlet eigenvalue_properties}, this implies $\lambda_1(G_U,\partial_DG_U)=\lambda(G,B)$.
The proof is complete.
\end{proof}

\begin{figure}[htbp]
\centering
\begin{tikzpicture}[
    edge/.style={line width=0.8pt},
    posnode/.style={circle,draw=red!70!black,fill=red!35,minimum size=6pt,inner sep=0pt},
    negnode/.style={circle,draw=blue!70!black,fill=blue!35,minimum size=6pt,inner sep=0pt},
    solidzeronode/.style={circle,draw=gray!75!black,fill=gray!45,minimum size=6pt,inner sep=0pt},
    bdynode/.style={circle,draw=orange!85!black,fill=white,minimum size=6pt,inner sep=0pt,line width=0.8pt},
    lab/.style={
        font=\scriptsize,
        fill=white,
        fill opacity=1,
        text opacity=1,
        inner sep=1.2pt,
        outer sep=0pt
    },
    weightlab/.style={
        font=\scriptsize,
        fill=white,
        fill opacity=1,
        text opacity=1,
        inner sep=1.2pt,
        outer sep=0pt
    },
    >=Latex
]


\node at (0,3.3) {\small $(G,B)$};

\node[solidzeronode] (z) at (0,0) {};

\node[posnode] (a) at (0,1.15) {};
\node[posnode] (b) at (-0.9,2.05) {};
\node[posnode] (c) at (0.9,2.05) {};

\node[negnode] (a0) at (0,-1.15) {};
\node[negnode] (c0) at (0.9,-2.05) {};
\node[negnode] (b0) at (-0.9,-2.05) {};

\node[bdynode] (p2) at (-1.85,2.75) {};
\node[bdynode] (p1) at (-1.85,-2.75) {};

\draw[edge] (z) -- (a);
\draw[edge] (z) -- (a0);

\draw[edge] (a) -- (b);
\draw[edge] (b) -- (c);
\draw[edge] (c) -- (a);

\draw[edge] (a0) -- (c0);
\draw[edge] (c0) -- (b0);
\draw[edge] (b0) -- (a0);

\draw[edge] (c) -- (c0);

\draw[edge] (p2) -- (b);
\draw[edge] (p1) -- (b0);

\coordinate (zcc0) at (0.9,0);
\draw[orange!85!black,line width=1pt]
    ($(zcc0)+(-0.07,-0.07)$) -- ($(zcc0)+(0.07,0.07)$);
\draw[orange!85!black,line width=1pt]
    ($(zcc0)+(-0.07,0.07)$) -- ($(zcc0)+(0.07,-0.07)$);

\draw[red!70!black,dashed,rounded corners] (-1.25,0.85) rectangle (1.25,2.30);
\draw[blue!70!black,dashed,rounded corners] (-1.25,-2.30) rectangle (1.25,-0.85);

\node[lab, text=gray!80!black, anchor=east]  
    at ($(z)+(-0.12,0)$) {$z:0$};

\node[lab, text=red!80!black, anchor=west]  
    at ($(a)+(0.12,0)$) {$a:\tau$};

\node[lab, text=red!80!black, anchor=east]  
    at ($(b)+(-0.20,0)$) {$b:\tau$};

\node[lab, text=red!80!black, anchor=west]  
    at ($(c)+(0.12,0)$) {$c:1$};

\node[lab, text=blue!80!black, anchor=east]  
    at ($(a0)+(-0.2,0)$) {$a_0:-\tau$};

\node[lab, text=blue!80!black, anchor=west]  
    at ($(c0)+(0.12,0)$) {$c_0:-1$};

\node[lab, text=blue!80!black, anchor=east]  
    at ($(b0)+(-0.20,0)$) {$b_0:-\tau$};

\node[lab, text=orange!85!black, anchor=east]  
    at ($(p2)+(-0.12,0)$) {$p_2:0$};

\node[lab, text=orange!85!black, anchor=east]  
    at ($(p1)+(-0.12,0)$) {$p_1:0$};

\node[lab, text=orange!85!black, anchor=west]  
    at ($(zcc0)+(0.12,0)$) {$z_{cc_0}:0$};

\draw[->,line width=1pt] (2.05,0) -- (4.00,0);


\node at (5.80,3.3) {\small $(G_U,\partial_D G_U)$};

\node[posnode] (Ua) at (5.75,-0.4) {};
\node[posnode] (Ub) at (4.95,0.7) {};
\node[posnode] (Uc) at (6.55,0.7) {};

\node[bdynode] (Up2)   at (3.95,1.7) {};
\node[bdynode] (Uzt)   at (5.75,-1.8) {};
\node[bdynode] (Uzcc0) at (6.55,-0.35) {};

\draw[edge] (Ua) -- (Ub);
\draw[edge] (Ub) -- (Uc);
\draw[edge] (Uc) -- (Ua);

\draw[edge] (Ub) -- (Up2);
\draw[edge,dashed]  (Ua) -- (Uzt);
\draw[edge,dashed] (Uc) -- (Uzcc0);

\node[weightlab] at (3.75,1.28) {$w=1$};
\node[weightlab] at (5.15,-1.3) {$w=1$};
\node[weightlab] at (7.15,0.12) {$w=2$};

\draw[red!70!black,dashed,rounded corners] (4.10,-1.05) rectangle (7.80,1.18);

\node[lab, text=red!80!black, anchor=east]  
    at ($(Ua)+(-0.12,0)$) {$a:\tau$};

\node[lab, text=red!80!black, anchor=east]  
    at ($(Ub)+(-0.12,0)$) {$b:\tau$};

\node[lab, text=red!80!black, anchor=west]  
    at ($(Uc)+(0.12,0)$) {$c:1$};

\node[lab, text=orange!85!black, anchor=east]  
    at ($(Up2)+(-0.12,0)$) {$p_2:0$};

\node[lab, text=orange!85!black, anchor=north]  
    at ($(Uzt)+(0,-0.12)$) {$\widetilde z:0$};

\node[lab, text=orange!85!black, anchor=west]  
    at ($(Uzcc0)+(0.12,0)$) {$z_{cc_0}:0$};

\node[lab, text=red!70!black] 
    at (7.70,1.6) {$\Omega_U=\{a,b,c\}$};

\node[align=center, font=\small] at (2.80,-3.52) {
$-\Delta_\Omega(G) f=(3-\sqrt3)f$ on $\Omega$
\qquad$\Longrightarrow$\qquad
$\Delta_{\Omega_U}({G_U}) f|_{\Omega_U}
=(3-\sqrt3) f|_{\Omega_U}$ on $\Omega_U$
};

\end{tikzpicture}
\caption{Construction of the nodal domain graph $G_U$ from the positive nodal domain $U$. Here $\tau= \frac{1+\sqrt3}{2}$.}
\label{fig:nodal-domain-dirichlet-subgraph}
\end{figure}
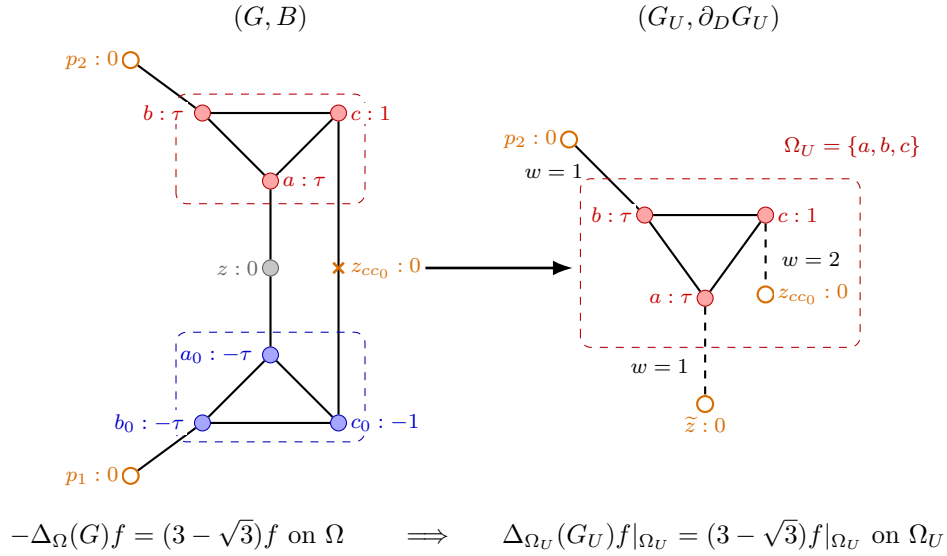

\begin{example}
To better understand Theorem~\ref{thm:dirichlet_nodal_D}, 
we consider the graph $G$ shown in~\cref{fig:nodal-domain-dirichlet-subgraph}.
Let $B=\{p_1,p_2\}$.
The characteristic polynomial of $L_\Omega$ is given by
$P(G,B;x)=(x-4)(x^2-6x+6)(x^4-10x^3+32x^2-34x+6)$.
Then $\lambda_2(G,B)=3-\sqrt3$.
Let $\tau=\frac{1+\sqrt3}{2}$.
With respect to the vertex ordering $(z,a,b,c,a_0,c_0,b_0)$, an eigenfunction
corresponding to $\lambda_2(G,B)$ can be chosen as
$f=(0,\tau,\tau,1,-\tau,-1,-\tau)^\top$.
Now consider the positive nodal domain $U$ with $\Omega_U=\{a,b,c\}$.
Since $f(c)=1$ and $f(c_0)=-1$, the point $z_{cc_0}$ is the midpoint of $cc_0$, and the induced boundary edge weight is
$w_{c z_{cc_0}}=\frac{f(c)-f(c_0)}{f(c)}=2$.
Furthermore, the original boundary vertex in $G_U$ is $p_2$, and the vertex $z$ gives rise to the new boundary vertex $\widetilde z$ in $G_U$.
Thus $B_U=\{p_2\}$, $B_D(G_U)=\{\widetilde z,z_{cc_0}\}$, and
$\partial_DG_U=B_U\cup B_D(G_U)=\{p_2,\widetilde z,z_{cc_0}\}$.
With respect to the ordering $(a,b,c)$ of $\Omega_U$, the Dirichlet Laplacian matrix of $G_U$ is
\begin{align*}
    L_{\Omega_U}(G_U,\partial_DG_U)
    =
    \begin{pmatrix}
    3&-1&-1\\
    -1&3&-1\\
    -1&-1&4
    \end{pmatrix}.
\end{align*}
Its Dirichlet eigenvalues are $3-\sqrt3$, $4$, and $3+\sqrt3$.
Moreover, $ f|_{\Omega_U}=(\tau,\tau,1)^\top$ satisfies
\begin{align*}
    L_{\Omega_U}(G_U,\partial_DG_U)
    \begin{pmatrix}
    \tau\\
    \tau\\
    1
    \end{pmatrix}
    =
    (3-\sqrt3)
    \begin{pmatrix}
    \tau\\
    \tau\\
    1
    \end{pmatrix}.
\end{align*}
Therefore, $-\Delta_{\Omega_U}({G_U})\bigl(f|_{\Omega_U}\bigr)
=(3-\sqrt3) f|_{\Omega_U}$ on $\Omega_U$.
\end{example}

In the proof of the main theorem, we only consider the unweighted case $m\equiv w\equiv1$ on $G$.

\subsection{\texorpdfstring{Trees with fixed $k$ and $b$}{Trees with fixed k and b}}

We consider a family of trees as follows.
For integers $k\geq 2$ and $b\geq 2$, let
 \begin{align*}
        \mathcal{T}_{k,b}
        =
        \left\{
        T:
        T\text{ is a tree with } k \text{ interior vertices and } b \text{ leaves.}
        \right\}.
\end{align*}
For $w>1$, let
\begin{align*}
        \mathcal{T}_{k,b,w}
        =
        \left\{
        T\in \mathcal{T}_{k,b}: 
        \begin{array}{l}
        \text{one boundary edge has weight }w,\\
        \text{and the remaining }b-1\text{ boundary edges have weight }1
        \end{array}
        \right\}.
\end{align*}
For $a>0$ and $s\geq 1$, let $T_s(a)$ be the following tridiagonal matrix
\begin{align*}
T_s(a)=
\begin{pmatrix}
a&-1\\
-1&2&-1\\
&-1&2&\ddots\\
&&\ddots&\ddots&-1\\
&&&-1&2
\end{pmatrix}.
\end{align*}
For simplicity, denote the smallest eigenvalue of $T_s(a)$  by $\sigma_s(a)$.
Note that the matrix $T_s(q)$ is the Dirichlet Laplacian matrix of $C_{s,q}$ when $s \geq 2$. 
We present a lemma about the monotonicity of $\sigma_s(a)$.
\begin{lemma}\label{lem:sigma}
For every $s\geq 1$ and every $0<e<f$,  $\sigma_s(e)<\sigma_s(f)$.
For every $s\geq 1$, $i>0$ and $j\geq 2$, $\sigma_{s+1}(i)<\sigma_s(j)$.
\end{lemma}

\begin{proof}
We first prove $\sigma_s(e)<\sigma_s(f)$. Let $s\geq 1$ and $0<e<f$.
By Lemma \ref{lem:Metzler matrix_PF_theorem},  $\sigma_s(f)$ is simple and has a positive unit eigenvector $\mathbf{x}=(x_1,\ldots,x_s)^\top$.
Then 
\begin{equation}
\begin{aligned}\label{eq:T_1}
\sigma_s(f)
&=\mathbf{x}^\top T_s(f)\mathbf{x}=\mathbf{x}^\top T_s(e)\mathbf{x}+(f-e)x_1^2\\
&\geq \sigma_s(e)+(f-e)x_1^2 >\sigma_s(e).
\end{aligned}
\end{equation}

Next we prove $\sigma_{s+1}(i)<\sigma_s(j)$.
Denote the  positive unit eigenvector of $\sigma_s(2)$ by $\mathbf{y}=(y_1,\cdots,y_s)^\top$.
Let $\mathbf{z}^\top=(0,\mathbf{y}^\top)$.
Then
\begin{align*}
\sigma_{s+1}(i)<\mathbf{z}^\top T_{s+1}(i)\mathbf{z} =\mathbf{y}^\top T_s(2)\mathbf{y}=\sigma_s(2).
\end{align*}
Combining this with~\eqref{eq:T_1}, we have $\sigma_{s+1}(i)<\sigma_s(2)\leq \sigma_s(j)$.
\end{proof}

Before proving the main theorem, we need a weighted version of Faber--Krahn inequality for the first Dirichlet eigenvalue of trees.
\begin{theorem}[\cite{biyikouglu2007faber}, Klobučarštel theorem]\label{Klobučarštel theorem}
    A tree $T$ has the Faber--Krahn property in the class $\mathcal{T}_{k,b}$ if and only if $T$ is a star with a long tail, i.e., a comet. 
    $T$ is then uniquely determined up to isomorphism.
\end{theorem}

\begin{theorem}\label{Klobučarštel theorem_weighted}
A tree $T$ has the Faber--Krahn property in the class $\mathcal{T}(k,b,w)$ if and only if $T\cong C_{k,b}^{w}$.
\end{theorem}
\begin{proof}
The proof is similar to that of Theorem~\ref{Klobučarštel theorem} and shows that the extremal tree is of the form of a comet $C_{k,b}$ with $f(v_k)>f(v_1)$, where $f$ is a positive unit eigenfunction corresponding to $\lambda_1(T)$. 
It remains to determine where the unique boundary edge of weight $w$ is placed.
We claim that this edge must be incident to $v_1$.
Suppose not, then the boundary edge of weight $w$ is incident to $v_k$.
Denote this edge by $v_0v_k$. 

Construct a new weighted tree $T'$ by replacing the two boundary edges $v_0v_k$ and $v_1v_{k+1}$ with $v_0v_1$ and $v_kv_{k+1}$, respectively,
where $v_0v_1$ has weight $w$ and $v_kv_{k+1}$ has weight $1$.
Then $T'\in \mathcal{T}_{k,b,w}$, and
\begin{align*}
        R_{T'}(f)-R_T(f)
        &=
        \left(f^2(v_k)+w f^2(v_1)\right)
        -
        \left(w f^2(v_k)+f^2(v_1)\right)  \\
        &=
        (w-1)\left(f^2(v_1)-f^2(v_k)\right)
        <0 .
\end{align*}
Therefore $\lambda_1(T')\leq R_{T'}(f)<R_T(f)=\lambda_1(T)$, which
contradicts the Faber--Krahn property of $T$.
Hence $T\cong C_{k,b}^{w}$.
\end{proof}

\begin{proof}[\bf{Proof of Theorem \ref{cor:tree_fixed_interior_leaves_lambda2}}]

Let $B$ be the set of leaves.
Let $f$ be an eigenfunction
corresponding to $\lambda_2(T)$. 
Let
$U_1,\cdots,U_\ell$ be the nodal domains of $f$, and let $k_h=|\Omega_{U_h}|$ and $b_h=|\partial_DG_{U_h}|$.
Since $f$ changes sign, we have
$\ell\geq 2$
Let $w_h\geq 1$ be the weight of new boundary edge incident to $\Omega_{U_h}$.
By Theorem~\ref{thm:dirichlet_nodal_D}, for every $h=1,\cdots,\ell$,
\begin{align}\label{eq:nodal_lambda2_all}
\lambda_2(T)=\lambda_1(G_{U_h},\partial_DG_{U_h}).
\end{align}
Since each nodal domain has at least two Dirichlet boundary vertices,  Lemma~\ref{lem:sigma} and Theorem~\ref{Klobučarštel theorem} give
\begin{align}\label{eq:basic_nodal_bound}
    \lambda_1(G_{U_h},\partial_DG_{U_h})\geq \sigma_{k_h}(b_h+w_h-1)\geq \sigma_{k_h}(b_h)\geq \sigma_{k_h}(2),
    \qquad h=1,\ldots,\ell .
\end{align}
\begin{case}\label{Case1}
    $k=2t+1$.

Since $\sum_{h=1}^{\ell}k_h\leq k=2t+1$, 
there exists a nodal domain $U_i$ such that $k_i\leq t$. 
By lemma \ref{lem:sigma}, \eqref{eq:nodal_lambda2_all}, and \eqref{eq:basic_nodal_bound}, we have $\lambda_2(T)\geq \sigma_t(2)$.

Now we characterize the condition that the equality holds.
Assume that equality $\lambda_2(T)=\sigma_t(2).$ 
 Then $k_i=t$, $b_i= 2$, $w_i=1$, and $G_{U_i}\cong C_{t,2}$ by~\eqref{eq:basic_nodal_bound}.  
Denote the unique Dirichlet boundary vertex in $B_D(G_{U_i})$ by $z$.  
$w_i=1$ implies that $z\in V(T)$.
For another nodal domain $U_j$, 
by \eqref{eq:nodal_lambda2_all} and
\eqref{eq:basic_nodal_bound}, we have
\begin{align*}
    \sigma_t(2)
    =
    \lambda_2(T)
    =
    \lambda_1(G_{U_j},\partial_DG_{U_j})
    \geq
    \sigma_{k_j}(2).
\end{align*}
By Lemma~\ref{lem:sigma}, $k_j\geq t$.
Therefore $k_i=k_j=t$ and $z$ is the unique zero point of $f$.
Applying the equality case of \eqref{eq:basic_nodal_bound} to $U_j$, we
also get $b_j=2$, $w_j=1$ and $G_{U_j}\cong C_{t,2}$.  
Therefore  two nodal domains are two copies of $C_{t,2}$, and they are linked by $z$.  
Each copy  contains exactly one leaf of $T$, and hence the remaining $b-2$ leaves cannot belong to either nodal domain.  
Therefore they must be attached to $z$.

Conversely, suppose that $T$ is obtained in this way.
Let $C_{t,2}^{+}$ and $C_{t,2}^{-}$ be the two copies of $C_{t,2}$ in $T$. 
Let $g$ be a $\sigma_t(2)$-eigenfunction of $C_{t,2}$. 
Define a test function $h$ on the whole $T$ by
\begin{align*}
    h(v)=
\begin{cases}
g(v), & v\in V(C_{t,2}^{+})\setminus\{z\},\\[2mm]
-g(v), & v\in V(C_{t,2}^{-})\setminus\{z\},\\[2mm]
0, & \text{otherwise}.
\end{cases}
\end{align*}
Then 
\begin{align*}
    \lambda_2(T)&\leq
    \frac{\sum\limits_{xy\in E(C_{t,2}^{+})}(h(x)-h(y))^2+\sum\limits_{xy\in E(C_{t,2}^{-})}(h(x)-h(y))^2}
         {\sum\limits_{x\in \Omega(C_{t,2}^{+})}h(x)^2+\sum\limits_{x\in \Omega(C_{t,2}^{-})}h(x)^2}\\
    &=
    \frac{\sum\limits_{xy\in E(C_{t,2})}(g(x)-g(y))^2}
         {\sum\limits_{x\in \Omega(C_{t,2})}g(x)^2}=
    \sigma_t(2).
\end{align*}
Together with the lower bound already proved, we obtain $\lambda_2(T)=\sigma_t(2)$.
\end{case}

\begin{case}
 $k=2t$. 

We first show that, for an extremal tree, the number of nodal domains must be exactly two, and both of them have $t$ interior vertices.
Indeed, otherwise there exists a nodal domain $U_i$ such that $k_i\leq t-1$.  
By \eqref{eq:nodal_lambda2_all},
\eqref{eq:basic_nodal_bound}, and Lemma~\ref{lem:sigma}, we have
\begin{align*}
    \lambda_2(T)
    \geq
    \sigma_{k_i}(2)
    \geq
    \sigma_{t-1}(2) >\sigma_{t}(a),
\end{align*}
where $a$ is the total edge weight incident with the interior vertex $v_1$ of the corresponding weighted comet. 
It follows that $T$ cannot be the extremal tree.  
Thus, an extremal
tree must have exactly two nodal domains, denoted by $U_1$ and $U_2$, with $k_1=k_2=t$. 
Hence the two nodal domains are linked by a unique sign-changing edge.  
Denote this edge by $uv$, where $u\in \Omega_{U_1}$, $v\in \Omega_{U_2}$, and
$f(u)>0>f(v)$.  
It follows that the zero point $z_{uv}$ satisfies
\begin{align*}
    l_{uz_{uv}}=
    \frac{| f(u)|}{| f(u)- f(v)|}l_{uv},
\quad
     l_{vz_{uv}}=
    \frac{| f(v)|}{| f(u)- f(v)|}l_{uv},
    \quad
    w_{uz_{uv}}=\frac{1}{l_{uz_{uv}}},
    \quad
     w_{vz_{uv}}=\frac{1}{l_{vz_{uv}}}.
\end{align*}
For simplicity, write $ w_1=w_{uz_{uv}}$ and $ w_2=w_{vz_{uv}}$.
Then
\begin{align}\label{eq:w_relation}
    \frac1{w_1}+\frac1{w_2}=1.
\end{align}
In particular, $w_1>1$ and $w_2>1$.

Let $p_i= |B_{U_i}|$. 
Then $p_1+p_2=b$.  
For $G_{U_i}$, the Dirichlet boundary consists of  $p_i$ leaves in $T$ and  $z_{uv}$.   
Hence, by \eqref{eq:nodal_lambda2_all} and \eqref{eq:basic_nodal_bound}, we have
\begin{align}\label{eq:even_weighted_bound_correct}
    \lambda_2(T)
    \geq
    \max\left\{
        \sigma_t(p_1+w_1),
        \sigma_t(p_2+w_2)
    \right\}.
\end{align}
Since $\sigma_t(a)$ is strictly increasing in $a$, it remains to minimize $\max\{p_1+w_1,\ p_2+w_2\}$ subject to the condition~\eqref{eq:w_relation}.

\begin{subcase}
  $b=2r$.

By~\eqref{eq:w_relation} and Cauchy--Schwarz inequality, we have $w_1+w_2\geq 4$, and the  equality holds if and only if $w_1=w_2=2$.
Therefore
\begin{align*}
    \max\{p_1+w_1,\ p_2+w_2\}
    \geq
    \frac{p_1+p_2+w_1+w_2}{2}
    \geq
    r+2 .
\end{align*}
It follows that 
   $ \lambda_2(T)\geq \sigma_t(r+2).$
Now suppose equality holds.  
Then we must have $w_1=w_2=2$ and $p_1=p_2=r$.
Moreover, Theorem~\ref{Klobučarštel theorem_weighted} implies that both nodal domains are weighted comets.  
Since each nodal domain contains $r$ leaves in $T$ and one new boundary edge of weight $2$,  both of $G_{U_1}-z_{uv}$ and $G_{U_2}-z_{uv}$ are a copy of $C_{t,r}$.  
The sign-changing edge must be attached at the edges with weight $2$ of two comets.  
Therefore $T$ is obtained from two copies of $C_{t,r}$ by joining $v_1$ in one $C_{t,r}$ and $v_1'$ in the other $C_{t,r}$. 

Conversely, the proof is similar to Case~\ref{Case1}.
\end{subcase}

\begin{subcase}
$b=2r+1$. 

WLOG, assume $p_1\leq p_2$.
Let $d=p_2-p_1$.
Then $d$ is a positive odd integer.  We claim that
\begin{align}\label{eq:odd_balance_bound}
    \max\{p_1+w_1,\ p_2+w_2\}
    \geq
    r+1+\frac{1+\sqrt5}{2}.
\end{align}
Indeed, if $d\geq 3$, then $p_2\geq r+2$.
Since $w_2>1$, we have
\begin{align*}
    \max\{p_1+w_1,\ p_2+w_2\}
    \geq
    p_2+w_2
    >
    r+3
    >
    r+1+\frac{1+\sqrt5}{2}.
\end{align*}
It remains to consider $d=1$.  
Then $p_1=r$ and $p_2=r+1$.
If $\max\{r+w_1,\ r+1+w_2\}$ is minimal, then $r+w_1=r+1+w_2$.
By \eqref{eq:w_relation}, we get $w_2=\frac{1+\sqrt5}{2}$, $w_1=\frac{3+\sqrt5}{2}$.
Therefore \eqref{eq:odd_balance_bound} holds.
Hence, by
\eqref{eq:even_weighted_bound_correct},
\begin{align*}
    \lambda_2(T)
    \geq
    \sigma_t\left(r+1+\frac{1+\sqrt5}{2}\right).
\end{align*}
Now suppose equality holds.  
Then $d=1$, and moreover, 
\begin{align*}
    p_1=r,\qquad p_2=r+1,\qquad
    w_1=\frac{3+\sqrt5}{2},\qquad
    w_2=\frac{1+\sqrt5}{2}.
\end{align*}
Theorem~\ref{Klobučarštel theorem_weighted} implies that both nodal domains are weighted comets.  
It follows that $G_{U_1}-z_{uv}\cong C_{t,r}$ and $G_{U_2}-z_{uv}\cong C_{t,r+1}$.
The sign-changing edge must be attached at the two  edges with weight $w_1$ and $w_2$ of two comets.  
Therefore $T$ is obtained from one copy of $C_{t,r}$ and one copy of
$C_{t,r+1}$ by joining $v_1$ in $C_{t,r}$ and $v_1'$ in  $C_{t,r+1}$. 

Conversely, the proof is similar to Case~\ref{Case1}.
\end{subcase}
\end{case}

 The proof is complete.
\end{proof}

\section{The second largest adjacency eigenvalue}

\subsection{Dirichlet adjacency operator}
Let $G$ be a connected graph of order $n$ and $A(G)$ the adjacency matrix of $G$.
Let $\rho_1 \geq \rho_2 \geq \cdots \geq \rho_n$ be the eigenvalues of $A(G)$.

For a weighted graph with boundary $(G,B,1,w)$,
let $\mathcal{A}_\Omega$ be the Dirichlet (Schr\"odinger type) adjacency operator on $\Omega=V(G)\setminus B$, where 
\begin{align*}
    \mathcal{A}_\Omega f(x)=
   \sum_{\substack{y\in \Omega\\ y\sim x}} w_{xy}f(y)-\left(\sum_{\substack{y\in B\\ y\sim x}} w_{xy}
    \right)f(x)
\end{align*}
for any $x\in \Omega$.
When $\Omega=V(G)$ and $w\equiv1$, this operator coincides with the standard adjacency operator (matrix).
We call $\mu$ a Dirichlet adjacency eigenvalue and $f$ the corresponding eigenfunction.
Define a matrix  $(a_{xy})_{\Omega \times \Omega}$, where
\begin{align*}
    a_{xy} = 
    \begin{cases}
        -\sum\limits_{\substack{z\in B\\ z\sim x}} w_{xz}, & \textif x = y, \\
        w_{xy}, & \textif x \sim y, \\
        0, &  \textotherwise.
    \end{cases}
\end{align*}
 It is clear that this matrix corresponds to $\mathcal{A}_\Omega$.
 We also use this notation to denote the corresponding matrix.
We order the eigenvalues of $\mathcal{A}_\Omega$ as
        $\mu_{1}\geq \mu_{2}\geq \cdots \geq \mu_{|\Omega|}$.
In this paper, we only consider graphs whose interior edge weights are all equal to $1$.

For an induced subgraph $G[\Omega]$ of $G$, we write $A(G[\Omega])$ for the principal submatrix of $A(G)$ indexed by $\Omega$. 
Let $\mathcal{A}_\Omega =A(G[\Omega])-W_{\Omega}$, where $W_{\Omega}$ be the diagonal matrix defined by
\begin{align*}
        W_{xx}=
    \sum_{\substack{y\in B\\ y\sim x}} w_{xy},
    \qquad x\in \Omega.
\end{align*}
When $\Omega = V(G)$, $\mathcal{A}_G=A(G)$.

Next, we introduce  notation for geometric realization  and  the edge weights in the adjacency setting.
By default,  vertex weight is identically equal to 1.
\begin{definition}\label{def:A_U}
Let $G=(V,E,1,w)$ be a finite weighted graph, and $K(G)$ be the geometric realization of $G$. 
For any connected open subset $U$ of $|K(G)|$, we define the induced graph $G_U$ as the following weighted graph with Dirichlet boundary.
\begin{enumerate}
    \item
    $V(G_U)=\{x\in \overline U:\ x\in V(G)\text{ or }x\in \partial U\}$;

    \item
    $E(G_U)=\{xy:\ x\neq y\in V(G_U),\ x,y\text{ lie on the same edge of }
    K(G),\text{ and }(xy)\subset U\}$;

    \item
    $B_D(G_U)=\partial U$;

    \item
    For any $xy \in E(G_U)$,
    \begin{align*}
        w_{xy}^U
        =
        \begin{cases}
        w_{xy}=\dfrac{1}{l_{xy}},
        & xy\in E(G)\cap E(G_U),\\[0.8ex]
        \dfrac{1}{l_{xy}}-\dfrac{1}{l_{uv}},
        & xy\in E(G_U)\setminus E(G),
        \end{cases}
    \end{align*}
    For simplicity, we usually omit the superscript $U$ of $w_{xy}^U$ when no confusion may arise.

\end{enumerate}
\end{definition}

Let $A(G)f=\rho f$, and let $\widetilde f$ be the edgewise linear extension
of $f$ on $K(G)$.
Let $U$ be a nodal domain of $f$; that is, a connected
component of $K(G)\setminus \widetilde f^{-1}(0)$. 
Let $\Omega_U=U\cap V(G)$ be the interior of $G_U$.
For each edge $uv\in E(G)$ with $u\in\Omega_U$, $v\in V(G)\setminus\Omega_U$
and $f(v)=0$, let $\tilde v$ be a copy of $v$. 
We regard $\tilde v$ as a
boundary vertex of $G_U$ generated by $\partial U$, distinct from the original vertex $v$ of $G$.
Hence, $w^U_{u,\tilde v}=0$.
For graph $G_U$ with Dirichlet boundary, its interior is $\Omega_U$, and its Dirichlet boundary can be partitioned into two parts as
$ B_D(G_U)=B_U^0\cup \widetilde B_U^0,$
where
\begin{align*}
    B_U^0
    &=
    \{\tilde v:\ u\in\Omega_U,\ v\in V(G)\setminus\Omega_U,\ uv\in E(G),
    \ f(v)=0\},
    \\
    \widetilde B_U^0
    &=
    \{z_{uv}:\ u\in\Omega_U,\ v\in V(G)\setminus\Omega_U,\ uv\in E(G),
    \ f(v)\neq 0,\ \widetilde f(z_{uv})=0\}.
\end{align*}
Here $z_{uv}$ denotes the zero point of $\widetilde f$ on the edge $uv$ if $f(u)$ and $f(v)$ have distinct signs.


As a direct application of Lemma~\ref{lem:Metzler matrix_PF_theorem}, we list some properties of $\mu_1(G_U)$.
\begin{corollary}\label{cor:A-W_properties}  
Let $G$ be a connected graph, and let $f$ be an
adjacency eigenfunction of $G$ corresponding to the eigenvalue $\rho$, which
is not the spectral radius of $G$. 
For each nodal domain $U$ of $f$, the following statements hold.  
\begin{enumerate}        
\item[(a)] The eigenvalue $\mu_1(G_U)$ is a simple eigenvalue;        
\item[(b)] there exists a positive eigenfunction corresponding to $\mu_1(G_U)$; 
\item[(c)]
$\mu_1(G_U)\leq \rho_1(G[\Omega_U]).$
The equality holds if and only if all weights of boundary edges in $G_U$ vanish.
\end{enumerate}
\end{corollary}

\begin{proof}
Since $G[\Omega_U]$ is a subgraph of $G$, then $\mathcal{A}_{\Omega_U}\coloneqq A(G[\Omega_U])-W_{\Omega_U}$. 
For convenience, we simply write $\mathcal{A}_{\Omega_U} \coloneqq A-W$.
We know that the eigenvalues $\sigma_i$ of $\mathcal{A}_{\Omega_U}$ are equivalent to $\mu_i(G_U)$. 
Therefore, it suffices to prove the corresponding properties for $\sigma_i$.
Let $s\geq \max\limits_{1\leq i\leq |\Omega_U|} W_{ii}$.
Since $\mathcal{A}_{\Omega_U}=(A-W+sI)-sI$, $(a)$ and $(b)$ follow directly from Lemma~\ref{lem:Metzler matrix_PF_theorem}.
The inequality in $(c)$ follows directly from Lemma~\ref{thm:Weyl_inequality}.
It remains to discuss the equality case.
Let $\textbf{x}$ be a positive eigenfunction  such that $A\textbf{x}=\sigma_1(A)\textbf{x}, W\textbf{x}=\sigma_1(W)\textbf{x},$
and $(A-W)\textbf{x}=\sigma_1(A-W)\textbf{x}$.
If $\sigma_1(A-W)=\sigma_1(A)$, then $(A-W)\textbf{x}= \sigma_1(A)\textbf{x}-\sigma_1(W)\textbf{x} =\sigma_1(A)\textbf{x}$, which implies $W=0$.
It follows that the weights of all boundary edges of $G_U$ vanish.
\end{proof}

Our proofs are based on the following nodal domain theorem.
\begin{theorem}\label{thm:adjacency_nodal_D}
Let $G=(V,E,1, w)$ be a finite connected weighted graph, and let $f$ be an
adjacency eigenfunction of $G$ corresponding to the eigenvalue $\rho$, which
is not the spectral radius of $G$. Then, for each nodal domain $U$ of $f$,
$f|_{\Omega_U}$ is a Dirichlet adjacency eigenfunction of $G_U$ with vanishing data on $B_D(G_U)$. 
Moreover, $\mu_1(G_U)=\rho$.
\end{theorem}

\begin{proof}
Consider a nodal domain $U$ of $f$.
By the definition of nodal domain, $\Omega_U\neq\emptyset$ and $G[\Omega_U]$ is
connected.
Moreover, since $\rho$ is not the spectral radius of $G$, the eigenfunction
$f$ cannot be everywhere positive or everywhere negative on $V(G)$. 
Then $U\neq |K(G)|$. 
Since $G$ is connected, we have
$B_D(G_U)=\partial U\neq\emptyset$.

Take $u\in\Omega_U$. 
If $v\sim u$ and $v\notin\Omega_U$, then there is a
boundary vertex associated with the edge $uv$. 
If $f(v)=0$, this vertex is
$\tilde v$, and
\begin{align*}
    l_{u\tilde v}=l_{uv},\qquad
    w_{u\tilde v}=\frac{1}{l_{u\tilde v}}-\frac{1}{l_{uv}}=0.
\end{align*}
If $f(v)\neq 0$, this vertex is $z_{uv}$. 
Since $\widetilde f$ is linear on
$uv$ and $\widetilde f(z_{uv})=0$, we have $l_{u z_{uv}}=\frac{f(u)}{f(u)-f(v)}l_{uv}$, and 
\begin{align*}
    w_{u z_{uv}}
    =
    \frac{1}{l_{u z_{uv}}}-\frac{1}{l_{uv}}
    =
    \frac{f(u)-f(v)}{f(u)}w_{uv}-w_{uv}
    =
    -\frac{w_{uv}f(v)}{f(u)}.
\end{align*}

Therefore, for any $u\in\Omega_U$,
\begin{align*}
\rho f(u)
&=
\sum_{v\sim u}w_{uv}f(v) =
\sum_{v\sim u}w_{uv}\bigl(f(v)-f(u)\bigr)
+
\sum_{v\sim u}w_{uv}f(u) \\
&=
\sum_{\substack{v\sim u\\ v\in\Omega_U}}w_{uv}\bigl(f(v)-f(u)\bigr)
+
\sum_{\substack{z_{uv}\sim u\\ z_{uv}\in\widetilde B_U^0}}
w_{uv}\bigl(f(v)-f(u)\bigr)
+
\sum_{\substack{\tilde v\sim u\\ \tilde v\in B_U^0}}
w_{uv}\bigl(0-f(u)\bigr)
+
\sum_{v\sim u}w_{uv}f(u) \\
&=
\sum_{\substack{v\sim u\\ v\in\Omega_U}}w_{uv}\bigl(f(v)-f(u)\bigr)
+
\sum_{\substack{z_{uv}\sim u\\ z_{uv}\in\widetilde B_U^0}}
w_{uv}\frac{-f(v)}{f(u)}\bigl(-f(u)\bigr)
+
\sum_{\substack{\tilde v\sim u\\ \tilde v\in B_U^0}}
0\cdot\bigl(-f(u)\bigr)
+
\sum_{\substack{v\sim u\\ v\in\Omega_U}}w_{uv}f(u) \\
&=
\sum_{\substack{v\sim u\\ v\in\Omega_U}}w_{uv}\bigl(f(v)-f(u)\bigr)
+
\sum_{\substack{z_{uv}\sim u\\ z_{uv}\in\widetilde B_U^0}}
w_{u z_{uv}}\bigl(-f(u)\bigr)
+
\sum_{\substack{\tilde v\sim u\\ \tilde v\in B_U^0}}
w_{u\tilde v}\bigl(-f(u)\bigr)
+
\sum_{\substack{v\sim u\\ v\in\Omega_U}}w_{uv}f(u) \\
&=
(\mathcal A_{\Omega_U} f)(u).
\end{align*}
It follows that $\rho$ is an eigenvalue of $\mathcal A_{\Omega_U}$ and $f|_{\Omega_U}$ is a Dirichlet adjacency eigenfunction of $G_U$
with vanishing data on $B_D(G_U)$.
It remains to show that $\rho=\mu_1(G_U)$. Since $U$ is a nodal domain,
$f$ has a fixed sign on $\Omega_U$ and $f(u)\neq 0$ for all $u\in\Omega_U$.
Then, by Corollary~\ref{cor:A-W_properties}, we have
$\mu_1(G_U)=\rho$.

The proof is complete.
\end{proof}


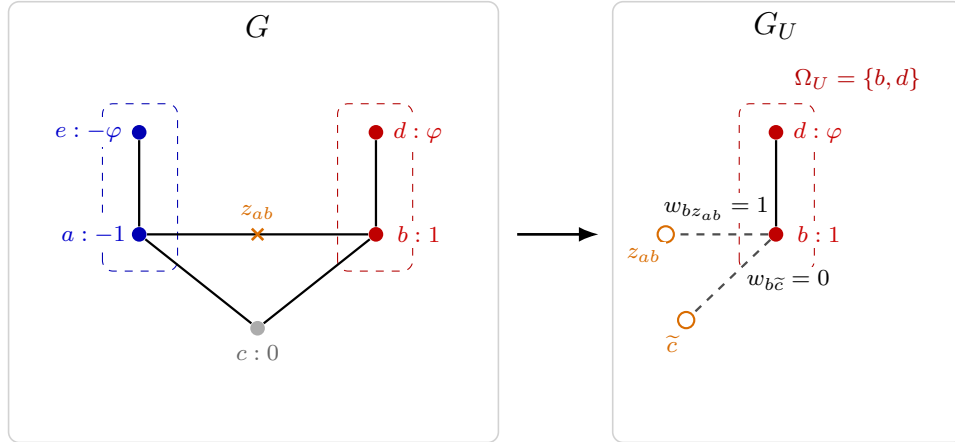
\begin{figure}[htbp]
\centering
\begin{tikzpicture}[
    scale=1.08,
    transform shape,
    >=Latex,
    edge/.style={line width=0.85pt},
    pospt/.style={circle,fill=red!75!black,inner sep=1.8pt},
    negpt/.style={circle,fill=blue!70!black,inner sep=1.8pt},
    zeropt/.style={circle,fill=gray!65,inner sep=1.8pt},
    bdpt/.style={circle,draw=orange!85!black,fill=white,inner sep=2pt,line width=0.8pt},
    lab/.style={
        font=\scriptsize,
        fill=white,
        fill opacity=1,
        text opacity=1,
        inner sep=1.8pt,
        outer sep=0pt
    },
    val/.style={
        font=\scriptsize,
        fill=white,
        fill opacity=1,
        text opacity=1,
        inner sep=1pt,
        outer sep=0pt
    },
    weightlab/.style={
        font=\scriptsize,
        text=black,
        fill=white,
        fill opacity=1,
        text opacity=1,
        inner sep=1.2pt,
        outer sep=0pt
    }
]


\node at (0,2.55) {$G$};

\draw[gray!35,rounded corners,line width=0.6pt]
    (-3.05,-2.55) rectangle (2.95,2.85);

\draw[blue!70!black,dashed,rounded corners]
    (-1.90,-0.45) rectangle (-0.98,1.60);

\draw[red!70!black,dashed,rounded corners]
    (0.98,-0.45) rectangle (1.90,1.60);

\node[negpt]  (a) at (-1.45,0) {};
\node[pospt]  (b) at (1.45,0) {};
\node[zeropt] (c) at (0,-1.15) {};
\node[negpt]  (e) at (-1.45,1.25) {};
\node[pospt]  (d) at (1.45,1.25) {};

\draw[edge] (a) -- (b);
\draw[edge] (a) -- (c);
\draw[edge] (b) -- (c);
\draw[edge] (b) -- (d);
\draw[edge] (a) -- (e);

\coordinate (zab) at (0,0);
\draw[orange!85!black,line width=1pt]
    ($(zab)+(-0.07,-0.07)$) -- ($(zab)+(0.07,0.07)$);
\draw[orange!85!black,line width=1pt]
    ($(zab)+(-0.07,0.07)$) -- ($(zab)+(0.07,-0.07)$);

\node[lab, text=orange!85!black, anchor=south] 
    at ($(zab)+(0,0.12)$) {$z_{ab}$};

\node[lab, text=blue!80!black, anchor=east]  
    at ($(a)+(-0.1,0)$) {$a:-1$};

\node[lab, text=red!80!black, anchor=west]   
    at ($(b)+(0.2,0)$) {$b:1$};

\node[lab, text=gray!80!black, anchor=north] 
    at ($(c)+(0,-0.14)$) {$c:0$};

\node[lab, text=blue!80!black, anchor=east]  
    at ($(e)+(-0.15,0)$) {$e:-\varphi$};

\node[lab, text=red!80!black, anchor=west] 
    at ($(d)+(0.15,0)$) {$d:\varphi$};

\draw[->,line width=1.0pt] (3.18,0) -- (4.18,0);


\node at (6.35,2.55) {$G_U$};

\draw[gray!35,rounded corners,line width=0.6pt]
    (4.35,-2.55) rectangle (8.65,2.85);

\draw[red!70!black,dashed,rounded corners]
    (5.90,-0.45) rectangle (6.82,1.60);

\node[pospt] (Ub)   at (6.35,0) {};
\node[pospt] (Ud)   at (6.35,1.25) {};
\node[bdpt]  (Uzab) at (5.00,0) {};
\node[bdpt]  (Uct)  at (5.25,-1.05) {};

\draw[edge] (Ub) -- (Ud);
\draw[edge,black!70,dashed] (Ub) -- (Uzab);
\draw[edge,black!70,dashed] (Ub) -- (Uct);

\node[weightlab] at (5.62,0.32) {$w_{bz_{ab}}=1$};
\node[weightlab] at (6.5,-0.56) {$w_{b\widetilde c}=0$};

\node[lab, text=red!80!black, anchor=west]       
    at ($(Ub)+(0.2,0)$) {$b:1$};

\node[lab, text=red!80!black, anchor=west]       
    at ($(Ud)+(0.15,0)$) {$d:\varphi$};

\node[lab, text=orange!85!black, anchor=east]    
    at ($(Uzab)+(0,-0.25)$) {$z_{ab}$};

\node[lab, text=orange!85!black, anchor=north east]    
    at ($(Uct)+(-0.05,-0.12)$) {$\widetilde c$};

\node[lab, text=red!70!black] 
    at (7.35,1.9) {$\Omega_U=\{b,d\}$};

\node[align=center, font=\small] at (2.80,-3.52) {
$A(G)f=\rho f$ on $G$
\qquad$\Longrightarrow$\qquad
$\mathcal A_{\Omega_U}(f|_{\Omega_U})
=\rho f|_{\Omega_U}$ on $\Omega_U$
};

\end{tikzpicture}
\caption{Construction of the subgraph $G_U$ from the positive nodal domain $U$. 
Here $\varphi=\frac{1+\sqrt5}{2}$.}
\label{fig:adjacency-nodal-domain-subgraph}
\end{figure}

\begin{example}
To better understand Theorem~\ref{thm:adjacency_nodal_D}, we consider the graph $G$ shown in
\cref{fig:adjacency-nodal-domain-subgraph}. \
The characteristic polynomial $P(G;x)$ of $G$ is given by $P(G;x) =x(x^2-x-3)(x^2+x-1)$.
Therefore, $\rho_2(G)=\frac{\sqrt5-1}{2}$.
Let $\varphi=\frac{1+\sqrt5}{2}$.
With respect to the vertex ordering $(a,b,c,d,e)$, an eigenvector corresponding to $\rho_2(G)$ can be chosen as $f=(-1,1,0,\varphi,-\varphi)^\top$.
Now consider the positive nodal domain $U$.
Then $B_D(G_U)=\{\widetilde c,z_{ab}\}$.
The corresponding boundary weights at $b$ are $w_{b\widetilde c}=\frac{f(c)}{f(b)}=0$ and $w_{bz_{ab}}=-\frac{f(a)}{f(b)}=1$.
Therefore, with respect to the ordering $b,d$ of $\Omega_U$, the  Dirichlet adjacency matrix is
\begin{align*}
    \mathcal A_{\Omega_U}
    =
    \begin{pmatrix}
    -1&1\\
    1&0
    \end{pmatrix}.
\end{align*}
Its characteristic polynomial is $\det(xI-\mathcal A_{\Omega_U})=x^2+x-1$.
Hence the eigenvalues of $\mathcal A_{\Omega_U}$ are $\frac{\sqrt5-1}{2}$ and $-\frac{1+\sqrt5}{2}$.
Moreover, $ f|_{\Omega_U} =
    \begin{pmatrix}
    1\\
    \varphi
    \end{pmatrix}$ satisfies $\mathcal A_{\Omega_U} \begin{pmatrix}
    1\\
    \varphi
    \end{pmatrix}
    =
    \frac{\sqrt5-1}{2}
    \begin{pmatrix}
    1\\
    \varphi
    \end{pmatrix}$.
Therefore, $\mathcal A_{\Omega_U}
    \bigl( f|_{\Omega_U}\bigr)
    =
    \rho_2(G) f|_{\Omega_U}$.
\end{example}

\subsection{Proofs of Theorems~\ref{lem:adjacency_nodal_domain} and ~\ref{thm:connected_even_bridge}}

\begin{lemma}[\cite{LiFeng1979}]\label{lem:proper_subgraph_spectral_radius}
Let $G$ be a connected graph, and let $H$ be a proper subgraph of $G$.
Then $\rho_1(H)<\rho_1(G)$.
\end{lemma}

Now we are ready to prove our main theorems.





\begin{proof}[\bf{Proof of Theorem~ \ref{lem:adjacency_nodal_domain}}]

We consider the connected case first. 
Let $f$ be an eigenfunction corresponding
to $\rho_2(G)$. 
Choose a nodal domain $U$ of $f$ such that $|\Omega_U|$ is
minimum. 
Since $\rho_2(G)$ is not the spectral radius, $f$ changes sign.
It follows that $f$ has at least two nodal domains, and $|\Omega_U|\leq \lfloor n/2\rfloor$. 
By Theorem~\ref{thm:adjacency_nodal_D}, Corollary~\ref{cor:A-W_properties}, and the fact that $\alpha(s)$ is strictly increasing with respect to $s$, we have
\begin{equation}\label{eq:nodal_chain}
\begin{aligned}
    \rho_2(G) = \mu_1(G_U)
    \leq \rho_1(G[\Omega_U]) \leq \alpha(\lfloor n/2\rfloor).
\end{aligned}
\end{equation}
Now we characterize the condition that the equality holds.
The equality in the first inequality of~\eqref{eq:nodal_chain}, by Corollary \ref{cor:A-W_properties}, implies that all boundary weights of $G_U$ vanish. 
Therefore, every boundary vertex is a copy of some zero point of $f$, which also implies that $f$ has at least one zero point.
The equality in the second inequality of~\eqref{eq:nodal_chain} implies that $\Omega_U$ is of the form $G_1^*(\lfloor n/2\rfloor)$.
Consider another nodal domain $U'\neq U$.
By Theorem \ref{thm:adjacency_nodal_D}, Corollary \ref{cor:A-W_properties}, we have $\mu_1(G_{U'}) = \rho_2(G) = \alpha(\lfloor n/2\rfloor)\leq \rho_1(G[\Omega_{U'}])$.
Since $\alpha(s)$ is strictly increasing with respect to $s$, $|\Omega_{U'}| \geq \lfloor n/2\rfloor$. 
Since the existence of zero point, $|\Omega_U|+|\Omega_{U'}|+1\leq n$.
It follows that $|\Omega_{U'}| = \lfloor n/2\rfloor$ and $n$ is odd.
Therefore, the graph is of the form $G_1^*(\lfloor n/2\rfloor) \bullet K_1 \bullet G_2^*(\lfloor n/2\rfloor) \in \mathcal{C}$.

Conversely, suppose that $n$ is odd and that $G$ is obtained by joining a vertex to two such extremal graphs of order $\lfloor n/2\rfloor$. 
Since the extremal graph need not be unique, write these two graphs as $H_1$ and $H_2$. 
The adjacency matrix
of $G$ has the form
\begin{align*}
    A(G)=
    \begin{pmatrix}
        A({H_1}) & 0 & {\bf b_1}  \\
        0 & A({H_2}) & {\bf b_2}  \\
        {\bf b_1^{\top}} & {\bf b_2^{\top}} & 0
    \end{pmatrix}.
\end{align*}
Let $\mathbf{u_i}$ be a Perron eigenvector of $A({H_i})$ corresponding to $\alpha(\lfloor n/2\rfloor)$ for $i=1,2$. Choose constants $c_1,c_2 \neq 0$, such that
$c_1 {\bf b_1^T} {\bf u_1}+c_2 {\bf b_2^T} {\bf u_2}=0$.  
Then
\begin{align*}
    A(G)
    \begin{pmatrix}
        c_1 {\bf u_1}\\
        c_2 {\bf u_2}\\
        0
    \end{pmatrix}
    =
    \begin{pmatrix}
        \alpha(\lfloor n/2\rfloor) c_1 {\bf u_1}\\
        \alpha(\lfloor n/2\rfloor) c_2 {\bf u_2}\\
        c_1 {\bf b_1^{\top}} {\bf u_1} + c_2 {\bf b_2^{\top}} {\bf u_2}
    \end{pmatrix}
    =
    \alpha(\lfloor n/2\rfloor)
    \begin{pmatrix}
        c_1 {\bf u_1}\\
        c_2 {\bf u_2}\\
        0
    \end{pmatrix}.
\end{align*}
Thus $\alpha(\lfloor n/2\rfloor)$ is an eigenvalue of $G$.
On the other hand, $A({H_1})\oplus A({H_2})$ is a principal submatrix of $A(G)$, and $\rho_1(A({H_1})\oplus A({H_2}))=\alpha(\lfloor n/2\rfloor)$. 
By Lemma~\ref{cor:spectral_radius_monotonicity} ,
$\rho_2(G)\leq \alpha(\lfloor n/2\rfloor)$. 
Moreover, since $G$ is connected and properly contains both $H_1$ and $H_2$, 
we have $\rho_1(G)>\alpha(\lfloor n/2\rfloor)$ by Lemma~\ref{lem:proper_subgraph_spectral_radius}. 
Since $\alpha(\lfloor n/2\rfloor)$ is an eigenvalue of $G$, we conclude that $\rho_2(G)=\alpha(\lfloor n/2\rfloor)$.
This proves the equality case in the connected setting.

We now remove the connectedness assumption.
Suppose that $G$ is disconnected, and let $H_1,\ldots,H_r$ be its connected
components. 
The eigenvalue $\rho_2(G)$ is either $\rho_2(H_k)$ for some component $H_k$, or the second largest value among $\rho_1(H_1),\cdots,\rho_1(H_r)$.
Combining this with the preceding discussion for the connected case, we obtain that, for each $i$,
\begin{equation}\label{eq:nodal_chain_2}
\begin{aligned}
    \rho_2(H_i)
    \leq \alpha(\lfloor |V(H_i)|/2\rfloor)
\leq  \alpha(\lfloor n/2\rfloor).
\end{aligned}
\end{equation}
Moreover, at most one component can have spectral radius larger than
$\alpha(\lfloor n/2\rfloor)$. Otherwise two components would both
have order larger than $\lfloor n/2\rfloor$, which is impossible for a graph of order $n$. 
Therefore $\rho_2(G) \leq \alpha(\lfloor n/2\rfloor)$.

If $n$ is even and the equality holds, then $\rho_2(G)$ cannot come from some $\rho_2(H_k)$.
Otherwise, since $G$ is disconnected, then $|V(H_k)|<n$, and we have $\lfloor |V(H_k)|/2\rfloor<\lfloor n/2\rfloor$, which implies that $\rho_2(H_k)\leq \alpha(\lfloor |V(H_k)|/2\rfloor)< \alpha( n/2)$.
Thus equality comes from a
component $H_{i_1}$ with $\rho_1(H_{i_1})=\alpha( n/2 )$.
Since $\rho_2(G)=\alpha( n/2)$, we have
$\rho_1(G)\geq \alpha( n/2 )$, which forces the spectral radius
of another connected component $H_{i_2}$ to be at least
$\alpha( n/2 )$.
Since $\alpha(s)$ is strictly increasing with respect to $s$, components $H_{i_1}$ and $H_{i_2}$ have order $ n/2$. 
Therefore, the extremal graph is of the form $G_1^*(n/2)\cup G_2^*(n/2)$.
The converse follows immediately from the spectrum of a disjoint union.

If $n$ is odd and the equality holds, then $\rho_2(G)$ cannot come from some $\rho_2(H_k)$.
Indeed, if $|V(H_k)|\leq 2\lfloor n/2\rfloor-1$, then the second inequality in inequality~\eqref{eq:nodal_chain_2} is strict.
If $|V(H_k)|= 2\lfloor n/2\rfloor$, then $\rho_2(H_k)<\alpha(\lfloor n/2\rfloor)$, as can be seen from the discussion of the connected case above.
Hence $\rho_2(H_i)<\alpha(\lfloor n/2\rfloor)$ for each component $H_i$.
Then equality must come from a
component $H_{j}$ with $\rho_1(H_{j})=\alpha(\lfloor n/2 \rfloor)$.
Denote the union of the remaining components by $G_2$, and we have $|V(G_2)|=\lceil n/2\rceil$.
We claim that $\rho_1(G_2)\geq \alpha(\lfloor n/2\rfloor)$. 
Indeed, if $\rho_1(G_2)<\alpha(\lfloor n/2\rfloor)$, then $\rho_2(G)=\max\{\rho_2(H_j),\rho_1(G_2)\}< \alpha(\lfloor n/2\rfloor)$, a contradiction.
Therefore, the graph is of the form $G_1^*(\lfloor n/2\rfloor)\cup G_2(\lceil n/2\rceil)$, where  $\rho_1(G_2(\lceil n/2\rceil))\geq \rho_1(G_1^*(\lfloor n/2\rfloor))$ and $\rho_1(G_1^*(\lfloor n/2\rfloor))=\alpha(\lfloor n/2\rfloor)$.
The converse is direct based on the monotonicity of $\alpha(s)$.

Combining the connected and disconnected cases gives the stated equality characterization.
The proof is complete.

\end{proof}


\begin{proof}[\bf{{Proof of Theorem \ref{thm:connected_even_bridge}}}]
We first show that $x_{u^*}=\min_{v\in V(G^*(t))}x_v$.
Suppose, to the contrary, that there exists a vertex $v\in V(G^*(\frac{n}{2}))$ such that $x_v<x_{u^*}$.
The pair $(G^*(\frac{n}{2}),v)$ is admissible in the definition of $\beta(\frac{n}{2})$.
Then
\begin{align*}
    \beta(\frac{n}{2})
    &\geq \sigma_1(A({G^*(\frac{n}{2})})-I_{vv}) \geq
    \frac{\mathbf{x}^{\top}(A({G^*(\frac{n}{2})})-I_{vv})\mathbf{x}}{\mathbf{x}^{\top}\mathbf{x}}=
    \frac{\mathbf{x}^{\top}(A({G^*(\frac{n}{2})})-I_{u^*u^*})\mathbf{x}}{\mathbf{x}^{\top}\mathbf{x}}
    +
    \frac{x_{u^*}^2-x_v^2}{\mathbf{x}^{\top}\mathbf{x}}  \\
    &=
    \beta(\frac{n}{2})+\frac{x_{u^*}^2-x_v^2}{\mathbf{x}^{\top}\mathbf{x}}>\beta(\frac{n}{2}),
\end{align*}
a contradiction. 
Therefore, $x_{u^*}=\min_{v\in V(G^*(\frac{n}{2}))}x_v$.

Next, we prove that $\rho_2(G)\leq \beta(\frac{n}{2})$.
Let $G\in\mathcal C$ be a connected graph of order $n$, and let $f$ be a $\rho_2(G)$-eigenfunction. 
Hence $f$ changes sign. 
Choose a nodal domain $U$ of $f$ such that $|\Omega_U|$ is minimum.
It follows that $|\Omega_U|\leq n/2$.
If $|\Omega_U|<n/2$, then, by  Corollary~\ref{cor:A-W_properties}, Theorem~\ref{thm:adjacency_nodal_D}, and $\beta(n/2)>\alpha(s)$ for $s<n/2$, we obtain
\begin{equation}
\begin{aligned}\label{eq:U_n_even_1}
    \rho_2(G) = \mu_1(G_U) \leq\rho_1(G[\Omega_U]) \leq
    \alpha(|\Omega_U|) < \beta(n/2).
\end{aligned}
\end{equation}
If $|\Omega_U|=n/2$, then $f$ has exactly two nodal domains.
Denote the other one by $U'$.
We have $|\Omega_{U'}|=n/2$.
Then, $\{v\in V(G): f(v)=0\}=\emptyset$.
Therefore, every boundary vertex of $G_U$ and $G_{U'}$ is produced by cutting a sign-changing edge between $\Omega_U$ and $\Omega_{U'}$.


Let $\{a\in \Omega_U: N_{\Omega_{U'}}(a)\neq \emptyset\}=\{a_1,\ldots,a_p\}$. 
For each $i\in\{1,\ldots,p\}$, write $N_{\Omega_{U'}}(a_i)=\{b_{i1},\ldots,b_{iq_i}\}$. 
Then the set of edges of $G$ joining the two nodal domains $U$ and $U'$ is 
\begin{align*}
    E(G[\Omega_U],G[\Omega_{U'}])=\{a_i b_{ij}\in E(G): 1\leq i\leq p,\ 1\leq j\leq q_i\}.
\end{align*}
Since $G$ is connected, $E(G[\Omega_U],G[\Omega_{U'}])\neq \emptyset$, and we have $\sum_{i=1}^{p}q_i\geq 1$.


Replacing $f$ by $-f$ if necessary, we may assume that $f(a_i)>0>f(b_{ij})$ for all $1\leq i\leq p$ and $1\leq j\leq q_i$. 
For each pair $(i,j)$ with $1\leq i\leq p$ and $1\leq j\leq q_i$, let $z_{ij}$ be the unique zero point of the edgewise linear extension $\widetilde f$ on the edge $a_i b_{ij}$. 
Then $z_{ij}$ is a Dirichlet boundary vertex of both $G_U$ and $G_{U'}$.
Since the original edge $a_i b_{ij}$ has weight $1$, its length is $1$. 
By the linearity of $\widetilde f$ on $a_i b_{ij}$, we have $l_{a_i z_{ij}}=\frac{f(a_i)}{f(a_i)-f(b_{ij})}$ and $l_{b_{ij}z_{ij}}=\frac{-f(b_{ij})}{f(a_i)-f(b_{ij})}$.
It follows by  Definition~\ref{def:A_U} that
$$w^{U}_{a_i z_{ij}}=\frac{1}{l_{a_i z_{ij}}}-1=\frac{f(a_i)-f(b_{ij})}{f(a_i)}-1=-\frac{f(b_{ij})}{f(a_i)}>0$$
and 
$$w^{U'}_{b_{ij} z_{ij}}=\frac{1}{l_{b_{ij} z_{ij}}}-1=\frac{f(a_i)-f(b_{ij})}{-f(b_{ij})}-1=-\frac{f(a_i)}{f(b_{ij})}>0.$$
Therefore, 
\begin{align*}
    \sum_{v\in \Omega_U}(W_{\Omega_U})_{vv}
    = \sum_{i=1}^{p}\sum_{j=1}^{q_i}\left(-\frac{f(b_{ij})}{f(a_i)}\right), \quad
    \sum_{v\in \Omega_{U'}}(W_{\Omega_{U'}})_{vv}
    = \sum_{i=1}^{p}\sum_{j=1}^{q_i}\left(-\frac{f(b_{ij})}{f(a_i)}\right)^{-1}.
\end{align*}
By the Cauchy--Schwarz inequality,
\begin{equation}
\begin{aligned}\label{theta}
    \left(\sum_{v\in \Omega_U}(W_{\Omega_U})_{vv}\right)
    \left(\sum_{v\in \Omega_{U'}}(W_{\Omega_{U'}})_{vv}\right) 
    & =\left(\sum_{i=1}^{p}\sum_{j=1}^{q_i} \left(-\frac{f(b_{ij})}{f(a_i)}\right)\right)\left(\sum_{i=1}^{p}\sum_{j=1}^{q_i}\left(-\frac{f(b_{ij} )}{f(a_i)}\right)^{-1}\right) \\
    &\geq \left(\sum_{i=1}^{p}\sum_{j=1}^{q_i}1 \right )^2
    =(\sum_{i=1}^{p}q_i)^2
    \geq 1.
\end{aligned}
\end{equation}
Consequently, 
\begin{align*}
    \max\left\{
    \sum_{v\in \Omega_U}(W_{\Omega_U})_{vv},
    \sum_{v\in \Omega_{U'}}(W_{\Omega_{U'}})_{vv}
    \right\}\geq 1.
\end{align*}
WLOG, let $\sum_{v\in \Omega_U}(W_{\Omega_U})_{vv}\geq \sum_{v\in \Omega_{U'}}(W_{\Omega_{U'}})_{vv}$.
Then $\sum_{v\in \Omega_U}(W_{\Omega_U})_{vv} \geq 1$.
Let $\theta_U=\sum_{v\in \Omega_U}(W_{\Omega_U})_{vv}$.
Let $\bf y$ be the positive unit eigenvector corresponding to $\sigma_1(A(G[\Omega_U])-W_{\Omega_U})$, and choose $u\in \Omega_U$ such that
$y_u=\min_{v\in \Omega_U}y_v$.
Then, we have
\begin{equation}
\begin{aligned}\label{eq:n_even_D_A_1}
    \sigma_1(A(G[\Omega_U])- I_{uu})
    & \geq \mathbf{y^{\top}}(A(G[\Omega_U])-I_{uu})\mathbf{y}=\mathbf{y^{\top}}(A(G[\Omega_U])-W_{\Omega_U})\mathbf{y}+\sum_{v\in \Omega_U}(W_{\Omega_U})_{vv}y_v^2-y_u^2 \\
   &\geq \mathbf{y^{\top}}(A(G[\Omega_U])-W_{\Omega_U})\mathbf{y}+\sum_{v\in \Omega_U}(W_{\Omega_U})_{vv}y_v^2-\theta_Uy_u^2 \\
   &=\sigma_1(A(G[\Omega_U])-W_{\Omega_U})
    +\sum_{v\in \Omega_U}(W_{\Omega_U})_{vv}(y_v^2-y_u^2) \\
    &\geq \sigma_1(A(G[\Omega_U])-W_{\Omega_U}).
\end{aligned}
\end{equation}
Since $\sigma_1(A(G[\Omega_U])-W_{\Omega_U})=\mu_1(G_U)$, by Corollary~\ref{cor:A-W_properties}, Theorem~\ref{thm:adjacency_nodal_D}, and  \eqref{eq:n_even_D_A_1}, we have 
\begin{align*}
    \rho_2(G)
    = \mu_1(G_U) =\sigma_1(A(G[\Omega_U])-W_{\Omega_U})
    \leq \max_{v\in \Omega_U}\sigma_1(A(G[\Omega_U])-I_{vv})
    \leq
    \beta(n/2).
\end{align*}
Combining the two cases gives $\rho_2(G)\leq \beta(n/2)$.

It remains to show that equality holds if and only if $G\cong 2G^*(n/2)+u_1u_2\in \mathcal C$.
Suppose that equality $\rho_2(G)=\beta(n/2)$ holds for some connected graph $G\in\mathcal C$ of
order $n$.
By~\eqref{eq:U_n_even_1}, the case $|\Omega_U|<n/2$ implies $\rho_2(G)<\beta(n/2)$.

If $|\Omega_U|=n/2$, then, as above, the nodal domains determined by $f$ are exactly two, denoted by $U$ and $U'$.
The equality  also implies 
$\max_{v\in \Omega_U}\sigma_1(A(G[\Omega_U])-I_{vv})=\beta(n/2)$.
Moreover, \eqref{eq:n_even_D_A_1} forces that ${\mathbf y}$ is an eigenvector corresponding to $\sigma_1(A(G[\Omega_U])-I_{vv})$ for any $v\in \Omega_U$, $\sum_{v\in \Omega_U}(W_{\Omega_U})_{vv}(y_v^2-y_u^2) =0$, and $(\theta_U-1)y_u^2=0$.
Hence $\theta_U=1$.
Let $\theta_{U'}=\sum_{v\in \Omega_{U'}}(W_{\Omega_{U'}})_{vv}$.
Since $\theta_U=1$ and $\theta_U=\max\{\theta_U,\theta_{U'}\}$, we have $\theta_{U'}\leq 1$.
It follows by~\eqref{theta} that $\theta_{U'}= 1$, which implies that there is exactly one edge in $E(\Omega_U,\Omega_{U'})$ because $(\sum_{i=1}^{p}q_i)^2 = 1$. 
Hence $U$ and $U'$ are linked by exactly one edge, denoted by $u_1 u_2$, where $u_1\in\Omega_U$ and $u_2\in\Omega_{U'}$, and $W_{\Omega_U}=I_{u_1u_1}$, $W_{\Omega_{U'}}=I_{u_2u_2}$.
Combining with $\sum_{v\in \Omega_U}(W_{\Omega_U})_{vv}(y_v^2-y_u^2) =0$, we know that $u_1$ is a vertex whose component of $y$ has the smallest value and $y$ is also a positive unit eigenvector corresponding to $\sigma_1(A(G[\Omega_U])-I_{u_1u_1})$.
The same argument applied to $U'$ shows that $u_2$ is a minimum-component vertex of the positive eigenvector corresponding to
$\sigma_1(A(G[\Omega_{U'}])-I_{u_2u_2})$.
Since  $\rho_2(G)=\mu_1(G_U)=\mu_1(G_{U'})=\beta(n/2)$,
we obtain $\sigma_1(A(G[\Omega_U])-I_{u_1u_1})=\sigma_1(A(G[\Omega_{U'}])-I_{u_2u_2})=\beta(n/2)$.
By the uniqueness assumption, 
\begin{align*}
    (G[\Omega_U],u_1)\cong (G^*(n/2),u^*),
    \qquad
    (G[\Omega_{U'}],u_2)\cong (G^*(n/2),u^*).
\end{align*}
Therefore $G\cong 2G^*(n/2)+u_1u_2$.

Conversely, suppose that $G\cong 2G^*(n/2)+u_1u_2$,
where $u_i$ is the copy of $u^*$ in the $i$-th copy of $G^*(n/2)$ for $i=1,2$.
Let $H=G^*(n/2)$, and let $\mathbf{x}>0$ be the positive unit eigenvector corresponding to $\sigma_1(A(H)-I_{u^*u^*})=\beta(n/2)$.
Thus $(A(H)-I_{u^*u^*})\mathbf{x}=\beta(n/2)\mathbf{x}$.
Define a vector $\boldsymbol{\xi}$ on $G$ by
\begin{align*}
    \xi_v=
    \begin{cases}
        x_v, & v\in V(H_1),\\
        -x_v, & v\in V(H_2),
    \end{cases}
\end{align*}
where $H_1$ and $H_2$ are the two copies of $H$.
For every vertex $v\neq u_1,u_2$, $(A(G)\boldsymbol{\xi})_v= (A(H)\mathbf{x})_v=\beta(n/2)\xi_v$.
For $u_1$, 
\begin{align*}
    (A(G)\boldsymbol{\xi})_{u_1}=(A(H)\mathbf{x})_{u^*}-x_{u^*}=((A(H)-I_{u^*u^*})\mathbf{x})_{u^*}=\beta(n/2)x_{u^*} =
    \beta(n/2)\xi_{u_1}.
\end{align*}
Similarly, we have $(A(G)\boldsymbol{\xi})_{u_2}=\beta(n/2)\xi_{u_2}$.
Therefore, $A(G)\boldsymbol{\xi}=\beta(n/2)\boldsymbol{\xi}$.
Since $\boldsymbol{\xi}$ changes sign and $G$ is connected, $\beta(n/2)$ is not the
spectral radius of $G$. 
Therefore, $\rho_2(G)\geq \beta(n/2)$.
On the other hand, by the upper bound already proved, we have $\rho_2(G)\leq \beta(n/2)$.
Thus $\rho_2(G)=\beta(n/2)$.
This proves the sufficiency.

In conclusion, the equality holds if and only if $G\cong 2G^*(n/2)+u_1u_2\in \mathcal{C}$.
The proof is complete.
\end{proof}

\begin{remark}
    The uniqueness assumption on the extremal pair may be removable.
In this situation, to obtain the extremal graphs one may need to compare directly the candidate extremal graphs obtained at the final step.
\end{remark}

\section{Applications}
In this section, we illustrate our main results through a number of specific applications.
First, we point out that there are many graph classes that satisfy the monotonicity conditions in our main results.
We recall that 
$\alpha(s)=\max\{\rho_1(H): H\in\mathcal C,\ |V(H)|=s\}$ and  $\beta(t):= \max\{\sigma_1(A(G)-I_{v v}):G\in\mathcal C,\ |V(G)|=t,\ v\in V(G)\}$.




We introduce a large class of graphs which satisfies our monotonicity conditions.
\begin{definition}[Pendant-extension property]\label{def:weak_pendant_extension_property}
Let $\mathcal C$ be a graph class. 
We say that $\mathcal C$ satisfies the {\it pendant-extension property} if, for every graph $G\in\mathcal C$, attaching a pendant edge to any vertex $u\in V(G)$ produces a graph still belonging to $\mathcal C$.
\end{definition}

\begin{lemma}\label{lem:alpha_beta_weak_pendant_monotonicity}
Let $\mathcal C$ be a graph class satisfying the pendant-extension property. 
Then, for any positive integers $s<t$, we have
\[
    \alpha(t)>\beta(t)>\alpha(s)>\beta(s).
\]
\end{lemma}

\begin{proof}
First we show that $\alpha(r+1)>\alpha(r)$.
Let $H\in\mathcal C$ be a graph of order $r$ such that $\rho_1(H)=\alpha(r)$.
Choose a connected component $H_0$ of $H$ such that
$\rho_1(H_0)=\rho_1(H)=\alpha(r)$.
By the pendant-extension property,  attaching a new pendant vertex $z$ to a vertex
$u\in V(H_0)$ produces a graph $H'$ still belongs to $\mathcal C$.
Let $H'_0$ be the connected component of $H'$ containing $H_0\cup\{z\}$.
Then $H_0$ is a proper subgraph of the connected graph $H'_0$. 
By  Lemma~\ref{lem:proper_subgraph_spectral_radius},  $\rho_1(H_0')>\rho_1(H_0)=\alpha(r)$.
Therefore $\alpha(r+1)\geq \rho_1(H')>\alpha(r)$.

Next, we prove $\beta(r+1)>\alpha(r)$.
Let ${\mathbf x}>{\mathbf 0}$ be the unit Perron vector of $A({H_0})$ corresponding to $\rho_1(H_0)=\alpha(r)$.
For $\varepsilon>0$, define a vector ${\mathbf y}$ on $V(H_0)\cup\{z\}$ by
$ {\mathbf y}|_{V(H_0)}= {\mathbf x}$ and $y_z=\varepsilon$.
 Hence
$$  {\mathbf y}^{\top}(A({H'_0})-I_{zz}){\mathbf y}
    =
     {\mathbf x}^{\top}A({H_0}) {\mathbf x}+2\varepsilon x_u-\varepsilon^2
    =
    \alpha(r)+2\varepsilon x_u-\varepsilon^2.$$
Since $ {\mathbf y}^{\top} {\mathbf y}=1+\varepsilon^2$, we have
\[
    \frac{ {\mathbf y}^{\top}(A({H'_0})-I_{zz}) {\mathbf y}}{ {\mathbf y}^{\top} {\mathbf y}}
    =
    \frac{\alpha(r)+2\varepsilon x_u-\varepsilon^2}{1+\varepsilon^2}.
\]
Choosing $0<\varepsilon<2x_u/(\alpha(r)+1)$, we have
$$  \frac{ {\mathbf y}^{\top}(A({H'_0})-I_{zz}) {\mathbf y}}{ {\mathbf y}^{\top} {\mathbf y}}>\alpha(r).
$$
Thus $ \sigma_1(A({H'_0})-I_{zz})>\alpha(r).$
It follows that $ \beta(r+1)\geq \sigma_1(A({H'})-I_{zz})\geq \sigma_1(A({H'_0})-I_{zz})
    >\alpha(r).$

It remains to prove that $\alpha(r)>\beta(r)$. 
Let $G\in\mathcal C$ be a graph of order $r$ consisting of connected components $G_1,\ldots,G_q$, and choose $u\in V(G_1)$. 
Then 
$$A(G)-I_{uu}=(A({G_1})-I_{uu})\oplus A({G_2})\oplus\cdots\oplus A({G_q}).$$
If $q\geq 2$, then $\rho_1(G_i)\leq \alpha(|V(G_i)|)< \alpha(r).$
Since $G_1$ is connected, by Lemma~\ref{lem:Metzler matrix_PF_theorem}, the largest eigenvalue of $A({G_1})-I_{uu}$ has a positive unit eigenvector, denoted by ${\mathbf w}$. 
Hence 
\begin{equation}
\begin{aligned}\label{component-upper-bound}
\sigma_1(A({G_1})-I_{uu})={\mathbf w}^{\top}A({G_1}){\mathbf w} - w_u^2< {\mathbf w}^{\top}A({G_1}) {\mathbf w}\leq\rho_1(G_1) \leq\alpha(|V(G_1)|)\leq \alpha(r).
\end{aligned}
\end{equation}
Therefore every block of $A(G)-I_{uu}$ has largest eigenvalue strictly smaller than $\alpha(r)$, and hence $\sigma_1(A(G)-I_{uu}) < \alpha(r).$
Taking the maximum over all  pairs $(G,u)$ gives
   $ \beta(r)<\alpha(r).$
If $q=1$, then \eqref{component-upper-bound} still holds.
It follows that $ \beta(r)<\alpha(r)$ for any $q\geq 1$.
Combining $\alpha(r)>\beta(r)$ and $\beta(r+1)>\alpha(r)$, we obtain the desired result.
\end{proof}

\subsection{Some known results}
In this subsection, we show that several known results follow directly from our main theorems.
\begin{corollary}[{\cite[Theorem~2]{Hong1988adjacency_disconnected_even}}]
Let $G$ be a graph with $n$ vertices. 
If $n$ is even, then 
\begin{align*}    
\rho_2(G) \leq \frac{n-2}{2}. 
\end{align*}
Equality holds if and only if $G \cong K_{\frac{n}{2}} \cup K_{\frac{n}{2}}$.
\end{corollary}
\begin{proof}    
Since the complete graph uniquely attains the maximum spectral radius among all graphs of a given order, and $\rho_1(K_s)=s-1< \rho_1(K_t)=t-1$ for $s< t$, the corollary follows immediately from Theorem \ref{lem:adjacency_nodal_domain}.
\end{proof}

\begin{corollary}[{\cite[Theorem~2.9]{Zhai2012adjacency_connected_graph}}]
Let $G$ be a connected graph with $n$ vertices.
\begin{enumerate}[(1)]
\item If $n$ is odd, then $\rho_2(G) \leq \frac{n-3}{2}$.
Equality holds if and only if $G \cong 2K_{\frac{n-1}{2}} \bullet K_1$.
\item If $n$ is even, then $\rho_2(G) \leq \rho_2(2K_{\frac{n}{2}} + e)$.
Equality holds if and only if $G \cong 2K_{\frac{n}{2}} + e$.
\end{enumerate}
\end{corollary}

\begin{proof}
Let $\mathcal C$ be the class of connected graphs, which satisfies pendant-extension property.
By Lemma~\ref{lem:alpha_beta_weak_pendant_monotonicity}, for any positive integers $s<t$, we have $\alpha(t)>\beta(t)>\alpha(s)>\beta(s)$.
If $n$ is odd, then the corollary follows immediately from Theorem~\ref{lem:adjacency_nodal_domain}. 
If $n$ is even, then the corollary follows immediately from Theorem~\ref{thm:connected_even_bridge}.
\end{proof}

\begin{corollary}[\cite{Powers1988adjacency_bipartite_graph}]
If $G$ is a bipartite graph with $n$ vertices and $k = \lfloor \frac{n}{4}\rfloor$, then
\begin{align*}    
\rho_2(G) \leq    
\begin{cases}        
k, & \text{if } n = 4k \text{ or } 4k + 1, \\        
\sqrt{k(k+1)}, & \text{if } n = 4k + 2 \text{ or } 4k + 3.    
\end{cases}
\end{align*}
For $n = 4k + r$, where $r \in \{0,1,2,3\}$,
$K_{k, k+\lfloor r/2 \rfloor} \cup K_{k, k+\lceil r/2 \rceil}$ is an extremal graph.
\end{corollary}

\begin{proof}    
Since the complete balanced bipartite graph uniquely attains the maximum spectral radius among all graphs of a given order, and $\rho_1(K_{\lfloor \frac{s}{2}\rfloor,\lceil\frac{s}{2}\rceil} )= \sqrt{\lfloor \frac{s}{2}\rfloor\lceil\frac{s}{2}\rceil}< \rho_1(K_{\lfloor \frac{t}{2}\rfloor,\lceil\frac{t}{2}\rceil} )= \sqrt{\lfloor \frac{t}{2}\rfloor\lceil\frac{t}{2}\rceil}$ for $s< t$, the corollary follows immediately from Theorem \ref{lem:adjacency_nodal_domain}.
\end{proof}

Let $S^k_{a,b}$ denote the tree obtained from two disjoint stars $K_{1,a}$ and $K_{1,b}$ by joining a path of length $k-1$ between their centers. 
\begin{corollary}[{\cite[Theorem~4.7]{Neumaier1982adjacency_odd_tree}, \cite{Shao1995adjacency_even_tree}}]
Let $T$ be a tree with $n$ vertices. 
Then the following statements hold.
\begin{enumerate}[(\roman*)]
\item If $n$ is odd, then $\rho_2(T) \leq \sqrt{\frac{n-3}{2}}$.
Equality holds if and only if $T$ is isomorphic to one of $S^3_{\frac{n-3}{2}, \frac{n-3}{2}}$, $S^4_{\frac{n-5}{2}, \frac{n-3}{2}}$, $S^5_{\frac{n-5}{2}, \frac{n-5}{2}}$.

\item If $n$ is even, then $\rho_2(T) \leq \rho_2(S^4_{\frac{n-4}{2},\frac{n-4}{2}})$.
Equality holds if and only if $T \cong S^4_{\frac{n-4}{2},\frac{n-4}{2}}$.
\end{enumerate}
\end{corollary}

\begin{proof}
Let $\mathcal C$ be the class of trees, which satisfies pendant-extension property.  
By Lemma~\ref{lem:alpha_beta_weak_pendant_monotonicity}, for any positive integers $s<t$, we have $\alpha(t)>\beta(t)>\alpha(s)>\beta(s)$. 

For odd $n$, write $n=2t+1$. 
By Theorem~\ref{lem:adjacency_nodal_domain}, we have
\begin{align*}
    \rho_2(T)\leq \alpha(t)=\rho_1(K_{1,t-1})=\sqrt{t-1}
    =\sqrt{\frac{n-3}{2}}.
\end{align*}
The star uniquely attains the maximum spectral radius among all graphs of a given order by~\cite{lovasz1973eigenvalues}. 
Theorem~\ref{lem:adjacency_nodal_domain} shows that an extremal tree is obtained from two copies of $K_{1,t-1}$ by gluing them through a vertex $v^*$ with $f(v^*)=0$.
 To  ensure that the extremal graph remains a tree,  $v^*$ is adjacent to exactly one vertex in each copy of $K_{1,t-1}$.
Based on specific way of gluing (two attachment vertices are centers or leaves), we obtain exactly three extremal trees $S^3_{t-1,t-1}$, $S^4_{t-2,t-1}$, $S^5_{t-2,t-2}$.

For even $n$, write $n=2t$. 
We claim that, for a leaf $v'$ of $K_{1,t-1}$, we have $\beta(t)=\sigma_1(A({K_{1,t-1}})-I_{v'v'})$.
Let $T$ be a tree of order $t$, $v\in V(T)$, and let $\mathbf{x}$ be a positive unit eigenvector of $A(T)-I_{vv}$ corresponding to $\sigma_1(A(T)-I_{vv})$. 
Choose vertices $v_1$ and $v_{t-1}$ such that $x_{v_1}=\max\limits_{z\in V(T)}x_z$ and $x_{v_{t-1}}=\min\limits_{z\in V(T)}x_z$.
By Rayleigh quotient,
\begin{align*}
    \sigma_1(A(T)-I_{vv})
    &=\mathbf{x}^{\top}(A(T)-I_{vv})\mathbf{x}  =2\sum_{z\sim w}x_zx_w-x_v^2 
    \leq 2x_{v_1}\sum_{z\neq v_1}x_z-x_{v_{t-1}}^2  \\
    &=\mathbf{x}^{\top}(A({K_{1,t-1}})-I_{v'v'})\mathbf{x} \leq \sigma_1(A({K_{1,t-1}})-I_{v'v'}).
\end{align*}
The equality forces $x_v=x_{v_{t-1}}$ and $x_w=x_{v_1}$ for every $w\neq v_1$. 
Then $v_1 \sim w$ for every $w\neq v_1$. 
Therefore $T\cong K_{1,t-1}$ and $v'$ is a leaf of $K_{1,t-1}$.
By Theorem~\ref{thm:connected_even_bridge} and the claim above, the extremal tree is obtained from two copies of the tree $K_{1,t-1}$ by joining their leaves, which implies that $T \cong S^4_{\frac{n-4}{2},\frac{n-4}{2}}$.
\end{proof}

\subsection{Outerplanar graphs}

Brooks, Gu, Hyatt, Linz and Lu~\cite{brooks2025maximum} considered the largest second adjacency eigenvalue of outerplanar graph and characterized the extremal graph for $n$ sufficiently large.
We extend their results by removing the condition that $n$ is sufficiently large.

Given two graphs $G_1$ and $G_2$,
$G_1\vee G_2$ denotes the graph obtained by adding all possible edges between $G_1$ and $G_2.$
The extremal outerplanar graph of adjacency spectral radius is shown as follows.
\begin{theorem}[\cite{lin2021complete}, Theorem 3]\label{thm:A_extremal graph}
Among all outerplanar graphs on $n$ vertices, $K_1 \vee P_{n-1}$ attains the maximum spectral radius, with the only exceptional case of $n=6$, in which $G_1$ attains the maximum spectral radius (see~\cref{fig:G1}).
\end{theorem}
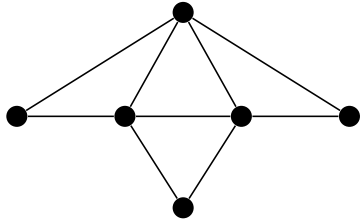
\begin{figure}[htbp]
\centering
\begin{tikzpicture}[scale=1.1]
    \tikzstyle{v}=[circle, fill=black, inner sep=2.8pt]

    \node[v] (a) at (-2,0) {};
    \node[v] (b) at (-0.7,0) {};
    \node[v] (c) at (0.7,0) {};
    \node[v] (d) at (2,0) {};
    \node[v] (e) at (0,1.25) {};
    \node[v] (f) at (0,-1.1) {};

    \draw[line width=0.6pt] (a)--(b)--(c)--(d);
    \draw[line width=0.6pt] (e)--(a);
    \draw[line width=0.6pt] (e)--(b);
    \draw[line width=0.6pt] (e)--(c);
    \draw[line width=0.6pt] (e)--(d);
    \draw[line width=0.6pt] (f)--(b);
    \draw[line width=0.6pt] (f)--(c);
\end{tikzpicture}
\caption{The graph $G_1$.}
\label{fig:G1}
\end{figure}

First, we consider all odd integers $n \geq 3$.
\begin{theorem}  
Let $G$ be an outerplanar graph of order $n$.    
If $n\geq 3$ and $n\neq 13$ is odd, then     
\begin{align*}        
\rho_2(G) \leq \rho_2(2(K_1 \vee P_{\frac{n-3}{2}}) \bullet K_1).    
\end{align*}    
Equality holds if and only if $G \cong 2(K_1 \vee P_{\frac{n-3}{2}}) \bullet K_1$ and $2(K_1 \vee P_{\frac{n-3}{2}}) \bullet K_1$ is still an outerplanar graph.
For $n=13$, the extremal graph is of form  $2G_1\bullet K_1$ if $2G_1\bullet K_1$ is still an outerplanar graph.
\end{theorem}  

\begin{proof}
It is clear that outerplanar graphs satisfy the pendant-extension property.
    By Theorem~\ref{lem:adjacency_nodal_domain} and Theorem~\ref{thm:A_extremal graph}, the result follows immediately.
\end{proof}

Next, to consider the even case, we need the following theorem about the extremal graph  for the largest adjacency eigenvalue with Dirichlet boundary condition.

\begin{theorem}\label{thm:M_extremal graph}
For $n\geq 24$, let $G$ be an outerplanar graph of order $n$ and $K_1 \vee P_{n-1}$ be the join with dominating vertex $v_1$ and path $P_{n-1}=v_2v_3\cdots v_n$.
    Then, for any $v \in V(G)$,
    \begin{align*}
        \sigma_1(A({K_1 \vee P_{n-1}})-I_{v_2v_2}) = \sigma_1(A({K_1 \vee P_{n-1}})-I_{v_nv_n}) \geq \sigma_1(A(G)-I_{v v}). 
    \end{align*}
    Equality holds if and only if $G \cong K_1 \vee P_{n-1}$ and $v=v_2$ or $v_n$.
\end{theorem}

\begin{proof}
For $n \geq 24$, let $(G,v)$ be the pair such that $\sigma_1(A(G)-I_{v v})$ attains the maximum among all outerplanar graphs of order $n$ and their vertices.
Let $\sigma = \sigma_1(A(G)-I_{v v})$. 
Based on~\cref{lem:Metzler matrix_PF_theorem}, let $\mathbf{x}$ be a positive eigenvector of $G$
corresponding to $\sigma$, scaled so that
$x_u=\max_{v\in V(G)} x_v=1$.
Let $A=N_G(u), B=V(G)\setminus (\{u\}\cup A)$, and $ d_G(u)=|A|$.
we shall prove that $B=\emptyset$ and that $G[A]$ is an induced path, which implies that $G\cong K_1\vee P_{n-1}$.

Let $L_1=\sqrt n+1-\frac{1}{n-\sqrt n}-\frac{n-\sqrt n}{2n}$ and $L_2=\sqrt n+1-\frac{3}{2(n-\sqrt n)}$. 

\begin{claim}\label{cla:minus-y1-lower-bound}
 $\sqrt n+1>\sigma\geq L_2 \ge L_1$.
\end{claim}

\begin{proof}
Let $\Gamma = K_1 \vee C_{n-1}$, where $C_{n-1}$ denotes a cycle on $n-1$ vertices. For $e \in E(C_{n-1})$, $K_1 \vee P_{n-1}$ is isomorphic to  $\Gamma-e$. 
For the upper bound, by Corollary \ref{cor:A-W_properties}, Lemma \ref{lem:proper_subgraph_spectral_radius}, and Theorem \ref{thm:A_extremal graph}, we have
$\sigma_1(A(G)-I_{vv})<\rho_1(G)\leq \rho_1(K_1\vee P_{n-1})<\rho_1(\Gamma)=\sqrt n+1$, where $\rho_1(\Gamma)$ has been calculated in~\cite{lin2021complete}. 

For the lower bound, let $\mathbf{y}=(y_1,y_2,\cdots,y_n)^\top$ be the Perron vector of $\Gamma$, where $y_1$ corresponds to the vertex of degree $n-1$.  
We have $y_1=\sqrt{\frac{n-\sqrt{n}}{2n}}$, and $y_i=\sqrt{\frac{1}{2(n-\sqrt{n})}}$ for $2 \leq i \leq n$, which were also computed in~\cite{lin2021complete}.
Then,
\begin{align*}
\sigma &\geq \sigma_1(A({K_1\vee P_{n-1}})-I_{v v}) \geq Y^\top(A({K_1\vee P_{n-1}})-I_{vv})Y 
= Y^\top(A({\Gamma})-I_{vv})Y-2y_2^2 \\
&\geq \sqrt{n}+1-2y_2^2-y_2^2 = \sqrt{n}+1-\frac{3}{2(n-\sqrt{n})}=L_2.
\end{align*}
By direct calculation, we have $L_2\geq L_1$.
\end{proof}

Next we recall the structure of $G[A]$ in~\cite{lin2021complete}. In fact, the proofs of Claims~2,~4 and~5 in \cite{lin2021complete} depend mainly on the outerplanarity of $G$ and the fact that $u$ has the largest eigenvector entry.  

\begin{claim}[\cite{lin2021complete}, Theorem 2, Claim 2]\label{cla:minus-y1-GA-paths}
$G[A]$ is a union of disjoint induced paths or an induced path. (In particular, we also view an isolated vertex in $G[A]$ as an induced path.)
\end{claim}

Let $S=\{w\in A:\ d_{G[A]}(w)=1\}.$ We want to show that $d_G(u)$ is close to $n-1$. 
As the first step, we associate $d_G(u)$ with $\sigma$ as follows.

\begin{claim}\label{cla:minus-y1-sigma-square}
\begin{align*}
    \sigma^2 
    <
    \begin{cases}
    \displaystyle
    d_G(u)+\sigma+2-\frac{2}{\sqrt n+1} +\sum_{w\in B}d_A(w)x_w,
    & \text{if } \ u = v; \\[1.2em]
    \displaystyle
    d_G(u)+2\sigma-\frac{2}{\sqrt n+2} +\sum_{w\in B}d_A(w)x_w,
    & \text{if } \ u \neq v.
    \end{cases}
\end{align*}

\end{claim}

\begin{proof}
If $u=v$, then the eigen-equation at $u$ gives $\sigma x_u = \sum_{w\in A}x_w-x_u$.
Since $x_u=1$, we have $\sigma = \sum_{w\in A}x_w-1$.
Then
\begin{align}\label{neq:sigma1}
        \sigma^2
        & = \sigma \left(\sum_{w\in A}x_w-1\right)                \le d_G(u)+\sum_{w\in A}d_A(w)x_w+
          \sum_{w\in B}d_A(w)x_w-\sigma.  
\end{align}
If $G[A]$ consists of isolated vertices, then $\sum_{w\in A}d_A(w)x_w=0$, and the desired bound is immediate. 
Otherwise, $G[A]$ contains at least one edge, and it follows that $|S|\geq 2$. 
For every $w\in S$, $\sigma x_w > 1.$
By Claim~\ref{cla:minus-y1-lower-bound}, $\sigma<\sqrt n+1$, and hence
$ x_w > \frac{1}{\sigma}> \frac{1}{\sqrt n+1}.$
By Claim~\ref{cla:minus-y1-GA-paths},
\begin{align*}
        \sum_{w\in A}d_A(w)x_w
        \le 2\sum_{w\in A}x_w-\sum_{w\in S}x_w < 2(\sigma+1)-\frac{2}{\sqrt n+1}.
\end{align*}
Substituting this into (\ref{neq:sigma1}) proves $ \sigma^2 < d_G(u)+\sigma+2-\frac{2}{\sqrt n+1} +\sum_{w\in B}d_A(w)x_w$.

If $u \neq v$, then
the eigen-equation at $u$ gives $\sigma x_u = \sum_{w\in A}x_w$.
Since $x_u=1$, we have $\sigma = \sum_{w\in A}x_w$.
Then
\begin{align}\label{neq:sigma2}
        \sigma^2
        & = \sigma \sum_{w\in A}x_w                \le d_G(u)+\sum_{w\in A}d_A(w)x_w+
          \sum_{w\in B}d_A(w)x_w. 
\end{align}
If $G[A]$ consists of isolated vertices, then $\sum_{w\in A}d_A(w)x_w=0$, and the desired bound is immediate. 
Otherwise, $G[A]$ contains at least one edge, and it follows that $|S|\geq 2$. 
For every $w\in S$, $\sigma x_w > 1-x_w.$
By Claim~\ref{cla:minus-y1-lower-bound}, $\sigma<\sqrt n+1$, and hence
$ x_w > \frac{1}{\sigma+1}> \frac{1}{\sqrt n+2}.$
By Claim~\ref{cla:minus-y1-GA-paths},
\begin{align*}
        \sum_{w\in A}d_A(w)x_w
        \le 2\sum_{w\in A}x_w-\sum_{w\in S}x_w \le 2\sigma-\frac{2}{\sqrt n+2}.
\end{align*}
Substituting this into (\ref{neq:sigma2}) proves $ \sigma^2 < d_G(u)+2\sigma-\frac{2}{\sqrt n+2} +\sum_{w\in B}d_A(w)x_v$.

The proof is complete.
\end{proof}

We now show that $|B|\leq 1$. 
Suppose to the contrary that $|B|\ge 2$.
Let
$B_1,\cdots,B_t$ be the vertex sets of all components of $G[B]$.

\begin{claim}[\cite{lin2021complete}, Theorem 2, Claim 4]\label{cla:minus-y1-B-components}
For every $i\in\{1,\ldots,t\}$, $d_A(B_i)=2$. Moreover, if $|B_i|\ge 2$, then
$2e(G[B_i])+e(A,B_i)\le 4|B_i|-3.$
In particular, $2e(G[B])+e(A,B)\le 4|B|-3.$
\end{claim}

\begin{claim}[\cite{lin2021complete}, Theorem 2, Claim 5]\label{cla:minus-y1-B-sum}
If $|B|\ge 2$, for $i=1,2$, we have
\[
        \sum_{w\in B}d_A(w)x_w
        \leq \frac{5|B|-2}{\sigma}
        =\frac{5n-5d_G(u)-7}{\sigma}
        \le \frac{5n-5d_G(u)-7}{L_i}.
\]
\end{claim}

Combining Claims~\ref{cla:minus-y1-sigma-square} and~\ref{cla:minus-y1-B-sum}, we obtain, for $u=v$,
\begin{align}
        \left(1-\frac{5}{L_1}\right)d_G(u)
        &\ge L_1^2-L_1-2+\frac{2}{\sqrt n+1}-\frac{5n-7}{L_1},       \label{eq:minus-y1-du-bound1}
\end{align}
and for $u \neq v$,
\begin{align}
        \left(1-\frac{5}{L_2}\right)d_G(u)
        &\ge L_2^2-2L_2+\frac{2}{\sqrt n+2}-\frac{5n-7}{L_2}.       \label{eq:minus-y1-du-bound2}
\end{align}
We compare the right-hand side with $(n-3)(1-5/L_i)$ for $i=1,2$.

For $u = v$,
let $s=\sqrt n$,
$\varepsilon=\frac{s-3}{2s(s-1)}$ and  $L_1=s+1/2+\varepsilon.$
Then, 
\begin{align*}
&L_1^2-L_1-2+\frac{2}{s+1}-\frac{5s^2-7}{L_1}
      -(s^2-3)\left(1-\frac{5}{L_1}\right)                                      \\
&= L_1^2-L_1-s^2+1+\frac{2}{s+1}-\frac{8}{L_1}                                      \\
&= \frac{7s-15}{4(s-1)}+\varepsilon^2+\frac{2}{s+1}-\frac{8}{L_1}              \\
&> \frac{7s-15}{4(s-1)}+\frac{2}{s+1}-\frac{8}{s+1/2}                   \\
&= \frac{14s^3-57s^2-46s+41}{4(s-1)(s+1)(2s+1)}.
\end{align*}
Let
$q(s)=14s^3-57s^2-46s+41.$
For \(s\ge\sqrt{22}\), we have \(q'(s)=42s^2-114s-46>0\), and
$q(\sqrt{22})=262\sqrt{22}-1213>0$. Therefore $q(s)>0$ for all $s \ge \sqrt{22}$, i.e. for all $n\ge 22$.  

For $u \neq v$,
let $s=\sqrt n$,
$\varepsilon'=\frac{3}{2s(s-1)}$ and  $L_2=s+1-\varepsilon'$, where $\varepsilon' < \frac12 $ for $s>3$.
Then
\begin{align*}
&L_2^2-2L_2+\frac{2}{s+2}-\frac{5s^2-7}{L_2}
      -(s^2-3)\left(1-\frac{5}{L_2}\right)                                      \\
&= L_2^2-2L_2-s^2+3+\frac{2}{s+2}-\frac{8}{L_2}                                      \\
&= \frac{2s-5}{s-1}+(\varepsilon')^2+\frac{2}{s+2}-\frac{8}{L_2}              \\
&> \frac{2s-5}{s-1}+\frac{2}{s+2}-\frac{8}{s+1/2}                   \\
&= \frac{4s^3 - 12s^2 - 39s + 20}{(s-1)(s+2)(2s+1)}.
\end{align*}
Let
$p(s)=4s^3-12s^2-39s+20.$
For $s\ge\sqrt{24}$, we have $p'(s)=12s^2-24s-39>0$, and
$p(\sqrt{24})=57\sqrt{24}-268>0$.
Therefore $p(s)>0$ for all $s \ge \sqrt{24}$, i.e. for all $n\ge 24$.

For $i=1,2$, by (\ref{eq:minus-y1-du-bound1}) and (\ref{eq:minus-y1-du-bound2}), we have
\[
        \left(1-\frac{5}{L_i}\right)d_G(u)>
        (n-3)\left(1-\frac{5}{L_i}\right).
\]
For $s\geq \sqrt{23}$, we have $L_2 > L_1 > s+\frac{1}{2}>5$, which implies $d_G(u)>n-3$. This contradicts $|B|=n-1-d_G(u)\ge2$.                              
Therefore, we have $|B|\le 1$. 
Suppose that $|B|=1$.
At this point, we can know the same information
on $G[A]$ as Claim 6 in \cite{lin2021complete}.

\begin{claim}[\cite{lin2021complete}, Theorem 2, Claim 6]\label{an induced path}
$G[A]$ is an induced path.
\end{claim}

Finally, we show that, indeed, $B$ is an empty set.

\begin{claim}
\label{B is empty}
$B=\emptyset$.
\end{claim}

\begin{proof}

Suppose that $|B|=1$. Let $B=\{b\}$. Since $G$ is $K_{2,3}$‐free, we have $d_G(b)=d_A(b)=2$.
By Claim~\ref{an induced path},
we write $G[A]=v_1v_2\cdots v_{n-2}.$
Let $N_G(b)=\{v_i,v_j\}$ with $i<j$.
If $|i-j| \neq 1$, then $G$ contains a $K_{2,3}$-minor, a contradiction. 
Thus, $|i-j| = 1$. 
WLOG, set $j=i+1$.
Let $X = (x_u, x_1, \dots, x_{n-2}, x_b)^\top$ be the eigenvector corresponding to $\sigma$, where $x_u = 1$, $x_b$ corresponds to $b$, and $x_k$ corresponds to $v_k$ for $k = 1, \dots, n-2$. 
Set $x_s \coloneqq \max\{x_k : k = 1, \dots, n-2\}$. 
Since $\sigma x_b \le x_{v_i}+x_{v_{i+1}}\le 2x_s$ and $\sigma\ge L_1 >5$, we have
$x_b<x_s$. 
Consequently, $\sigma x_s\le 1+2x_s+x_b<1+3x_s$,
which gives $x_s<\frac{1}{\sigma-3}$.
Also, $v_1$ is adjacent to $u$, so
$\sigma x_{v_1} \ge x_u+x_{v_2}-x_{v_1}$, which gives $x_{v_1}>\frac{1}{\sigma+1}.$
Let $G'=G-bv_i-bv_{i+1}+bu+bv_1.$
Note that $G'$ is also outerplanar and the diagonal perturbation is still at $v$.
 Let $M_z(G)=A({G})-I_{zz}$.
By Rayleigh quotient, we have
\begin{align*}
        \sigma_1(M_v(G'))-\sigma
        &\ge \frac{\mathbf{x}^{\top}(M_v(G')-M_v(G))\mathbf{x}}{\mathbf{x}^{\top}\mathbf{x}} \\
        &=\frac{2x_b(1+x_{v_1}-x_{v_i}-x_{v_{i+1}})}{\mathbf{x}^{\top}\mathbf{x}}       \\
        &>\frac{2x_b}{\mathbf{x}^{\top}\mathbf{x}}
          \left(1+\frac{1}{\sigma+1}-\frac{2}{\sigma-3}\right)                               \\
        &=\frac{2x_b}{\mathbf{x}^{\top}\mathbf{x}}
          \cdot\frac{\sigma^2-3\sigma-8}{(\sigma+1)(\sigma-3)}.
\end{align*}
Since $\sigma\ge L_1>5>\frac{3+\sqrt{41}}{2}$ for $n \ge23$, we have $\frac{\sigma^2-3\sigma-8}{(\sigma+1)(\sigma-3)} >0$. Therefore $\sigma_1(M(G'))>\sigma$, a contradiction. This proves the claim.
\end{proof}

Combining all the claims, it follows that
$G\cong K_1\vee P_{n-1}.$


Now it suffices to show that $v$ is either $v_2$ or $v_n$.
Since $v_2$ and $v_n$ are symmetric,
$\sigma_1(M_{v_2}(G))=\sigma_1(M_{v_n}(G))$.
It remains to show that every other choice of $v$ gives a strictly smaller value.

Let $\mathbf{w}>0$ be a unit eigenvector corresponding to $\sigma_1(M_v(G))$.
We claim that $w_v\leq w_z$  for $z\in V(G)$.
Indeed, if there exists $z\neq v$ such that $w_z<w_v$, then
$M_z(G)=M_v(G)+I_{vv}-I_{zz}$. 
It follows that
$\sigma_1(M_z(G))
    \geq
    \mathbf{w}^{\top}M_z(G)\mathbf{w}
    =
    \mathbf{w}^{\top}M_v(G)\mathbf{w}+w_v^2-w_z^2
    >
    \sigma_1(M_v(G)),$
which contradicts the choice of $v$.

Assume first that $v=v_k$ with $3\leq k\leq n-1$. 
Then the claim above gives $w_v\leq w_{v_2}$. 
Construct a new path on the same vertex set by replacing $v_2v_3\cdots v_{k-1}v_kv_{k+1}\cdots v_n$ with $v_k v_{k-1}\cdots v_2 v_{k+1}\cdots v_n$.
Equivalently, delete the edge $v_k v_{k+1}$ and add the edge $v_2 v_{k+1}$,
while keeping $v_1$ adjacent to all path vertices. 
Let $G'$ be the resulting
graph and $M'_v=A({G'})-I_{vv}$. 
Then $G'\cong K_1\vee P_{n-1}$, and $v=v_k$ is an endpoint of the path in $G'$. 
Hence $\sigma_1(M_v(G'))=\sigma_1(M_{v_2}(G))$.
Moreover,
$\mathbf{w}^{\top}M_v(G')\mathbf{w}-\mathbf{w}^{\top}M_v(G)\mathbf{w}
=
2w_{v_{k+1}}(w_{v_2}-w_v)\geq 0.$
Thus
$        \sigma_1(M_{v_2}(G))
        =
        \sigma_1(M_v(G'))
        \geq
        \mathbf{w}^{\top}M_v(G')\mathbf{w}
        \geq
        \mathbf{w}^{\top}M_v(G)\mathbf{w}
        =
        \sigma_1(M_v(G)).$
We claim that the  inequality is strict. 
If the equality holds, then
$\mathbf{w}$ would also be a positive eigenvector of $M_v(G')$ corresponding to
$\sigma_1(M_v(G'))$. 
Notice that $\sigma_1(M_v(G))w_v = w_{v_1}+w_{v_{k-1}}+w_{v_{k+1}}-w_v,$  and $\sigma_1(M_v(G'))w_v= w_{v_1}+w_{v_{k-1}}-w_v.$
Since $\sigma_1(M_v(G'))=\sigma_1(M_v(G))$, this gives $w_{v_{k+1}}=0$, a
contradiction. 
Therefore $\sigma_1(M_{v_2}(G))>\sigma_1(M_v(G))$, which contradicts the choice of $v$.
Hence no path interior vertex can maximize $\sigma_1(M_v(G))$.

It remains to exclude $v=v_1$. 
Assume that $v=v_1$, and let $\mathbf{w}>0$ be a unit eigenvector corresponding to $\sigma_1(M_{v_1}(G))$. 
Recall that $w_{v_1}=w_v \leq w_z$ for $z\in V(G)$.
On the other hand, by the proof of $G \cong K_1 \vee P_{n-1}$ as above, we have $w_{v_1}\geq w_z$ for $z\in V(G)$.
Hence $w_{v_1}=w_z=c>0$ for any $z\in V(G)$.
For $n\geq 4$, $N_G(v_2)=\{v_1,v_3\}$, $N_G(v_3)=\{v_1,v_2,v_4\}.$
The eigen-equation $M_{v_1}(G)\mathbf{w}=\sigma_1(M_{v_1}(G))\mathbf{w}$  gives $\sigma_1(M_{v_1}(G))c=2c$ at $v_2$ and  $\sigma_1(M_{v_1}(G))c= 3c$ at $v_3$, which is impossible. 
Therefore $v\neq v_1$.

The proof is complete.
\end{proof}

We extend Theorem 2 in~\cite{brooks2025maximum} by proving that the same conclusion holds for all even integers $n \geq 48$.

\begin{theorem}
Let $G$ be a connected outerplanar graph with $n$ vertices. 
If $n\geq 48$ is even, then 
\begin{align*}
    \rho_2(G) \leq \rho_2\left(\{(K_1\vee P_{\frac{n-2}{2}})^{(1)}\cup(K_1\vee P_{\frac{n-2}{2}})^{(2)}\}+v_1^{(1)}v_1^{(2)}\right),
\end{align*}
where $(K_1\vee P_{\frac{n-2}{2}})^{(1)}$ and
$(K_1\vee P_{\frac{n-2}{2}})^{(2)}$ are two copies of
$K_1\vee P_{\frac{n-2}{2}}$, and $P_{\frac{n-2}{2}}^{(i)}=v_1^{(i)}v_2^{(i)}\cdots v_{\frac{n-2}{2}}^{(i)}$ for $i=1,2$. 
Equality holds if and only if, up to isomorphism,  $G\cong\{(K_1\vee P_{\frac{n-2}{2}})^{(1)}\cup (K_1\vee P_{\frac{n-2}{2}})^{(2)}\} +v_1^{(1)}v_1^{(2)}$.
\end{theorem}

\begin{proof}    
It is clear that outerplanar graphs satisfy the pendant-extension property.    
If $n$ is even, then by Theorem~\ref{thm:connected_even_bridge} and Theorem~\ref{thm:M_extremal graph}, the result follows immediately.
\end{proof}



\subsection{The proof of Aouchiche--Hansen Conjecture}
In this subsection, we resolve Aouchiche--Hansen's Conjecture.
First, we recall some classical estimates on spectral radius.

\begin{theorem}[\cite{Nikiforov2002}]\label{thm:Nikiforov_fixed_size}
Let $G$ be a $K_{r+1}$-free graph with $m$ edges. 
Then $\rho_1(G)\leq \sqrt{2m(1-\frac1r)}$.
\end{theorem}


\begin{theorem}[{\cite{Hong1988mn}}]\label{thm:Hong}
   Let $G$ be a connected simple graph with $m$ edges and $n$ vertices.
Then 
\begin{align*}
    \rho_1(G)\leq \sqrt{2m-n+1}.
\end{align*}
Moreover, equality holds if and only if $G$ is isomorphic to one of the following two graphs:
\begin{enumerate}
    \item[(a)] the star $K_{1,n-1}$;
    \item[(b)] the complete graph $K_n$.
\end{enumerate}
\end{theorem}



\begin{theorem}[Rowlinson {\cite{Rowlinson1988}}]\label{thm:Rowlinson_fixed_size}
Let
\begin{align*}
    m=\binom{r}{2}+R,\qquad \text{where } 0<R<r,
\end{align*}
and let $G_m$ be the graph obtained from the complete graph $K_r$ by adding one new vertex of degree $R$. 
If $G$ is a graph with maximum spectral radius among the graphs with $m$ edges, then $G$ has a unique nontrivial component $H$ and $H=G_m$.
\end{theorem}

We also need the quotient matrix to estimate eigenvalues of some specific graphs.
\begin{definition}[{\cite[Definition~1.1]{brouwer2011spectra}}]\label{def:block_matrix}
Let $M$ be a complex matrix of order $n$ described in the following block form
\begin{equation*}
M=
\begin{pmatrix}
M_{11} & \cdots & M_{1t}\\
\vdots & \ddots & \vdots\\
M_{t1} & \cdots & M_{tt}
\end{pmatrix},
\end{equation*}
where the blocks $M_{ij}$ are $n_i\times n_j$ matrices for any $1\leq i,j\leq t$ and $n=n_1+\cdots+n_t$.
For $1\leq i,j\leq t$, let $b_{ij}$ denote the average row sum of $M_{ij}$, i.e., $b_{ij}$ is the sum of all entries in $M_{ij}$ divided by
the number of rows. 
Then $B(M)=(b_{ij})$, or simply $B$, is called the quotient matrix of $M$. 
If, in addition, for each pair $i,j$, the block $M_{ij}$ has constant row sum; that is, $M_{ij} {\mathbf e}_{n_j}=b_{ij} {\mathbf e}_{n_i}$, then $B$ is called the equitable quotient matrix of $M$, where $ {\mathbf e}_k=(1,1,\ldots,1)^\top\in\mathbb C^k$, and $\mathbb C$ denotes the field of complex numbers.
\end{definition}


\begin{corollary}[{\cite[Corollary~2.6]{you2019equitable}}]\label{cor:equitable_partition_adjacency}
If $M$ is the adjacency matrix of a graph, not necessarily connected, with an equitable partition and $B$ is the adjacency matrix of a divisor with
respect to the partition, then $\sigma_1(M)=\sigma_1(B)$, hence $\sigma_1(M)$ is an eigenvalue of $B$.
\end{corollary}

\begin{lemma}\label{lem:GrRs_spectral_radius_bound}
Let $m=\binom{r}{2}+R$, where $0<R<r$. 
Then
$\rho_1(G_m)< r-1+\frac{R+1}{r+2}.$
\end{lemma}

\begin{proof}
Let $A$ be the set of the $R$ vertices of $K_r$ adjacent to $u$, where $u$ is the new vertex in $G_m$. 
Let $C=V(K_r)\setminus A$. Then
$|A|=R$ and $|C|=r-R$.
By the partition $A,C,\{u\}$, the adjacency matrix $A(G_m)$ has the equitable quotient matrix
\begin{align*}
Q=
\begin{pmatrix}
R-1 & r-R & 1\\
R & r-R-1 & 0\\
R & 0 & 0
\end{pmatrix}.
\end{align*}
Since $A(G_m)$ is nonnegative, by
Corollary~\ref{cor:equitable_partition_adjacency}, we have $\rho_1(G_m)=\sigma_1(Q)$.
The characteristic polynomial of $Q$ is
\begin{align*}
p(\sigma) =\sigma^3-(r-2)\sigma^2-(r+R-1)\sigma+R(r-R-1).
\end{align*}
Then $p'(\sigma) =3\sigma^2-2(r-2)\sigma-(r+R-1)$, and
$p''(\sigma) = 6\sigma-2(r-2)$.
For $\sigma\geq r-1$, we have $p''(\sigma)\geq 6(r-1)-2(r-2)=4r-2>0$, which implies that $p'$ is strictly increasing on $[r-1,\infty)$. 
Moreover, $p'(r-1)=r(r-1)-R>0$.
It follows that $p$ is strictly increasing on $[r-1,\infty)$.
Let $a=\frac{R+1}{r+2}$. 
Then 
\begin{align*}
p(r-1+a)
&=a(r+a)(r-1+a)-R(R+a)=\bigl(a(r+a)-R\bigr)(r-1+a)+R(r-1-R)\\
&=\frac{(r-R+1)^2}{(r+2)^2}(r-1+a)+R(r-1-R)>0.
\end{align*}
Since $p(r-1)=-R^2<0$, $p$ has exactly one root in $(r-1,\infty)$, especially in $(r-1,r-1+\frac{R+1}{r+2})$, which implies that $\sigma_1(Q)<r-1+\frac{R+1}{r+2}$. 
This proves the lemma.
\end{proof}

We use $K_{n}-e$ to denote the graph obtained from the complete graph $K_{n}$ by deleting one edge.
We have the following lemma.

\begin{lemma}\label{lem:near_complete_clique_estimate_combined_combined}
Let $\omega\geq 3$ and let $G$ be a connected graph with $m$ edges.
Then
\begin{align*}
    \rho_1(G)
    <
    \begin{cases}
    \displaystyle
    \omega-1+
    \frac{m-\binom{\omega}{2}}{\omega+1},
    & \text{if } \omega(G)=\omega,\ G\not\cong K_{\omega+1}-e,\ 
    \binom{\omega}{2}<m\leq \binom{\omega+1}{2}, \\[1.2em]
    \displaystyle
    \omega-1+
    \frac{m-\binom{\omega}{2}-1}{\omega+1},
    & \text{if } \ G\not\cong K_\omega,\ 
    m\leq \binom{\omega}{2}.
    \end{cases}
\end{align*}
\end{lemma}

\begin{proof}
We first prove the upper bound under the first condition.
Fix a maximum clique $K_{\omega}$ of order $\omega$.
Let $S=V(G)\setminus V(K_\omega)$ and $s=|S|$.
Let $e(G)=\binom{\omega}{2}+R$. 
Then $1\leq R\leq \omega$.
Since $G$ is connected, $s\leq R$.
If $s\geq 3$, then $n=\omega+s$ and $m=\binom{\omega}{2}+R$. 
By Theorem~\ref{thm:Hong}, we have
\begin{align*}
\rho_1(G)^2
\leq 2m-n+1
=(\omega-1)^2+2R-s .
\end{align*}
Let $t=\omega-1+\frac{R}{\omega+1}$, 
as $s\geq 3$ and $R\leq \omega$, then
\begin{align*}
t^2-\bigl((\omega-1)^2+2R-s\bigr)
=s-\frac{4R}{\omega+1}+\frac{R^2}{(\omega+1)^2}
\geq (s-4)+\left(\frac{\omega+2}{\omega+1}\right)^2>0.
\end{align*}
Thus $\rho_1(G)<t$ in this case.

It remains to consider $s=1$ and $s=2$.
Write the adjacency matrix of $G$ as
\begin{align*}
A(G)=
\begin{pmatrix}
J_\omega-I_\omega & M\\
M^{\top} & C
\end{pmatrix},
\end{align*}
where $J_\omega$ is the all-one matrix and  $C=A(G[S])$.
Since $t>\omega-1$, the matrix
$P\coloneqq tI_\omega-(J_\omega-I_\omega)$ is positive definite, and then
$P^{-1}
=
\alpha I_\omega+\beta J_\omega$
where $\alpha=\frac1{t+1}$ and $\beta=\frac1{(t+1)(t-\omega+1)}$.
 Since $A(G)$ is symmetric, $\rho_1(G)\leq t$ is equivalent to showing that $tI-A(G)$ is positive semi-definite, where
\begin{align*}
    tI-A(G)=
    \begin{pmatrix}
        tI_\omega-(J_\omega-I_\omega) & -M\\
        -M^{\top} & tI_s-C
    \end{pmatrix}.
\end{align*}
Therefore, by the Schur complement, $tI-A(G)\geq 0$ is equivalent to $ tI_s-C-M^{\top}P^{-1}M\geq 0$.
Since $C+M^{\top}P^{-1}M$ is symmetric, it remains to prove $\sigma_1\bigl(C+M^{\top}P^{-1}M\bigr)< t$.

For $s=1$, the unique vertex of $S$ is adjacent to exactly $R$ vertices of $K$.
Since $\omega(G)=\omega$ and $G\not\cong K_{\omega+1}-e$, we have $R\leq \omega-2$.
Thus $ C+M^{\top}(\alpha I_\omega+\beta J_\omega)M = \alpha R+\beta R^2$.
A direct calculation gives
\begin{align*}
t-\alpha R-\beta R^2
=
\frac{
R^2-R(\omega^3+2\omega^2+4\omega+3)
+\omega^4+\omega^3-\omega^2-\omega
}{
(\omega+1)(R+\omega^2+\omega)
}.
\end{align*}
The numerator is decreasing in $R$ for $1\leq R\leq \omega-2$, and at $R=\omega-2$ it equals $\omega^3+10>0$.
Hence $\rho_1(G)< t$.

For $s=2$, we denote $S=\{x,y\}$.
Let $a=|N_G(x)\cap K|$ and $b=|N_G(y)\cap K|$.

If $xy\notin E(G)$, then $a+b=R$, and $a,b\geq 1$ by connectedness.
The matrix $M^{\top}(\alpha I_\omega+\beta J_\omega)M$ is positive semi-definite, so its largest eigenvalue is at most its trace. Hence
\begin{align*}
\sigma_1\bigl(M^{\top}(\alpha I_\omega+\beta J_\omega)M\bigr) \leq \alpha(a+b)+\beta(a^2+b^2)\leq \alpha R+\beta\bigl((R-1)^2+1\bigr).
\end{align*}
We now show that $\alpha R+\beta\bigl((R-1)^2+1\bigr)< t$.
Let $D=R+\omega^2+\omega$, and we have $\alpha=\frac{1}{t+1} =\frac{\omega+1}{D}$, $\beta = \frac{(\omega+1)^2}{RD}.$
Thus
\begin{align*}
    t-\alpha R-\beta\bigl((R-1)^2+1\bigr)=
    \frac{F(R)}{R(\omega+1)D},
\end{align*}
where $ F(R)=  R^3  -R^2(\omega^3+2\omega^2+4\omega+3) +R(\omega^4+3\omega^3+5\omega^2+5\omega+2) -2(\omega+1)^3$.
It follows that
\begin{align*}
   F''(R)
    =
    6R-2(\omega^3+2\omega^2+4\omega+3)
    \leq
    -2\omega^3-4\omega^2-2\omega-6
    <0.
\end{align*}
Hence $F(R)$ is concave on $[2,\omega]$. 
Therefore we only need to check $F(2)$ and $F(\omega)$. 
We have
$F(2)=2(\omega^4-2\omega^2-6\omega-1)>0$ and  $F(\omega)=\omega^4-4\omega^2-4\omega-2>0$
for every $\omega\geq 3$. 
It follows that $F>0$ for $2\leq R\leq \omega$.
Consequently, $\alpha R+\beta\bigl((R-1)^2+1\bigr)<t$.

If $xy\in E(G)$, then $a+b=R-1$.
If $a,b\geq 1$, by Lemma~\ref{thm:Weyl_inequality} and the similar calculation as above,
\begin{align*}
\sigma_1\bigl(C+M^{\top}(\alpha I_\omega+\beta J_\omega)M\bigr)
&\leq 1+\operatorname{tr}\bigl(M^{\top}(\alpha I_\omega+\beta J_\omega)M\bigr)\\
&\leq 1+\alpha(R-1)+\beta\bigl((R-2)^2+1\bigr)
< t.
\end{align*}
If one of $a,b$ is zero, WLOG let $b=0$, then $a=R-1$, and
\begin{align*}
C+M^{\top}(\alpha I_\omega+\beta J_\omega)M
=
\begin{pmatrix}
u & 1\\
1 & 0
\end{pmatrix},
\quad
u=\alpha(R-1)+\beta(R-1)^2.
\end{align*}
By direct calculation, $\rho_1(G)< t$.

It remains to prove the upper bound under the second condition.
If $m\leq \binom{\omega}{2}-1$, then $m=\binom r2+R$ with $0\leq R\leq r-1$, where $r\leq \omega-1$.
For $0< R\leq r-1$, by Theorem~\ref{thm:Rowlinson_fixed_size} and Lemma~\ref{lem:GrRs_spectral_radius_bound}, we have 
$$\rho_1(G)\leq \rho_1(G_m)< r-1+\frac{R+1}{r+2}=\frac{\binom{r+2}{2}+m-2}{r+2}.$$ 
For $R=0$,  by Theorem~\ref{thm:Hong}, $\rho_1(G)\leq \sqrt{2\binom{r}{2}-r+1}=r-1<r-1+\frac1{r+2}$.
It follows that $\rho_1(G)<\frac{\binom{r+2}{2}+m-2}{r+2}$ holds for all $0\leq R\leq r-1$.
Define
$\phi(x)=\frac{\binom{x}{2}+m-2}{x}$.
Since $m\leq \binom{r}{2}+r-1$, we have $\phi'(x)>0$ for $x\geq r+2$. 
Since $r+2\leq \omega+1$, then $\frac{\binom{r+2}{2}+m-2}{r+2}<\frac{\binom{\omega+1}{2}+m-2}{\omega+1}=\omega-1+\frac{m-\binom{\omega}{2}-1}{\omega+1}$.
Therefore, we have $\rho_1(G)<\omega-1+\frac{m-\binom{\omega}{2}-1}{\omega+1}$.

If $m=\binom{\omega}{2}$, then $\omega-1+\frac{m-\binom{\omega}{2}-1}{\omega+1}=\omega-1-\frac1{\omega+1}$. 
It is clear that $\omega(G)\leq \omega-1$ because $G\not\cong K_{\omega}$.
For $\omega=3$, $G$ is a tree with $m=3$.
Since the star uniquely attains the maximum spectral radius among all graphs of a given order~\cite{lovasz1973eigenvalues}, we have $\rho_1(G)\leq\rho_1(K_{1,3})=\sqrt3<\frac74=\omega-1-\frac1{\omega+1}$. 
For $\omega\geq 4$, if $\omega(G)\leq \omega-2$, we have $\rho_1(G)^2\leq 2\binom{\omega}{2}\left(1-\frac1{\omega-2}\right)=\omega(\omega-1)\frac{\omega-3}{\omega-2}<\left(\omega-1-\frac1{\omega+1}\right)^2$ by Theorem~\ref{thm:Nikiforov_fixed_size}.
It remains to consider that $\omega(G)=\omega-1$.
Since $m=\binom{\omega}{2}$, then $G\not\cong K_\omega-e$. 
By the first upper bound, we have $\rho_1(G)< \omega-2+\frac{\omega-1}{\omega}=\omega-1-\frac1\omega<\omega-1-\frac1{\omega+1}$.
The proof is complete.
\end{proof}

\begin{definition}[$k$-coalescence of graphs]
For a pair of connected graphs $G_1$ and $G_2$ with $n_1,n_2$ vertices and $m_1,m_2$ edges, respectively, each having an induced complete graph of order
$k$ with $n_1,n_2\geq k$, the graph obtained by merging $k$ vertices on
$\binom{k}{2}$ edges of the induced $K_k$ is called the $k$-coalescence of
$G_1$ and $G_2$, denoted by $G_1O_kG_2$. The graph $G_1O_kG_2$ is of order
$n_1+n_2-k$ with $m_1+m_2-\binom{k}{2}$ edges.
\end{definition}

\begin{lemma}\label{lem:coalesced-cliques}
Let $\omega \geq 2$. 
Let $m$ be the number of edges of the graph $K_\omega O_1 K_{\omega-1}$.
Then $\rho_2(K_\omega O_1 K_{\omega-1})<\frac{m-2}{\omega}$.
\end{lemma}

\begin{proof}
For $\omega=2$, $K_2O_1K_1\cong K_2$.
Clearly, $\rho_2(K_2)=-1<-\frac12=\frac{m-2}{\omega}$.
Assume that $\omega\geq 3$.
Since $m=\binom{\omega}{2}+\binom{\omega-1}{2} =(\omega-1)^2$, we have $\frac{m-2}{\omega} = \frac{(\omega-1)^2-2}{\omega} =\omega-2-\frac{1}{\omega}$.
The characteristic polynomial $P(K_{n_1}O_kK_{n_2};x)$ of $k$-coalescence $K_{n_1}O_kK_{n_2}$ is given by
\begin{align*}
P(K_{n_1}O_kK_{n_2};x)
&=
(x+1)^{n_1+n_2-k-3}
\Big[x^3-(n_1+n_2-k-3)x^2  \\
&+
\big((n_1-k-1)(n_2-k-1)
+(k-1)(n_1+n_2-2k-2)  \\
&-
k(n_1+n_2-2k)\big)x +
k(n_1-k)(n_2-k-1)\\
&+k(n_2-k)(n_1-k-1)
-(k-1)(n_1-k-1)(n_2-k-1)
\Big];
\end{align*}
see~\cite{J-P2021}.
Then
\begin{align*}
P(K_{\omega}O_1K_{\omega-1};x)=
(x+1)^{2\omega-5}
\Big(
x^3-(2\omega-5)x^2  
+
(\omega^2-7\omega+9)x
+2\omega^2-8\omega+7
\Big).
\end{align*}
Let $f(x)= x^3-(2\omega-5)x^2 +(\omega^2-7\omega+9)x +2\omega^2-8\omega+7$.
Denote three roots of $f$ by $\alpha_1\geq \alpha_2 \geq \alpha_3$.
Since
$f(-2)=1-2\omega<0$, $f(-1)=(\omega-2)(\omega-1)>0$, $f(\omega-3)=\omega-2>0$, $f\left(\omega-2-\frac{1}{\omega}\right) = -\frac{\omega^4-2\omega^3+\omega+1}{\omega^3}<0$, $f(\omega-1)=2-\omega<0$, and $f(\omega)=\omega+7>0$,
 we have
\begin{align*}
    \alpha_1\in(\omega-1,\omega),
    \qquad
    \alpha_2\in
    \left(\omega-3,\ \omega-2-\frac{1}{\omega}\right),
    \qquad
    \alpha_3\in(-2,-1).
\end{align*}
All the remaining $2\omega-5$ eigenvalues of $K_{\omega}O_1K_{\omega-1}$ are $-1$. 
Since $\omega\geq 3$, we have $\alpha_2>\omega-3\geq 0>-1$.
Then $\rho_2(K_\omega O_1K_{\omega-1})=\alpha_2$.
Therefore
\begin{align*}
        \rho_2(K_\omega O_1K_{\omega-1}) =
    \alpha_2 <
    \omega-2-\frac{1}{\omega}=\frac{m-2}{\omega}.
\end{align*}
This completes the proof.
\end{proof}







Then, we present a computational lemma.

\begin{lemma}\label{lem:new_packing}
Let $\omega\geq 4$ and let $m$ be a fixed real number.
Let $a,b,c$ be nonnegative integers such that $a+b+c=\omega$, $a\geq b\geq0$, and $b+c\geq2$.  
Let $r=ab+c(a+b)+\binom c2$.
For $b \geq 2$, define $L(a,b,c)\coloneqq \min\left\{ m-r, 2\left(m-\binom\omega2+\binom b2\right) \right\}$.
For $b\in\{0,1\}$, define $L(a,b,c)\coloneqq\min\left\{ m-r, 2\left(m-\binom\omega2\right) \right\}$.
Then
\begin{equation}\label{eq:new_packing}
L(a,b,c)\frac{\omega-3}{\omega-2}
    \leq \frac{(m-2)^2}{\omega^2}.
\end{equation}
Moreover, if $b=0$ or $b=1$, then the inequality is strict.
\end{lemma}

\begin{proof}
We first consider that $c=0$.  
For $c=0$, we have $a+b=\omega$ and $r=ab$. 
Let $ A(x)=x-ab $ and $B(x)=2\left(x-\binom{\omega}{2}+\binom b2\right)$.
Moreover, $b\geq 2$ and $a\leq \omega-2$.
Clearly,
\begin{align*}
    \min\{A(x),B(x)\}=
    \begin{cases}
        B(x), & x\leq M,\\
        A(x), & x\geq M,
    \end{cases}
\end{align*}
where $M=a(\omega-1)$ and $A(M)=B(M)$.
Let $F(x) = \frac{(x-2)^2}{\omega^2}-\frac{\omega-3}{\omega-2}\min\{A(x),B(x)\}.$
For $x\leq M$,
$
F'(x)=\frac{2(x-2)}{\omega^2}
-\frac{2(\omega-3)}{\omega-2}
\leq  
\frac{2(a(\omega-1)-2)}{\omega^2}
-\frac{2(\omega-3)}{\omega-2}
\leq 
\frac{2((\omega-1)(\omega-2)-2)}{\omega^2}
-\frac{2(\omega-3)}{\omega-2}
<0.$
It follows that $F(x)\geq F(M)=\frac{(b-2)P(a,b)}{\omega^2(\omega-2)}$, where $P(a,b)=a^3+(2b-5)a^2+(b^2-5b+2)a+4.$
Then $\frac{\partial P}{\partial a}= 3a^2+2(2b-5)a+(b^2-5b+2)\geq 8b^2-15b+2>0$
due to $a\geq b$, which implies that $P(a,b)\geq P(b,b)=2(b-2)(b-1)(2b+1)\geq 0.$
Therefore, we have $F(x) \geq F(M)\geq 0$ for $x\leq M$.

For $x\geq M$, $F'(x)=\frac{2(x-2)}{\omega^2}
-\frac{\omega-3}{\omega-2}.$
Solving $F'(x)=0$ gives $x_0=2+\frac{\omega^2(\omega-3)}{2(\omega-2)}.$
If $x_0\leq M$, then $F$ is increasing on $[M,\infty)$, and $F(x)\geq F(M)\geq 0$.

Suppose that $x_0>M$. 
Since $F$ is convex on $[M,\infty)$,  $F$ attains its minimum at $x_0$.
By calculation, the inequality $x_0>M$ is equivalent to $a<\frac{\omega^3-3\omega^2+4\omega-8}{2(\omega-1)(\omega-2)}.$ 
Note that $F(x_0)=\frac{\omega-3}{\omega-2}
\left(ab-2-\frac{\omega^2(\omega-3)}{4(\omega-2)}
\right).$
It suffices to prove $ab-2-\frac{\omega^2(\omega-3)}{4(\omega-2)}>0.$
Let $A=\frac{\omega^3-3\omega^2+4\omega-8}{2(\omega-1)(\omega-2)}.$
Since $a\geq b$, we have $a\geq \omega/2$. 
Moreover, the function $g(y)=y(\omega-y)$ is decreasing on $[\omega/2,\infty)$. 
Then
$$ab=a(\omega-a)>A(\omega-A)=
2+\frac{\omega^2(\omega-3)}{4(\omega-2)}+
\frac{(\omega-4)^2(\omega-3)(\omega^2-\omega+2)}
{4(\omega-2)^2(\omega-1)^2}.$$
So $ab>2+\frac{\omega^2(\omega-3)}{4(\omega-2)}.$
Consequently, $F(x)\geq F(x_0)>0$ for all $x\geq M$.

Taking $x=m$, we obtain the desired inequality for $c=0$.

Next we consider that $c>0$.
Suppose that $b\geq 2$.
We prove that $L(a,b,c)\leq L(a+c,b,0)$.
Note that
$r-\bigl((a+c)b\bigr)=ab+c(a+b)+\binom c2-(a+c)b=ac+\binom c2> 0.$
Then $m-r < m-(a+c)b$.
Since $b$ is unchanged, the second term in $L(a,b,c)$ is unchanged. 
Therefore $L(a,b,c)\leq L(a+c,b,0)$.

If $b=1$, then $c\geq1$ and $a+c=\omega-1$.  
We compare $L(a,1,c)$ with $L(\omega-2,2,0)$.
Note that $r-2(\omega-2)=
a(c-1)+\frac{c^2-3c+4}{2}
>0$, and we have $m-r<m-2(\omega-2)$.
Besides, $x-\binom{\omega}{2} <x-\binom{\omega}{2}+\binom 22$.
Consequently, $L(a,1,c)< L(\omega-2,2,0)$, which implies that~\eqref{eq:new_packing} is strict.
If $b=0$, then $c\geq2$ and $a+c=\omega$. 
Again we compare $L(a,0,c)$ with $L(\omega-2,2,0)$.  
Since $r-2(\omega-2)=
1+\frac{(c-2)(2\omega-c-3)}{2}
>0$, we have $m-r<m-2(\omega-2)$.
Besides, $x-\binom{\omega}{2} <x-\binom{\omega}{2}+\binom 22$.
Then $L(a,0,c)< L(\omega-2,2,0)$, which implies that~\eqref{eq:new_packing} is strict.

The proof is complete.
\end{proof}

Now we are ready to prove the conjecture.
\begin{proof}[\bf{{Proof of Theorem~\ref{thm:second largest eigenvalue}}}]

Let $G$ be a connected graph of order $n\geq 2$, size $m$, and clique number $\omega$.
Clearly, $\omega\geq 2$.
If $n=2$, then  $G\cong K_2$, so the assertion is immediate.
Assume that $n\geq 4$.
If $\rho_2(G)=0$, then $|\rho_2(G)|\omega=0$, and there is nothing to prove.
If $\rho_2(G)<0$, then $G$ is a complete graph. 
In this case $\omega=n$, $m=\binom n2$, and $|\rho_2(G)|=1$.
If $n$ is odd, then $|\rho_2(G)|\omega=n \leq \binom n2-2=m-2$, where the inequality strictly holds for odd order $n\geq 5$.
If $n$ is even, then we claim that $|\rho_2(G)|\omega-e(G) \leq -2< |\rho_2(2K_\omega+e)|\omega-e(2K_\omega+e))$.
The characteristic polynomial $P(2K_\omega+e;x)$ of $2K_\omega+e$ is given by $P(2K_\omega+e;x)
=(x+1)^{2\omega-4}
\left(x^2-(\omega-1)x-1\right)
\left(x^2-(\omega-3)x-(2\omega-3)\right)$.
Then $\rho_2(2K_\omega+e)=
\frac{\omega-3+\sqrt{(\omega-1)(\omega+3)}}{2}$.
 For $\omega\geq 2$, we have 
$\omega\rho_2(2K_\omega+e)-e(2K_\omega+e)=
\frac{\omega(\omega-3)+\omega\sqrt{(\omega-1)(\omega+3)}}{2}-\omega(\omega-1)-1=
\frac{
\omega\sqrt{(\omega-1)(\omega+3)}
-\omega^2-\omega-2
}{2}>-2.$
Consequently, $|\rho_2(G)|\omega-e(G) \leq -2< |\rho_2(2K_\omega+e)|\omega-e(2K_\omega+e)$.
Therefore, in the rest of the proof, we may assume that $\rho_2>0$.

Let $f$ be a $\rho_2$-eigenfunction, and let $U_1,\ldots,U_l$ be the nodal domains of $f$. 
Write $\Omega_{U_i}$ as $\Omega_i$ for simplicity.
Let $G[\Omega_i]$ be the subgraph of $G$ induced by $\Omega_i$.
Let $\kappa = \max_{1\leq i\leq l}\omega(G[\Omega_i])$.
Clearly, \(\kappa\leq \omega\).
Denote $m_i=e(G[\Omega_i])$.
Let $r=m-\sum_{i=1}^l m_i$. 
By Theorem~\ref{thm:adjacency_nodal_D} and Corollary~\ref{cor:A-W_properties}, for each nodal domain $U_i$,
\begin{align}\label{eq:nodal_basic_comparison}
    \rho_2=\mu_1(G_{U_i})\leq \rho_1(G[\Omega_i]).
\end{align}
If there exists a boundary edge with positive weight in $G_{U_i}$, then the inequality of~\eqref{eq:nodal_basic_comparison} is strict.



We first establish the following auxiliary inequality.
If $m_i+r \geq \binom{\omega}{2}+2$, then there exists $j\neq i$ such that
\begin{equation}
    \begin{aligned}\label{eq_1_1_1}
    2m_j\frac{\omega-1}{\omega}
    \leq
    2\left(m-\binom{\omega}{2}-2\right)\frac{\omega-1}{\omega}
    \leq
    \frac{(m-2)^2}{\omega^2},
\end{aligned}
\end{equation}
where the equality in \eqref{eq_1_1_1} holds if and only if $m=\omega^2-\omega+2$ and $m_j=m-\binom{\omega}{2}-2=\binom{\omega}{2}$.

Let $Z\coloneqq V(G)\setminus \bigcup_{k=1}^l \Omega_k =\{u\in V(G):f(u)=0\}$.
We split the proof into two cases according to whether $Z$ is empty.

\setcounter{mycase}{0}

\mycase{case:Z-nonempty}{$Z\neq\emptyset$.}

We distinguish the following three subcases according to the value of $\kappa$.

\mysubcase{subcase:Z-nonempty-kappa-omega}{$\kappa=\omega$.}

Choose a nodal domain $U$ such that $G[\Omega_U]$ contains a clique $K_\omega$.
WLOG, we may assume that $U=U_1$.
Since  $Z\neq\emptyset$, at least two edges of $G$ are not internal to any nodal domain. 
It follows that $r\geq 2$.
Then $m_1+r \geq \binom{\omega}{2}+2$.
By~\eqref{eq_1_1_1}, there exists another nodal domain $U_j$  such that $2m_j\frac{\omega-1}{\omega} \leq \frac{(m-2)^2}{\omega^2}$.
Thus, by~\eqref{eq:nodal_basic_comparison} and Theorem~\ref{thm:Nikiforov_fixed_size}, we have
\begin{equation}
    \begin{aligned}\label{eq_Z_a_1}
        \rho_2^2 \leq \rho_1^2(G[\Omega_j])\leq 2m_j\frac{\omega-1}{\omega} \leq  \frac{(m-2)^2}{\omega^2}.
\end{aligned}
\end{equation}

\mysubcase{subcase:Z-nonempty-kappa-omega-1}{$\kappa=\omega-1$.}
 
If $\omega=2$, then $\kappa =1$. 
Hence every $G[\Omega_i]$ is a single vertex. 
By~\eqref{eq:nodal_basic_comparison}, we have $\rho_2\leq0$, contrary to our assumption $\rho_2>0$. 
Therefore, we may assume that $\omega\geq3$.

Choose a nodal domain $U$ such that $G[\Omega_U]$ contains a clique $K_{\omega-1}$.
WLOG, we may assume that $U=U_1$.
Then $m_1\geq \binom{\omega-1}{2}.$
Since no nodal domain contains $K_\omega$, the clique $K_\omega$ in $G$ contributes at least $\omega-1$ edges that are not internal to any nodal domain. 
Moreover, since $Z\neq\emptyset$ and $f$ has both positive and negative entries, there is at least one more edge which is not internal to any nodal domain. 
Hence $r\geq \omega$.
Let $U_2$ be another nodal domain different from $U_1$.
It follows that $m_2 \leq m-\binom{\omega-1}{2}-\omega$ due to
$\min\{2m_1,2m_2\}\leq m_1+m_2\leq m-\omega$.
By \eqref{eq:nodal_basic_comparison} and Theorem~\ref{thm:Nikiforov_fixed_size},
\begin{equation}
    \begin{aligned}\label{eq_Z_a_2}
\rho_2^2 \leq \min\{\rho_1^2(G[\Omega_1]), \rho_1^2(G[\Omega_2])\} \leq \min\{m-\omega, 2(m-\binom{\omega-1}{2}-\omega)\}\frac{\omega-2}{\omega-1}.
\end{aligned}
\end{equation}
We claim that $\rho_2^2\le \frac{(m-2)^2}{\omega^2}.$
Let $A(x)=(x-\omega)$
and $B(x)= 2(x-\binom{\omega-1}{2}-\omega).$
It suffices to show that $\min\{A(m),B(m)\}\frac{\omega-2}{\omega-1}\leq \frac{(m-2)^2}{\omega^2}$.
Clearly,
\begin{align*}
    \min\{A(x),B(x)\}=
    \begin{cases}
        B(x), & x\leq M,\\
        A(x), & x\geq M,
    \end{cases}
\end{align*}
where $M=\omega^2-2\omega+2$ and $A(M)=B(M)$.
Let $F(x)= \frac{(x-2)^2}{\omega^2}-\min\{A(x),B(x)\}\frac{\omega-2}{\omega-1}$.
For $x\leq M$,  $F'(x) = \frac{2(x-2)}{\omega^2}- 2\frac{\omega-2}{\omega-1}< 0$, which implies that $F(x)\geq F(M)=0$.  
For $x\geq M$, $F'(x) = \frac{2(x-2)}{\omega^2} - \frac{\omega-2}{\omega-1} \geq 0$, which implies that $F(x)\geq F(M)=0$.  
Therefore, we have $\rho_2^2\le \frac{(m-2)^2}{\omega^2}$.

\mysubcase{subcase:Z-nonempty-kappa-omega-2}{$\kappa \leq \omega-2$.}

Since $\kappa\geq1$, $\omega\geq 3$.
If $\omega=3$, then $\kappa=1$. 
Hence each $G[\Omega_i]$ is a single vertex. 
By~\eqref{eq:nodal_basic_comparison}, we have $\rho_2\leq0$.
Assume that $\omega\geq4$.

Choose a maximum clique $K_\omega$ of $G$. 
Let $U$ be a nodal domain such that $|K_\omega\cap \Omega|=\max_i |K_\omega\cap \Omega_i|$, where $\Omega$ is the interior of $U$.
WLOG, assume that $U=U_1$. 
Let $a=|K_\omega\cap \Omega_1|$.
Replacing $f$ by $-f$ if necessary, assume that $U_1$ is a positive nodal domain.
Note that $K_\omega$ can merely intersect with one positive nodal domain.
If a negative nodal domain intersects 
$K_\omega$, let $U_2$ be that domain; otherwise, choose an arbitrary negative nodal domain and still denote it by $U_2$.
Let $b=|K_\omega\cap \Omega_2|$ and $c=|K_\omega\cap Z|$. 
Then $a+b+c=\omega$ and $a\geq b\geq0$.
Since $\kappa\leq\omega-2$, we have $a\leq\omega-2$. 
Then $b+c\geq2$.
Let
\begin{align}\label{eq:clique_crossing_r}
    r_K=ab+c(a+b)+\binom c2,
\end{align}
where $r_K$ counts the edges of $K_\omega$ which are not contained in any nodal domain.
It follows that $r\geq r_K$.
So $\min\{2m_1,2m_2\}\leq m_1+m_2 \leq m-r\leq m-r_K$, and $m_1+r_K\geq \binom{\omega}{2}-\binom b2$.
Then by~\eqref{eq:nodal_basic_comparison}, \eqref{eq:clique_crossing_r}, Theorem~\ref{thm:Nikiforov_fixed_size}, and Lemma \ref{lem:new_packing}, we have
\begin{align*}
    \rho_2^2 \leq \min\{2m_1,2m_2\}\frac{\omega-3}{\omega-2}
    \leq
    \min\left\{m-r_K,\,2\left(m-\binom{\omega}{2}+\binom b2\right)\right\}
    \frac{\omega-3}{\omega-2}\leq \frac{(m-2)^2}{\omega^2}.
\end{align*}
If $b\neq 0$, then $K_\omega$ intersects both positive and negative nodal domains, which implies that the total boundary edge weights in $G_U$ do not vanish.
Therefore the inequalities in~\eqref{eq:nodal_basic_comparison} are strict for $i=1,2$.
If $b=0$, then the inequality in the Lemma~\ref{lem:new_packing} is strict.
Therefore, the equality cannot hold in this subcase.

\mycase{case:Z-empty}{$Z=\emptyset$.}

Since $Z=\emptyset$, by Theorem~2 in ~\cite{nodal_domain_LAA}, the eigenfunction $f$ corresponding to $\rho_2(G)$ has exactly two nodal domains, denoted  by $U_1$ and $U_2$.
Moreover, for each $i=1,2$, every boundary edge weight in $G_{U_i}$ is nonzero.
Thus the inequality~\eqref{eq:nodal_basic_comparison} is strict.

We distinguish the following three subcases.

\mysubcase{subcase:Z-empty-kappa-omega}{$\kappa = \omega$.}

We may assume that $G[\Omega_1]$ contains a clique $K_\omega$, and we have $m_1\geq \binom{\omega}{2}$. 
If $m_1+r\geq \binom{\omega}{2}+2$, then for the nodal domain $U_2$, by~\eqref{eq:nodal_basic_comparison}, ~\eqref{eq_1_1_1}, and Theorem \ref{thm:Nikiforov_fixed_size}, we have $\rho_2^2 < \rho_1^2(G[\Omega_2])\leq 2m_2\frac{\omega-1}{\omega} \leq  \frac{(m-2)^2}{\omega^2}.$

It remains to consider $m_1+r=\binom{\omega}{2}+1$.
Since $m_1\geq\binom{\omega}{2}$ and $r\geq1$, this forces $m_1=\binom{\omega}{2}$ and $r=1$.
Then $G[\Omega_1]\cong K_\omega$. 
Since the inequality \eqref{eq:nodal_basic_comparison} is strict, it gives $\rho_2<\rho_1(K_\omega)=\omega-1$.
By~\eqref{eq:nodal_basic_comparison} and Lemma~\ref{cor:spectral_radius_monotonicity}, we have $\rho_2 = \sigma_1(A(K_\omega)-I_{vv})\geq\sigma_1(A(K_{\omega-1}))=\omega-2$.
Thus $\rho_2 \in [\omega-2, \omega-1)$.
Next, we divide the size of $m_2$ into three cases.

Firstly,  $m_2\geq \binom{\omega}{2}+1$.
It follows that $m=m_1+m_2+r\geq \omega(\omega-1)+2$.
Hence $\rho_2<\omega-1 \leq \frac{m-2}{\omega}$.

Secondly, $m_2\leq \binom{\omega-1}{2}$.
If $G[\Omega_2]\not\cong K_{\omega-1}$, then, by Lemma~\ref{lem:near_complete_clique_estimate_combined_combined} and \eqref{eq:nodal_basic_comparison}, we have $\rho_2<\omega-2-\frac{1}{\omega}$, a contradiction with $\rho_2\geq\omega-2$.
If $G[\Omega_2]\cong K_{\omega-1}$, then, by \eqref{eq:nodal_basic_comparison}, we have $\rho_2<\rho_1(K_{\omega-1})=\omega-2$, a contradiction.

Finally, we assume that $m_2=\binom{\omega-1}{2}+R,$ where $1\leq R\leq \omega-1.$
For $\omega(G[\Omega_2])\leq \omega-2$, by \eqref{eq:nodal_basic_comparison} and Theorem~\ref{thm:Nikiforov_fixed_size}, we have 
$$\rho_2^2< 2\left(\binom{\omega-1}{2}+R\right)\frac{\omega-3}{\omega-2}< \left(\omega-2+\frac{R}{\omega}\right)^2= \left(\frac{m-2}{\omega}\right)^2.$$

For  $\omega(G[\Omega_2])=\omega-1$,
we claim that $\omega\geq3$.
If $\omega=2$, then $G[\Omega_2]$ is a single vertex. 
By~\eqref{eq:nodal_basic_comparison}, we have $\rho_2\leq0$, contrary to our assumption $\rho_2>0$. 
If $G[\Omega_2] \not\cong K_{\omega}-e$, then by \eqref{eq:nodal_basic_comparison} and Lemma~\ref{lem:near_complete_clique_estimate_combined_combined}, we have $\rho_2<\rho_1(G[\Omega_2]) \leq \omega-2+\frac{R}{\omega}=\frac{m-2}{\omega}$.
It remains to consider the exceptional case $G[\Omega_2]\cong K_{\omega}-e$, which implies that $|\Omega_2|=\omega$.
The graph $G$ now satisfies $G[\Omega_1]\cong K_\omega$, $G[\Omega_2]\cong K_\omega-e$, and $r=1$, where $G$ is of even order.
Let $u\in V(K_\omega)$.
Let $v_1v_2$ be the missing edge of $K_\omega-e$.
Define $G_1\coloneqq (K_\omega-e)\cup K_\omega+uv_1$.
For some $v'\in V((K_\omega-e))\setminus\{v_1,v_2\}$, define $G_2\coloneqq (K_\omega-e)\cup K_\omega+uv'$.
Let $q_\omega(x)=x^2-(\omega-3)x-2(\omega-2)$, $f_1(x)=x(x-\omega+1)q_\omega(x)-(x-\omega+2)^2$, and $f_2(x)=
(x+1)^2(x-\omega+1)q_\omega(x)
-(x-\omega+2)q_{\omega-1}(x).$
The characteristic polynomial $P(G_1;x)$ and $P(G_2;x)$ are given by
\begin{align*}
    P(G_1;x)= (x+1)^{2\omega-4}f_1(x),\quad  
P(G_2;x)=x(x+1)^{2\omega-6}f_2(x).
\end{align*} 
For each $i=1,2$, we claim that the polynomial $f_i(x)$ has exactly one root in $(\omega-2,\omega-1)$ and this root is precisely $\rho_2(G_i)$.
Note that $f_1(\omega-2)f_1(\omega-1)=-(\omega-2)^2<0$ and $f_2(\omega-2)f_2(\omega-1)=-(\omega-1)^2(\omega-2)(\omega+3)<0$, which implies that $f_i(x)$ has some roots in $(\omega-2,\omega-1)$.
Since $G_i$ is obtained from $(K_\omega-e)\cup K_\omega$ by adding one edge, by Lemma~\ref{thm:Weyl_inequality}, $\rho_3(G_i)\leq \rho_3((K_\omega-e)\cup K_\omega))+1\leq \omega-2$.
On the other hand, since $G_i$ is connected and properly contains $K_\omega$, by Lemma~\ref{lem:proper_subgraph_spectral_radius}, we have $\rho_1(G_i)>\omega-1$.
Therefore, $f_i(x)$ has exactly one root in $(\omega-2,\omega-1)$, and this root is precisely $\rho_2(G_i)$.
Let $\alpha=\rho_2(G_1)$.
Then $f_1(\alpha)=0$, and $\alpha(\alpha-\omega+1)q_\omega(\alpha)=(\alpha-\omega+2)^2$.
Using this equality, we obtain that
\begin{align*}
f_2(\alpha)
&=(\alpha+1)^2(\alpha-\omega+1)q_\omega(\alpha)
-(\alpha-\omega+2)q_{\omega-1}(\alpha)\\
&=\frac{(\alpha+1)^2}{\alpha}(\alpha-\omega+2)^2
-(\alpha-\omega+2)q_{\omega-1}(\alpha)\\
&=(\alpha-\omega+2)
\left(
\frac{(\alpha+1)^2(\alpha-\omega+2)}{\alpha}
-q_{\omega-1}(\alpha)
\right)\\
&=-\frac{(\alpha-\omega+2)(\alpha+\omega-2)}{\alpha}<0.
\end{align*}
Note that $f_2(\omega-2)>0$.
Thus $f_2$ has a root in $(\omega-2,\alpha)$, which implies that $\rho_2(G_2)<\alpha=\rho_2(G_1)$.
Then $\omega\rho_2(G_2)-e(G_2) < \omega\rho_2(G_1)-e(G_1)$.

For $\omega(G[\Omega_2])=\omega$,  $m_2=\binom{\omega}{2}$.
Then $G$ consists of two copies of $K_{\omega}$ joined by an edge; that is, $G\cong 2K_{\omega}+e$.
Now we show that $\omega\rho_2(G_1)-e(G_1) <
    \omega\rho_2(2K_{\omega}+e)-e(2K_{\omega}+e)$.
Recall that $\rho_2(2K_\omega+e)=\frac{\omega-3+\sqrt{(\omega-1)(\omega+3)}}{2}$.
Note that $f_1(\omega-1-\frac{2}{\omega}+\frac{1}{2\omega^2})<0$.
Therefore, we have
\begin{align*}
        \rho_2(G_1)<\omega-1-\frac{2}{\omega}+\frac{1}{2\omega^2}<\frac{\omega-3+\sqrt{(\omega-1)(\omega+3)}}{2}-\frac{1}{\omega}=\rho_2(2K_{\omega}+e)-\frac{1}{\omega}.
\end{align*}
It follows that
$$\omega\rho_2(G_1)-e(G_1)< \omega\left(\rho_2(2K_{\omega}+e)-\frac{1}{\omega}\right)-e(G_1) =\omega\rho_2(2K_{\omega}+e)-e(2K_{\omega}+e).$$
Moreover, recall that $\omega\rho_2(2K_\omega+e)-e(2K_\omega+e) >-2$.
Overall, we have  $\rho_2(G)\omega-e(G)<-2$ for odd order,
and the graph $2K_\omega+e$ attains the maximum value of $\rho_2(G)\omega-e(G)$ for even order. 

\mysubcase{subcase:Z-empty-kappa-omega-1}{$\kappa = \omega-1$.}

If $\omega=2$, then $\kappa=1$. 
So every $G[\Omega_i]$ is a single vertex. 
By~\eqref{eq:nodal_basic_comparison}, we have $\rho_2\leq0$, contrary to our assumption $\rho_2>0$. 
Therefore, we may assume that $\omega\geq3$.

Assume that $G[\Omega_1]$ contains $K_{\omega-1}$.
Then $m_1\geq \binom{\omega-1}{2}$.
Since no nodal domain contains $K_\omega$, every maximum clique $K_\omega$  with both $G[\Omega_1]$ and $G[\Omega_2]$. 
Therefore, such a clique contributes at least
$\omega-1$ edges between $G[\Omega_1]$ and $G[\Omega_2]$, which implies that $r\geq \omega-1$. 
Then, $m_1+r\geq \binom{\omega-1}{2}+\omega-1=\binom{\omega}{2}$.
If $m_1+r\geq \binom{\omega}{2}+2$, by \eqref{eq:nodal_basic_comparison}, \eqref{eq_1_1_1}, and Theorem~\ref{thm:Nikiforov_fixed_size}, we have  $\rho_2^2 < \frac{(m-2)^2}{\omega^2}$.

It remains to consider that $m_1+r=\binom{\omega}{2}+\delta$ with $\delta\in\{0,1\}$.
Suppose first that $m_2\leq \binom{\omega-1}{2}$. 
If $G[\Omega_2]\not\cong K_{\omega-1}$, then,  by~\eqref{eq:nodal_basic_comparison} and Lemma~\ref{lem:near_complete_clique_estimate_combined_combined}, we have
\begin{align*}
    \rho_2<
    \rho_1(G[\Omega_2])
    <\omega-2+
    \frac{m_2-\binom{\omega-1}{2}-1}{\omega}
    \leq
    \frac{m-2}{\omega}.
\end{align*}
If $G[\Omega_2]\cong K_{\omega-1}$ and $\delta=1$, then by~\eqref{eq:nodal_basic_comparison}, we have $\rho_2<\rho_1(G[\Omega_2])=\omega-2=\frac{m-2}{\omega}$.
If $G[\Omega_2]\cong K_{\omega-1}$ and $\delta=0$, then $m_1=\binom{\omega-1}{2}$ and $r=\omega-1$.
Hence $G[\Omega_1]\cong K_{\omega-1}$, and the equality $r=\omega-1$ forces $G\cong K_\omega O_1K_{\omega-1}$. 
By~\cref{lem:coalesced-cliques}, we have  $\rho_2(K_\omega O_1K_{\omega-1})< \frac{m-2}{\omega}$.

Now suppose that $m_2\geq \binom{\omega-1}{2}+1$. 
Note that $\binom{\omega-1}{2} \leq m_1\leq \binom{\omega-1}{2}+\delta$.
If $m_1=\binom{\omega-1}{2}$, then $G[\Omega_1]\cong K_{\omega-1}$.
By~\eqref{eq:nodal_basic_comparison}, we have $ \rho_2<\rho_1(G[\Omega_1])=\omega-2\leq \frac{m-2}{\omega}.$
If $m_1=\binom{\omega-1}{2}+1$ and $G[\Omega_1] \not\cong K_{\omega}-e$ , then by Lemma~\ref{lem:near_complete_clique_estimate_combined_combined}, we have $\rho_2<\rho_1(G[\Omega_1])<\omega-2+\frac{1}{\omega} \leq \frac{m-2}{\omega}.$
For $G[\Omega_1]\cong K_\omega-e$,
since $m_1=\binom{\omega-1}{2}+1$ and $e(K_\omega-e)=\binom{\omega}{2}-1$, we obtain that $\omega=3$. 
Hence $G[\Omega_1]\cong K_3-e\cong P_3$ and $m_1=2$.
Since $m_1+r=\binom{3}{2}+1=4$, we have $r=2$.
If $m_2\geq3$, then by~\eqref{eq:nodal_basic_comparison}, $\rho_2<\rho_1(G[\Omega_1])=\rho_1(P_3)=\sqrt2<\frac53 \leq \frac{m_2+2}{3} =\frac{m-2}{3}$.
So it remains to consider $m_2=2$. Since $G[\Omega_2]$ is connected and has two edges,
we have $G[\Omega_2]\cong P_3$.
Thus $G$ is obtained from two copies of $P_3$ by adding two edges between them. 
Since $\omega(G)=3$, these two added edges must form a triangle together with one edge of one of the two paths. 
Up to isomorphism, there are only two
types.
Let $P_3=x_1x_2x_3$ and $P_3'=y_1y_2y_3$. 
The first type is $H_1=P_3\cup P_3'+\{x_1y_1,x_1y_2\}$, and the second type is $H_2=P_3\cup P_3'+\{x_2y_1,x_2y_2\}$.
Their characteristic polynomials are given respectively by
\begin{align*}
    P(H_1;x)=
    x^6-6x^4-2x^3+7x^2+2x-1, \quad P(H_2;x)=
    x^2\left(x^4-6x^2-2x+5\right).
\end{align*}
It follows that $\rho_2(H_i)<\frac43$.
Since  $m=m_1+m_2+r=2+2+2=6$ and $\omega=3$, we have $\rho_2(G)<\frac{m-2}{\omega}=\frac43$.

Therefore, in this subcase, we have $\rho_2<\frac{m-2}{\omega}.$

\mysubcase{subcase:Z-empty-kappa-omega-2}{$\kappa \leq \omega-2$.}

Since $\kappa \geq 1$, $\omega\geq 3$.
If $\omega=3$, then $\kappa = 1$, and hence each connected graph $G[\Omega_i]$ is a single vertex. By~\eqref{eq:nodal_basic_comparison}, this gives
$\rho_2\leq0$, contrary to our assumption $\rho_2>0$.
We may assume that $\omega\geq4$.

Choose a maximum clique $K_\omega$ of $G$.
Write $a=|K_\omega\cap\Omega_1|$ and $b=|K_\omega\cap\Omega_2|$.
WLOG, we may assume that $a\geq b$.
Since $Z=\emptyset$, we have $a+b=\omega$. 
The $ab$ edges of $K_\omega$ joining $K_\omega\cap\Omega_1$ and
$K_\omega\cap\Omega_2$ are not internal to any nodal domain. 
Then $r\geq r_K= ab$.
Since $\kappa\leq\omega-2$, we also have
$a\leq\omega-2$.
So $b\geq2$.
By~\eqref{eq:nodal_basic_comparison}, Theorem~\ref{thm:Nikiforov_fixed_size}, and Lemma \ref{lem:new_packing}, we have
\begin{align*}
    \rho_2^2
    <  \min\{\rho_1^2(G[\Omega_1]), \rho_1^2(G[\Omega_2])\}
    \leq
    \frac{(m-2)^2}{\omega^2}.
\end{align*}


\noindent\textbf{Characterization of extremal graphs}

For $n$ is even, we know that $2K_\omega+e$ attains the maximum value of $\rho_2(G)\omega-e(G)$  from the preceding discussion; that is, this maximum is attained
if and only if $G\cong 2K_{\frac n2}+e$ with even $n$.

For $n$ is odd,  we have $\rho_2(G)\leq \frac{m-2}{\omega}$. 
The equality may hold only  for the subcases $\kappa=\omega$ and $\kappa=\omega-1$ in Case~\ref{case:Z-nonempty}.

In Subcase~\ref{subcase:Z-nonempty-kappa-omega} $\kappa=\omega$, suppose that equality holds in~\eqref{eq_Z_a_1}. 
Then $m=\omega^2-\omega+2$, $m_2=m-\binom{\omega}{2}-2=\binom{\omega}{2}$, and $m_1+r=\binom{\omega}{2}+2$.
Since $\kappa=\omega$, the graph $G[\Omega_1]$ contains a clique $K_\omega$, which implies that $m_1\geq \binom{\omega}{2}$. 
Note that $r \geq 2$.
Thus $m_1=\binom{\omega}{2}$ and $r=2$. 
Then $G[\Omega_1]\cong K_\omega$.
Moreover, since $Z\neq\emptyset$ and $G$ is connected, the equality $r=2$ implies that $Z$ consists of exactly one vertex, which is adjacent to both nodal domains.
By Theorem~\ref{thm:Hong} and the equality in~\eqref{eq_1_1_1} and~\eqref{eq_Z_a_1}, we have
\begin{align*}
\mu_1(G_{U_i})^2=\rho_1(G[\Omega_2])^2=\rho_1(G[\Omega_1])^2=(\omega-1)^2
    \leq
    2\binom{\omega}{2}-|\Omega_2|+1=
    \omega(\omega-1)-|\Omega_2|+1.
\end{align*}
Then $|\Omega_2|\leq \omega$.
On the other hand, $\binom{\omega}{2}=m_2\leq \binom{|\Omega_2|}{2}$, which implies $|\Omega_2|\geq \omega$. 
Therefore, we have $|\Omega_2|=\omega$ and $G[\Omega_2]\cong K_\omega$.
Then $G$ is composed of two copies of $K_\omega$ linked by a path of length two.
Conversely, suppose that $G$ is obtained from two disjoint copies $K_\omega^{(1)}$ and $K_\omega^{(2)}$ by adding a new vertex z and two
edges $zx_1$,$zx_2$, where $x_i\in V\bigl(K_\omega^{(i)}\bigr)$ for $i=1,2$.
Define a vector $\mathbf{x}\in\mathbb{R}^{V(G)}$ by
\begin{align*}
        x_u=
    \begin{cases}
        1, & u\in V\bigl(K_\omega^{(1)}\bigr),\\
       -1, & u\in V\bigl(K_\omega^{(2)}\bigr),\\
        0, & u=z.
    \end{cases}
\end{align*}
For any $v\in V(K_\omega^{(i)})$, we have $(A(G)\mathbf{x})_v=(\omega-1)x_v$.
For the vertex $z$, we get $(A(G)\mathbf{x})_z=1+(-1)=0=(\omega-1)x_z$.
Thus, $\omega-1$ is an adjacency eigenvalue of $G$.
Since $G$ is connected and $K_\omega$ is a proper subgraph of $G$, by Lemma~\ref{lem:proper_subgraph_spectral_radius}, we have $\rho_1(G)>\rho_1(K_\omega)=\omega-1$.
It follows that $\rho_2(G)\geq \omega-1$.
Together with the bound $\rho_2(G)\leq \frac{m-2}{\omega}=\omega-1$ already proved, we obtain $\rho_2(G)=\omega-1=\frac{m-2}{\omega}$.

In Subcase~\ref{subcase:Z-nonempty-kappa-omega-1}  $\kappa=\omega-1$, suppose that equality holds in~\eqref{eq_Z_a_2}. 
Then $m=\omega^2-2\omega+2$, $m_2=m-\binom{\omega-1}{2}-\omega=\binom{\omega-1}{2}$, $m_1=\binom{\omega-1}{2}$, and $r=\omega$.
The selected clique $K_\omega$ consists of $z$ and $\omega-1$ vertices of $G[\Omega_1]$.
Hence $z$ is adjacent to all vertices of
$G[\Omega_1]$, contributing $\omega-1$ edges. 
Since $r=\omega$ and $G$ is connected, there is exactly one more edge, which joins $z$
to $G[\Omega_2]$. 
In particular, $Z=\{z\}$.
Since $\kappa=\omega-1$, we have $G[\Omega_1]\cong K_{\omega-1}$.
By Theorem~\ref{thm:Hong} and the equality in
\eqref{eq_Z_a_2} and  \eqref{eq_1_1_1}, we have
\begin{align*}
    \rho_1(G[\Omega_2])^2=\rho_1(G[\Omega_1])^2=(\omega-2)^2
    \leq
    2\binom{\omega-1}{2}-|\Omega_2|+1=
    (\omega-1)(\omega-2)-|\Omega_2|+1,
\end{align*}
 which implies $|\Omega_2|\leq \omega-1$.
On the other hand, $\binom{\omega-1}{2}=m_2\leq \binom{|\Omega_2|}{2}$, which implies $|\Omega_2|\geq \omega-1$. 
Hence $|\Omega_2|=\omega-1$ and $G[\Omega_2]\cong K_{\omega-1}$.
Therefore, $G$ is obtained from $K_\omega$ and $K_{\omega-1}$ by adding exactly one edge between them.
The converse is proved in the same way as in Subcase~\ref{subcase:Z-nonempty-kappa-omega} .

In conclusion, for odd $n$, the equality holds if and only if $G$ is composed of $K_{\frac{n+1}{2}}$ and $K_{\frac{n-1}{2}}$ linked by an edge, or $K_{\frac{n-1}{2}}$ and $K_{\frac{n-1}{2}}$ linked by a path.

The proof is complete.
\end{proof}

\section{Concluding remarks}
Let $\mathcal G=\{\mathcal G_n\}_{n\ge 1}$ be a graph class satisfying the pendant-extension property, which holds for many classes of graphs.  
When $n$ is odd,
by Theorem~\ref{lem:adjacency_nodal_domain} and
Lemma~\ref{lem:alpha_beta_weak_pendant_monotonicity}, it shows that the $\rho_2$-extremal graphs in $\mathcal G_n$ are determined by the corresponding $\rho_1$-extremal graphs if the graphs obtained by gluing  two $\rho_1$-extremal graphs are still in $\mathcal G_n$.
When $n$ is even, Theorem~\ref{thm:connected_even_bridge} implies that we should find the extremal graph of $\beta(n/2)$.
A general idea is to use Tait-Tobin's method~\cite{tait2017three} or the similar proof in  Byrne, Desai and Tait's result~\cite{byrne2025general} for $n$ sufficiently large.

Consequently, many known spectral extremal results can be transferred directly to the $\rho_2$-problem.
For the second smallest Dirichlet eigenvalue and algebraic connectivity, the similar general theorems can also be obtained via geometric representation and nodal domain.

 \section*{Declaration of competing interest}
 The authors declare that they have no known competing financial interests or personal relationships that could have appeared to influence the work reported in this paper.

\section*{Data availability}

No data was used for the research described in the article.

\section*{Acknowledgements}
We are grateful to Professor Chengjie Yu for his inspiring insights on the nodal domain and geometric representation.
All authors Huiqiu LIN, Lianping LIU, Xilong YIN, and Zhe YOU are indexed in alphabetical order. All authors are co-first authors.
 Huiqiu LIN was supported by the National Natural Science Foundation of China (No. 12271162, No. 12326372), and Natural Science Foundation of Shanghai (No. 22ZR1416300 and No. 23JC1401500) and The Program for Professor of Special Appointment (Eastern Scholar) at Shanghai Institutions of Higher Learning (No. TP2022031).

\end{document}